\documentclass{amsart}
\usepackage[latin1]{inputenc} 
\pdfoutput=1

\usepackage{graphicx}
\usepackage{float}
\usepackage{amssymb}

\usepackage{xcolor}
\usepackage{colonequals}
\usepackage{mathdots}

\usepackage{tikz}
\usetikzlibrary{shapes.geometric,decorations.markings}

\makeatletter
\@namedef{subjclassname@2020}{\textup{2020} Mathematics Subject Classification}
\makeatother

\usepackage{hyperref}
\usepackage{xurl}
\usepackage{cleveref}

\newtheorem{theorem}{Theorem}[section]
\newtheorem{definition}[theorem]{Definition}

\newtheorem{lemma}[theorem]{Lemma}
\newtheorem{corollary}[theorem]{Corollary}
\newtheorem{remark}[theorem]{Remark}
\newtheorem{proposition}[theorem]{Proposition}

\newtheorem*{claim*}{Claim}
\newtheorem{question}[theorem]{Question}

\numberwithin{equation}{section}

\newcommand{\abs}[1]{\lvert#1\rvert}

\newcommand{\CC}{\mathbb{C}}
\newcommand{\NN}{\mathbb{N}}

\newcommand{\QQ}{\mathbb{Q}}
\newcommand{\RR}{\mathbb{R}}
\newcommand{\ZZ}{\mathbb{Z}}

\DeclareMathOperator{\sizes}{\mathsf{s}}
\DeclareMathOperator{\sizer}{\mathsf{r}}

\DeclareMathOperator{\diam}{diam}

\begin{document}
\title{$\inf(M\setminus L)=3$}


\author[H. Erazo]{Harold Erazo}
\address{Harold Erazo: IMPA, Estrada Dona Castorina, 110. Rio de Janeiro, Rio de Janeiro-Brazil.}
\email{harold.erazo@impa.br}

\author[D. Lima]{Davi Lima}
\address{Davi Lima: Instituto de Matem\'atica, UFAL, Av. Lourival Melo Mota s/n, Macei\'o, Alagoas, Brazil}

\email{davi.santos@im.ufal.br}


\author[C. Matheus]{Carlos Matheus}
\address{Carlos Matheus: CMLS, \'Ecole Polytechnique, CNRS (UMR 7640),
91128, Palaiseau, France}
\email{matheus.cmss@gmail.com}

\author[C. G. Moreira]{Carlos Gustavo Moreira}
\address{Carlos Gustavo Moreira: SUSTech International Center for Mathematics, Shenzhen, Guangdong, P. R. China and IMPA, Estrada Dona Castorina, 110. Rio de Janeiro, Rio de Janeiro-Brazil.}
\email{gugu@impa.br}
\thanks{}

\author[S. Vieira]{Sandoel Vieira}
\address{Sandoel Vieira: UFPI, Rua Dirce Oliveira, Teresina, Piau\'i, Brazil.}
\email{sandoel.vieira@ufpi.edu.br}

\thanks{The first author was partially supported by CAPES and FAPERJ. The second author was partially supported by CNPq: 409198/2021-8 and FAPEAL: E60030.0000002330/2022. The fourth and fifth authors were partially supported by CNPq and FAPERJ. The first and second authors would like to express their gratitude to SUSTech for its generous support during the preparation of this article. Additionally, the third and fourth authors would like to extend their thanks to the Coll\`{e}ge de France for its hospitality}


\date{\today}


\begin{abstract}
The Lagrange and Markov spectra $L$ and $M$ describe the best constants of Diophantine approximations for irrational numbers and binary quadratic forms. In 1880, A. Markov showed that the initial portions of these spectra coincide: indeed, $L\cap (0,3) = M\cap (0,3)$ is a discrete set of explicit quadratic irrationals accumulating only at $3$. 

In this article, we show that the statement above ceases to be true immediately after $3$: in particular, $L\cap (3,3+\varepsilon)\neq M\cap (3,3+\varepsilon)$ for all $\varepsilon>0$, and thus $\inf(M\setminus L)=3$. In fact, we derive this result as a by-product of lower bounds on the Hausdorff dimension of $(M\setminus L)\cap (3,3+\varepsilon)$ implying that $\liminf\limits_{\varepsilon\to 0} \frac{\dim_H((M\setminus L)\cap(3,3+\varepsilon))}{\dim_H(M\cap (3,3+\varepsilon))}\geq \frac{1}{2}$ and, as it turns out, these bounds are obtained from the study of projections of Cartesian products of almost affine dynamical Cantor sets via an argument of probabilistic flavor based on Baker--W{\"u}stholz theorem on linear forms in logarithms of algebraic numbers.  
\end{abstract}

\maketitle

\section{Introduction}

The Lagrange and Markov spectra ${L}$ and ${M}$ are classically defined as 
$${L}:=\left\{k(\alpha):=\limsup\limits_{\substack{p,q\to\infty\\p,q\in\mathbb{Z}}} \frac{1}{|q(q\alpha-p)|}<\infty: \alpha\notin\mathbb{Q}\right\}$$
and 
$${M}:=\left\{\frac{1}{\inf\limits_{\substack{(x,y)\in\mathbb{Z}^2\\(x,y)\neq(0,0)}}|q(x,y)|}<\infty: q\in \mathcal{Q} \right\}$$ 
where $\mathcal{Q}=\{q(x,y)=ax^2+bxy+cy^2; a, b, c\in {\mathbb R}, \ \mbox{$q$ indefinite and} \ b^2-4ac=1\}$. In other words, the Lagrange and Markov spectra are subsets of the real line consisting of the best constants of Diophantine approximations of certain irrational numbers and indefinite binary quadratic forms.

Both spectra have been intensively studied since the seminal works of Markov \cite{M:formes_quadratiques1} and \cite{M:formes_quadratiques2} from 1879 and 1880 showing that 
\[
    L \cap [0, 3) = M \cap [0, 3) = \left\{ \sqrt{5} < \sqrt{8} < \frac{\sqrt{221}}{5} < \dotsb \right\},
\]
that is, $L$ and $M$ coincide below $3$, and they consist of a sequence of \emph{explicit}\footnote{The description of these numbers involve the notions of \emph{Markov numbers} and \emph{Markov tree}. For more discussions of these objects, the reader can consult the articles of Zagier \cite{Zagier}, Bombieri \cite{B:markoff_numbers}, Bourgain--Gamburd--Sarnak \cite{BourgainGamburdSarnak}, Chen \cite{Chen} and de Courcy-Ireland \cite{MdCI}, for instance.} quadratic surds accumulating only at $3$. 

In this paper, we show that the statement in the previous paragraph can not be extended beyond the number $3$:

\begin{theorem}\label{t.infML3}
For any $\varepsilon>0$ we have
\begin{equation}
    M\cap(3,3+\varepsilon) \neq L\cap(3,3+\varepsilon).
\end{equation}
Consequently, $\inf(M\setminus L)=3$.
\end{theorem}

This result was \emph{conjectured} by the last four authors in \cite{MminusLnear3}: in fact, they exhibited a decreasing sequence $m_k\in M$, $m_k\to 3$ as $k\to\infty$ with $m_1, m_2, m_3, m_4\in M\setminus L$, and conjectured that $m_k\in M\setminus L$ for all $k\in\mathbb{N}$, but they could \emph{not} conclude that $m_k\in M\setminus L$ for $k\geq 5$ because of important difficulties of combinatorial nature related to the so-called \emph{local uniqueness mechanism} for the construction of elements of $M\setminus L$. These difficulties are related to the rich combinatorial structure of the Markov and Lagrange spectra near $3$, as described in \cite{B:markoff_numbers} and \cite{EGRS2024}, of which we will give a rough idea in \Cref{sec:outline}. In fact, instead of working with a specific sequence like in \cite{MminusLnear3}, as a consequence of our main result, we will demonstrate that there \emph{exist} many sequences in $M\setminus L$ that converge to $3$.

\subsection{Perron's description of the Markov and Lagrange spectra} The derivation of Theorem \ref{t.infML3} follows the tradition of many results about $L$ and $M$, namely, it relies on Perron's dynamical characterization of these sets. More concretely, given a bi-infinite sequence $\theta=(a_n)_{n\in\mathbb{Z}}\in(\mathbb{N}^*)^{\mathbb{Z}}$, let 
$$\lambda_i(\theta):=a_i+\frac{1}{a_{i+1}+\frac{1}{a_{i+2}+\frac{1}{\ddots}}}+\frac{1}{a_{i-1}+\frac{1}{a_{i-2}+\frac{1}{\ddots}}}.$$
The \emph{Markov value} $m(\theta)$ of $\theta$ is $m(\theta):=\sup\limits_{i\in\mathbb{Z}} \lambda_i(\theta)$. In this context, Perron showed in 1921 that 
\begin{itemize}
\item $M=\{m(\theta)<\infty: \theta\in(\mathbb{N}^*)^{\mathbb{Z}}\}$ and $L=\{\limsup\limits_{n\to\infty}\lambda_n(\theta)<\infty: \theta\in(\mathbb{N}^*)^{\mathbb{Z}}\}$;  
\item $\sqrt{12}, \sqrt{13}\in M$, and $m(\theta)\leq \sqrt{12}$ if and only if $\theta\in\{1,2\}^{\mathbb{Z}}$. 
\end{itemize} 
Furthermore, these characterizations of $L$ and $M$ imply that $L\subset M$ are closed subsets of the real line and, in fact, the Lagrange spectrum $L$ is the closure of the set of Markov values of periodic words in $(\mathbb{N}^*)^{\mathbb{Z}}$. Also, Hall showed in 1947 that the spectra contain the half-line $[6,\infty)$ and Freiman computed in 1975 the smallest constant $c_F=4.5278\dots$ such that $L\supset [c_F,\infty)$: in the literature, $[c_F,\infty)$ is called \emph{Hall's ray} and $c_F$ is called \emph{Freiman's constant}. Nevertheless, these spectra are \emph{not} equal: between 1968 and 1977, the works of Freiman \cite{F:non-coincidence}, \cite{Fr73} and Flahive \cite{Fl77}  showed that $M\setminus L$ contains two explicit countable subsets near $3.11$ and $3.29$, and this was essentially all known information\footnote{In the PhD thesis by \cite{Michiganthesis} in 1976, it was proved that there is a Gauss-Cantor set near $3.11$ contained in $M\setminus L$. Consequently, it was proved that $M\setminus L$ has positive Hausdorff dimension. It is unfortunate that this work went essentially unnoticed; it is not even mentioned in the book of \cite{Cusick-Flahive}.} about $M\setminus L$ until 2017. More recently, it was shown in \cite{Fractalgeometryofcomplement} and \cite{MMPV} that the Hausdorff dimension of $M\setminus L$ satisfies $0.537 < \dim_H(M \setminus L) < 0.797$ and, \emph{a fortiori}, these spectra have the same interior: $\mathrm{int}(L)=\mathrm{int}(M)$. The reader can consult the book \cite{Cusick-Flahive} by Cusick and Flahive for a nice review of the literature on this topic until the mid-eighties, and the book \cite{LMMR} for a discussion of recent developments.

\subsection{Fractal geometry of $M\setminus L$ close to $3$} The intimate relationship between $L$, $M$ and dynamics discovered by Perron was explored by the fourth author in \cite{M:geometric_properties_Markov_Lagrange} to prove that the Hausdorff dimensions of the Lagrange and Markov spectra match when truncated: 
\begin{equation*}
    d(t):=\dim_H(L\cap(-\infty,t))=\dim_H(M\cap(-\infty,t))
\end{equation*}
is a continuous function of $t\in\mathbb R$ such that $d(t)>0$ for all $t >3$. Subsequently, in \cite{EGRS2024} the first and fourth authors together with Guti\'errez-Romo and Roma\~na obtained a quite precise asymptotic estimate of the Hausdorff dimension of $L\cap (3,3+\varepsilon)$: they proved that, for all $\varepsilon>0$ small, one has 
\begin{equation}\label{eq:dimension_near_3}
    d(3+\varepsilon)=2\cdot\frac{W(e^{c_0}|\log \varepsilon|)}{|\log \varepsilon|}+O\left(\frac{\log |\log \varepsilon|}{|\log \varepsilon|^2}\right)
\end{equation}
where $W$ is Lambert's function and $c_0:=-\log\log((3+\sqrt{5})/2) \approx 0.0383$. 

As it turns out, Theorem \ref{t.infML3} is a direct consequence of the main result of this paper, namely, the Hausdorff dimension of $(M\setminus L)\cap (3,3+\varepsilon)$ is ``morally'' bounded from below by $\frac{1}{2}d(3+\varepsilon)$ as $\varepsilon$ goes to zero:
\begin{theorem}\label{thm:dimension}
For any $\varepsilon>0$ small, we have
\begin{equation}
    \dim_H((M\setminus L)\cap(3,3+\varepsilon))\geq \frac{W(e^{c_0}|\log \varepsilon|)}{|\log \varepsilon|}-O\left(\frac{\log |\log \varepsilon|}{|\log \varepsilon|^2}\right),
\end{equation}
where $c_0=-\log\log((3+\sqrt{5})/2) \approx 0.0383$. In particular, given that the Lambert function has an expansion $W(x)=\log x - \log\log x + o(1)$ for $x\to\infty$, one has 
\[\dim_H((M\setminus L)\cap(3,3+\varepsilon))\geq \frac{\log|\log\varepsilon|-\log\log|\log\varepsilon|+c_0}{|\log \varepsilon|} - o\left(\frac{1}{|\log\varepsilon|}\right)\] for all $\varepsilon>0$ small.
\end{theorem}

\begin{remark} A forthcoming work by the fourth author and C. Villamil will show that there is a constant $C>0$ such that 
\begin{equation*}
    \dim_H((M\setminus L)\cap(3,3+\varepsilon))\leq \frac{\log|\log\varepsilon|-\log\log|\log\varepsilon|+C}{|\log \varepsilon|}.
\end{equation*} 
for any $\varepsilon>0$ small. By combining this estimate with \eqref{eq:dimension_near_3} and Theorem \ref{thm:dimension} above, one concludes that 
\begin{equation*}
    \dim_H((M\setminus L)\cap(3,3+\varepsilon))=\left(\frac{1}{2}+O\left(\frac{1}{\log|\log\varepsilon|}\right)\right)\dim_H(M\cap(3,3+\varepsilon)).
\end{equation*}
\end{remark}

\begin{question} 
Is it true that
\begin{equation*}
    \dim_H((M\setminus L)\cap(3,3+\varepsilon))=\frac{1}{2}\cdot\dim_H(M\cap(3,3+\varepsilon))
\end{equation*} 
holds for all $\varepsilon>0$ sufficiently small?
\end{question}

\subsection{Outline of the proof of Theorem \ref{thm:dimension}}\label{sec:outline} In this article, we study exclusively the portion of $M$ below $\sqrt{12}$ and, for this reason, we assume that all sequences appearing in the sequel consist of $1$ and $2$ (i.e., all sequences in this paper belong to $\{1,2\}^{\mathbb{Z}}$ by default). Furthermore, we use superscripts to indicate the repetition of a certain character: for example, $1^2 2^4$ is the string $112222$. Also, $\overline{a_1,\dots, a_l}$ is the periodic word obtained by infinite concatenation of the string $(a_1,\dots, a_l)$. Moreover, unless explicitly stated otherwise, we indicate the zeroth position $a_0$ of a string $(a_{-m}, \dots, a_{-1}, a_0^*, a_1,\dots, a_n) = (a_{-m}, \dots, a_{-1}|a_0, a_1,\dots, a_n)$ by an asterisk or by a vertical bar. Given a finite word $w=w_1\dots w_\ell \in\{1,\dots,A\}^{\ell}$, we denote its \emph{transpose} by $w^T:=w_\ell\dots w_1$ and use $|w|:=\ell$ to denote its length. 

Our ultimate goal is to construct many elements of $M\setminus L$ nearby $3$. For this sake, we employ the sole currently known mechanism to build elements of $M\setminus L$, namely, we seek for the \emph{local uniqueness} and \emph{self-replication} properties in the vicinities of \emph{non semisymmetric} words. Let us now briefly explain this technique while refereeing the reader to \cite{Fractalgeometryofcomplement} for further discussions (including a dynamical interpretation of the method). 

Recall that we say that a finite word $w$ is \emph{semisymmetric} if it is a palindrome or the concatenation of two palindromes. Equivalently, it is semisymmetric if the orbits by the shift map of the bi-infinite sequences $\overline{w}$ and $\overline{w^T}$ are equal. For example, 1111222 and 1211121 are semisymmetric while 2112122 and 122122221 are not. It is an easy exercise to show that most finite words are not semisymmetric. In this context, one hopes to find elements of $M\setminus L$ along the following lines. First, take a non semisymmetric word of odd length $w=w_1\dots w_{2n+1}$ and denote $w^*=w_1\dots w_{n+1}^*\dots w_{2n+1}$. Suppose that this word is such that the Markov value is attained exactly at the middle position $m(\overline{w})=\lambda_0(\overline{w_1\dots w_n w_{n+1}^{*}w_{n+2}\dots w_{2n+1}})$. The method requires checking two properties: 

\medbreak
\textit{Local uniqueness:}
\medbreak
There is an $\varepsilon_1>0$ such that for any bi-infinite sequence $\underline{\theta}\in\{1,\dots,A\}^{\ZZ}$ with $m(\underline{\theta})=\lambda_0(\underline{\theta})$, it holds that if $|m(\underline{\theta})-m(\overline{w})|<\varepsilon_1$, then $\underline{\theta}$ must be of the following form
\begin{align*}
    \underline{\theta}&=\dots w_{n+2}\dots w_{2n+1} w_1 \dots w_{n}w_{n+1}^*w_{n+2}\dots w_{2n+1} w_1 \dots w_{n}\dots \\
    &= \dots  w_{n+2}\dots w_{2n+1} w^* w_1 \dots w_{n}\dots 
\end{align*}

\medbreak
\textit{Self-replication (to the left):}
\medbreak
There is an $\varepsilon_2>0$ such that for any bi-infinite sequence $\underline{\theta}\in\{1,\dots,A\}^{\ZZ}$ with $m(\underline{\theta})=\lambda_0(\underline{\theta})$, it holds that if $m(\underline{\theta})<m(\overline{w})+\varepsilon_2$ and $\underline{\theta}=\dots  w_{n+2}\dots w_{2n+1} w^* w_1 \dots w_{n}\dots$, then $\underline{\theta}$ must be of the following form
\begin{align*}
    \underline{\theta}&=\dots w_{n+2}\dots w_{2n+1} w_1 \dots w_{n}w_{n+1}^*w_{n+2}\dots w_{2n+1} w_1 \dots w_{n}\dots \\
    &= \dots  w_{n+2}\dots w_{2n+1} ww^* w_1 \dots w_{m}\dots 
\end{align*}
for some $m<|w|+\lfloor |w|/2\rfloor=3n+1$ (with the convention $w_{k}:=w_{k-2n-1}$ when $k>2n+1$). In particular, for all such $\underline{\theta}$ it holds
\begin{equation*}
    \underline{\theta}=\overline{w}w^*w_1\dots w_m\dots
\end{equation*}

Self-replication to the right is defined analogously and coincides with self-replication to the left for the transpose word $w^T:=w_{2n+1}\dots w_1$.

Before proceeding further with our sketch of proof of Theorem \ref{thm:dimension}, let us briefly explain the main obstacles  faced by the last four authors to show local uniqueness for the sequence $m_k$ of elements of $M$ converging to $3$ they introduced in \cite{MminusLnear3}.

Let $a=22$ and $b=11$. We define recursively {\it admissible alphabets} as follows: $(a, b)$ is an admissible alphabet and $(\alpha,\beta)\neq (a, b)$ is an admissible alphabet if, and only if, there is an admissible alphabet $(u, v)$ such that $(\alpha, \beta)=(u, uv)$ or $(\alpha, \beta)=(uv, v)$. It is possible to prove (see \cite{B:markoff_numbers}) that $m(\theta)<3$ for a sequence $\theta=(a_n)_{n\in\mathbb{Z}}\in(\mathbb{N}^*)^{\mathbb{Z}}$ if and only if, $\theta$ is periodic with period $1$, $2$ or $\alpha\beta$, for some admissible alphabet $(\alpha, \beta)$. It turns out that sequences $\theta$ obtained by concatenations of words in a large admissible alphabet (in the sense that the sum of the sizes of the words in the alphabet is large) have $m(\theta)$ very close to $3$ (but not conversely: there are sequences $\theta$ with $m(\theta)$ larger but very close to $3$ having $aa$ and $bb$ as factors, very far away from each other, and thus that cannot be obtained by concatenations of words in any admissible alphabet); in particular, if $(\alpha, \beta)$ is a large admissible alphabet, then the periodic word $\theta$ with period $\alpha\alpha\beta\beta$ has $m(\theta)$ larger than $3$ but close to $3$. The examples in \cite{MminusLnear3} are related to alphabets of the form $(a, a^m b)$, with $m$ large, but there are many other quite different large alphabets that could produce infinite words with Markov value very close to the $m_k$.

This kind of issue will be solved in the present paper by an application of Erd\"os probabilistic method: it follows from the techniques of \cite{EGRS2024} that most large finite words that are subwords (factors) of some sequence $\theta=(a_n)_{n\in\mathbb{Z}}\in(\mathbb{N}^*)^{\mathbb{Z}}$ with $m(\theta)$ very close to $3$ can be written in an alphabet of the type $(221^n, 1)$, for some large $n$. Most (odd length, non semisymmetric) words constructed from such an alphabet $(221^n, 1)$ will not produce (for the corresponding periodic infinite words) Markov values too similar to those corresponding to sequences written in an alphabet of different nature (see Section \ref{ss.minority-report} for more details). So, our candidates for periods will not be explicit, but will be typical among the candidates constructed from (odd, non semisymmetric) written in such an alphabet $(221^n, 1)$.

\medbreak
\textit{Method for building elements on the difference set $M\setminus L$:}
\medbreak

If both local uniqueness and self-replication hold, then for $\varepsilon:=\min\{\varepsilon_1,\varepsilon_2\}$ one has that $L\cap(m(\overline{w})-\varepsilon,m(\overline{w})+\varepsilon)=\{m(\overline{w})\}$, in other words the value $m(\overline{w})$ is an isolated point in the Lagrange spectrum: indeed, since the Lagrange spectrum coincides with the closure of the Markov values of periodic points, if $\overline{p}$ is a periodic point with $|m(\overline{p})-m(\overline{w})|<\varepsilon$, then in particular one has that $\overline{p}=\overline{w}w^*w_1\dots w_m\dots$ which can only occur if $\overline{p}=\overline{w}$.

On the other hand, since we are not forced \textit{a priori} to connect to the right (or left) with $w$ as well, one might expect to glue orbits that have smaller value to the right. In fact, if we are able to glue an entire subshift that has smaller Markov values, one will obtain a whole Cantor set near $m(\overline{w})$, which necessarily belongs to $M\setminus L$. More precisely, we expect to find a finite word $\tau\in\{1,\dots,A\}^{|\tau|}$ and a finite set of finite words $B=\{\beta_1,\dots,\beta_\ell\}$ such that
\begin{equation*}
    C=\{\lambda_0(\overline{w}w^*w_1\dots w_m\tau\gamma_1\gamma_2\dots):\gamma_i\in B, \forall i\geq 1\}
\end{equation*}
satisfies $C\subset M\cap (m(\overline{w}),m(\overline{w})+\varepsilon)$. In particular, $C$ will be a subset of the difference set $M\setminus L$ which is bi-Lipschitz to a (Gauss--)Cantor set. As consequence one has that $\dim_H(C)=\dim_H(K(B))$ (see \Cref{def:Gauss_Cantor} for the definition of $K(B)$) and one can use the elementary techniques of \cite[Chapter 4]{PalisTakens} to estimate $\dim_H(K(B))$ and finally, obtain a lower bound for the difference set $\dim_H\left((M\setminus L)\cap (m(\overline{w}),m(\overline{w})+\varepsilon)\right)\geq\dim_H(K(B))$ (since the Hausdorff dimension is monotonous).

Restricting our attention, as we explained above, above, to (odd length, non semisymmetric) words constructed from an alphabet $(221^n, 1)$ (for a fixed large $n$), our problem reduces to prove \emph{almost injectivity} of the projections of some finite steps of the construction of Cartesian products of \emph{almost affine} dynamical Cantor sets associated to these alphabets. As it turns out, the proof of the desired almost injectivity involves an application of Baker--W\"ustholz theorem on linear forms in logarithms of algebraic numbers. See \Cref{sec:Cantor_sets} for more details.

\subsection{Organization of the paper}
The remainder of this paper is devoted to establishing Theorem \ref{thm:dimension} according to the sketch in the previous paragraph. In particular, in \Cref{sec:preliminaries}, we review some basic facts and elementary estimates about continued fractions, regular Cantor sets and linear forms in logarithms. In \Cref{sec:Cantor_sets} we employ all this to study affine projections of certain products of Gauss-Cantor sets. Finally, in \Cref{sec:construction_of_words}, we build a lot of non semisymmetric words and we prove that most of them satisfy local uniqueness and self-replication. 


\section{Preliminaries}\label{sec:preliminaries}

\subsection{Some standard notation}

Given two functions $f,g:A\to\CC$ defined on a subset $A\subset\RR$, we use the notation
\begin{itemize}
\item $f(x)=O(g(x))$ to denote that there is a positive constant $C$ such that $|f(x)|\leq C|g(x)|$ for all $x$ large.
\item $f(x)=o(g(x))$ to denote that 
\begin{equation*}
    \lim_{x\to\infty}\frac{f(x)}{g(x)}=0.
\end{equation*}

\end{itemize}

\subsection{Some classical facts about continued fractions} 

Given $x\in \mathbb{R}$ we denote $\lfloor x\rfloor=\max (\mathbb{Z}\cap (-\infty, x])$. The continued fraction expansion of an irrational number $x$ is denoted by 
$$x=[a_0;a_1,a_2,\dots] = a_0 + \frac{1}{a_1+\frac{1}{a_2+\frac{1}{\ddots}}},$$ 
so that the Gauss map $g:(0,1)\to[0,1)$, $g(x)=\dfrac{1}{x}-\left\lfloor \dfrac{1}{x}\right\rfloor$ acts on continued fraction expansions by $g([0;a_1,a_2,\dots]) = [0;a_2,\dots]$. 

Given $x=[a_0;a_1,\dots, a_n, a_{n+1},\dots]$ and $\tilde{x}=[a_0;a_1,\dots, a_n, b_{n+1},\dots]$ with $a_{n+1}\neq b_{n+1}$, recall that $x>\tilde{x}$ if and only if $(-1)^{n+1}(a_{n+1}-b_{n+1})>0$. 

For an irrational number $x=\alpha_0$, the continued fraction expansion $x=[a_0;a_1,\dots]$ is recursively obtained by setting $a_n=\lfloor\alpha_n\rfloor$ and $\alpha_{n+1} = \frac{1}{\alpha_n-a_n} = \frac{1}{g^n(\alpha_0)}$. The rational approximations  
$$\frac{p_n}{q_n}:=[a_0;a_1,\dots,a_n]\in\mathbb{Q}$$ 
of $\alpha$ satisfy the recurrence relations $p_n=a_n p_{n-1}+p_{n-2}$ and $q_n=a_n q_{n-1}+q_{n-2}$ (with the convention that $p_{-2}=q_{-1}=0$ and $p_{-1}=q_{-2}=1$). Moreover, $p_{n+1}q_n-p_nq_{n+1}=(-1)^n$ and $x=\frac{\alpha_n p_{n-1}+p_{n-2}}{\alpha_n q_{n-1}+q_{n-2}}$. In particular, given $x=[a_0;a_1,\dots, a_n, a_{n+1},\dots]$ and $\tilde{x}=[a_0;a_1,\dots,a_n,b_{n+1},\dots]$, we have
\begin{equation}\label{eq:compare_continued_fractions}
    x-\tilde{x} = (-1)^n\frac{\tilde{\alpha}_{n+1}-\alpha_{n+1}}{q_n^2(\beta_n+\alpha_{n+1})(\beta_n+\tilde{\alpha}_{n+1})} 
\end{equation}
where $\beta_n:=\frac{q_{n-1}}{q_n}=[0;a_n,\dots,a_1]$. 

Moreover if $a_0,a_1,a_2,\dots,b_{n+1},\dots \leq T$ and $a_{n+1}\neq b_{n+1}$, then from \cite[Lemma A.1]{M:geometric_properties_Markov_Lagrange} it follows
\begin{equation}\label{eq:lower_bound_interval_gap}
    \left|[a_0;a_1,\dots,a_n,a_{n+1},\dots]-[a_0;a_1,\dots,a_n,b_{n+1},\dots]\right|>\frac{1}{(T+1)(T+2)q_{n+1}^2}
\end{equation}
where $q_n=q_n(a_1,\dots,a_n)$.

In general, given a finite string $(a_1,\dots, a_l)\in(\mathbb{N}^*)^l$, we write 
$$[0;a_1,\dots,a_l] = \frac{p(a_1\dots a_l)}{q(a_1\dots a_l)}.$$ 
By Euler's rule, $q(a_1\dots a_l) = q(a_1\dots a_m) q(a_{m+1}\dots a_l) + q(a_1\dots a_{m-1}) q(a_{m+2}\dots a_l)$ for $1\leq m<l$, and $q(a_1\dots a_l) = q(a_l\dots a_1)$.

\subsection{Intervals of the Gauss map}

For a finite word $\alpha \in (\mathbb N^*)^n$ written as $\alpha = c_1 c_2 \ldots c_n$, we define by $I(\alpha)$ the interval
\[
	I(\alpha) \colonequals \{x\in[0,1] : x=[0; c_1,c_2,\dots,c_n,t], t\ge 1\} \cup \{[0, c_1,c_2, \dotsc, c_n]\}
\]
consisting of the numbers in $[0,1]$ whose continued fractions start with $\alpha$. The set $I(\alpha)$ is a closed interval in $[0,1]$.

We denote by $\sizes(\alpha)\colonequals |I(\alpha)|$ the \emph{size} of that interval. The endpoints of $I(\alpha)$ are $[0;c_1,c_2,\dots,c_n]=p_n/q_n$ and $[0;c_1,c_2,\dots,c_{n-1},c_n+1]=\frac{p_n+p_{n-1}}{q_n+q_{n-1}}$. Thus,
\[
    \sizes(\alpha)=\left| \frac{p_n}{q_n}-\frac{p_n+p_{n-1}}{q_n+q_{n-1}} \right| = \frac1{q_n(q_n+q_{n-1})},
\]
since $p_nq_{n-1}-p_{n-1}q_n=(-1)^{n-1}$. In particular we always have the inequality
\begin{equation}\label{eq:sizes_and_q_n}
    \frac{1}{2q_n^2}<\sizes(\alpha)<\frac{1}{q_n^2}.
\end{equation}

We define $\sizer(\alpha):=\lfloor\log (\sizes(\alpha)^{-1}) \rfloor$, which controls the order of magnitude of the size of $I(\alpha)$. Observe that $\sizer(\alpha) \leq r$ if and only if $\sizes(\alpha) > e^{-r-1}$. 

We also define, for $r \in \mathbb N$, the set
\[
    \mathcal{Q}_r= \{\alpha=(c_1,c_2,\dots,c_n)\ \mid\ \sizer(\alpha) \ge r, \sizer(c_1,c_2,\dots,c_{n-1})<r\}.
\]
Observe that $\alpha \in \mathcal{Q}_r$ if and only if $\sizes(\alpha) < e^{-r}$ and $\sizes(\alpha') \geq e^{-r}$, where $\alpha'$ is the word obtained by removing the last letter from $\alpha$.

For any finite words $\alpha$, $\beta$, we have \cite[Lemma A.2]{M:geometric_properties_Markov_Lagrange} that
\begin{equation}\label{eq:intervals_length}
    \frac{1}{2} \sizes(\alpha)\sizes(\beta)<\sizes(\alpha\beta)<2\sizes(\alpha)\sizes(\beta);
\end{equation}
it follows that 
\begin{equation}\label{eq:sizer_length}
    \sizer(\alpha)+\sizer(\beta)-1\le \sizer(\alpha\beta)\le \sizer(\alpha)+\sizer(\beta)+2.
\end{equation} 

Let us recall some useful estimates from \cite[Appendix A]{EGRS2024}. If $\alpha$ is a subword of $\beta$ then $\sizes(\beta)\leq \sizes(\alpha)$. We also have that for a finite word $\alpha=a_1\dots a_n$ and its transpose $\alpha^T=a_n\dots a_1$, it holds that
\begin{equation}\label{eq:transpose_interval_sizes}
    \frac{[1;a_n+1]}{[1;a_1]}\leq \frac{\sizes(\alpha^T)}{\sizes(\alpha)}=\frac{1+[0;a_n,\dots,a_1]}{1+[0;a_1,\dots,a_n]}\leq \frac{[1;a_n]}{[1;a_1+1]}.
\end{equation}

From \cite[Lemma A.2]{EGRS2024} we have an explicit comparison between the size of intervals and the length of the words that generate them: if $w$ is a word in $\{1,2\}^{|w|}$ then
\begin{equation}\label{eq:rcomparedwithsize}
    (|w|-3) \log\left( \frac{3 + \sqrt{5}}{2} \right) \leq \sizer(w) \leq (|w|+1) \log(3 + 2 \sqrt{2}).
\end{equation}

A crucial tool to understand the spectra around 3 has been a special identity of continued fractions, namely
\begin{equation}\label{eq:cont_frac_identity}
    [2;2,x]+[0;1,1,x]=3, \qquad \text{for all $x>0$}.
\end{equation}

Using this identity and \eqref{eq:compare_continued_fractions} one can deduce the following \cite[Lemma 3.2]{EGRS2024}.

	\begin{lemma} \label{lem:calc_s}
        Let $\omega$ be a bi-infinite word in $1$ and $2$ not containing $121$ and $212$ and such that $\omega = R^T w^T11|22wS$, where $w$ is a finite word, $R = R_1 R_2 \ldots$, $S = S_1 S_2 \ldots$ and $R_1 \neq S_1$, with $R_i, S_i \in \{1,2\}$ for each $i$. Then if $w$ has even length, $R_1=1$ and $S_1=2$,
        \[
            \sizes(11w11)<\lambda(\omega)-3<\sizes(11w1)
        \]
    \end{lemma}

We will need an elementary lemma that compares the values of certain cuts for subwords of the form $\dots 1^{k_{-1}}221^{k_0}221^{k_{1}}\dots$.

\begin{lemma}\label{lem:cut_comparison}
Suppose $k_0,k_1,k_{-1}$ are positive integers such that $k_0$ is odd and $k_{-1},k_1 > k_0\geq 3$. Let $\omega=\dots 1^{k_{-1}}221^{k_0}221^{k_{1}}\dots$ be a bi-infinite sequence. Then we have 
\begin{itemize}
\item If $k_{-1}$ is even and $k_1$ is odd, or $k_1,k_{-1}$ are both odd and $k_{-1}>k_1$ then
\begin{equation*}
    \lambda(\dots 1^{k_{-1}}22|1^{k_0}221^{k_{1}}\dots)>\lambda(\dots 1^{k_{-1}}221^{k_0}|221^{k_{1}}\dots)>3.
\end{equation*}
\item If $k_1,k_{-1}$ are even and $k_{-1}>k_1$ then
\begin{equation*}
    \lambda(\dots 1^{k_{-1}}221^{k_0}|221^{k_{1}}\dots)>\lambda(\dots 1^{k_{-1}}22|1^{k_0}221^{k_{1}}\dots)>3.
\end{equation*}
\end{itemize}
Moreover in any case
\begin{equation*}
    \lambda(\dots 1^{k_{-1}}221^{k_0}|221^{k_{1}}\dots)>\lambda(\dots 1^{k_{-1}}221^{k_0}22|1^{k_{1}}\dots),
\end{equation*}
and
\begin{equation*}
    \lambda(\dots 1^{k_{-1}}22|1^{k_0}221^{k_{1}}\dots)>\lambda(\dots 1^{k_{-1}}|221^{k_0}221^{k_{1}}\dots).
\end{equation*}

\end{lemma}

\begin{proof}
From \eqref{eq:cont_frac_identity} one has that
\begin{equation*}
    \lambda(\dots 1^{k_{-1}}221^{k_0}|221^{k_{1}}\dots)=3+[0;1^{k_0},2,2,1^{k_{-1}},\dots]-[0;1^{k_1+2},2,2,\dots],
\end{equation*}
and 
\begin{equation*}
    \lambda(\dots 1^{k_{-1}}22|1^{k_0}221^{k_{1}}\dots)=3+[0;1^{k_0},2,2,1^{k_{1}},\dots]-[0;1^{k_{-1}+2},2,2,\dots].
\end{equation*}
Since $k_0$ is odd both expressions are larger than 3. In the case that $k_{-1}$ is even and $k_1$ is odd, one clearly see which expression is larger.

Now assume that $k_{-1}>k_1$. In particular
\begin{equation*}
    |[0;1^{k_0},2,2,1^{k_{-1}},\dots]-[0;1^{k_0},2,2,1^{k_{1}},\dots]|<\sizes(1^{k_0}221^{k_1}).
\end{equation*}
On the other hand, using \eqref{eq:lower_bound_interval_gap} we have
\begin{equation*}
    (-1)^{k_1}([0;1^{k_{-1}+2},2,2,\dots]-[0;1^{k_1+2},2,2,\dots])> \frac{1}{12q(1^{k_1+3})^2}>\frac{1}{12}\sizes(1^{k_1+3}).
\end{equation*}
Finally, since $\sizes(221^{n})<\sizes(1^{n+2})/2$ for any $n\geq 1$ and $\sizes(1^{n})\leq1/15$ for $n\geq 3$,  using \eqref{eq:intervals_length}, one has $\sizes(1^{k_0}221^{k_1})<2\sizes(1^{k_0})\sizes(221^{k_1})<\sizes(1^{k_1+2})/12$. 

The last two claims follow from the fact that $k>k_0$ implies that
\begin{align*}
    \lambda_0(\dots 221^{k}22|1^{k_0}22\dots) &= 3 + [0;1^{k_0},2,2,\dots]-[0;1^{k+2},2,2,\dots] \\
    &> 3 + [0;1^{k},2,2,\dots]-[0;1^{k_0},2,2,\dots] \\
    &=\lambda_0(\dots 221^{k}|221^{k_0}22\dots)
\end{align*}

\end{proof}

\subsection{Gauss--Cantor sets}

A dynamical (or regular) Cantor set $K$ is a compact set defined by a Markov partition $\mathcal{P}=\{I_1,I_2,...,I_r \}$, a topologically mixing $C^{s}$-expanding map $\psi:I_1\cup I_2\cup ...\cup I_r\rightarrow I$, where $I$ is the convex hull of  $I_1\cup I_2\cup ...\cup I_r$ and

\begin{enumerate}
\item[1)] for any $j=1,2,...,r$ the $\psi(I_j)$ is the convex hull of an union $\cup_{t=1}^{R(j)}I_{m_t}$

\item[2)] the set $K$ is given by $$K=\bigcap_{n\ge 0}\psi^{-n}(I_1\cup I_2\cup... \cup I_r).$$
\end{enumerate}

A remarkable fact about $C^{1+\varepsilon}, \varepsilon>0$, regular Cantor sets is the \emph{bounded distortion property}, which in rough terms says that relative distances on a microscopic scale are at most distorted by a constant factor with respect to relative distances corresponding on a larger scale, and increasingly less distorted if we decrease the size of the major scale. We will give a proof of it just to highlight explicitly the dependence of this distortion in terms of the scale, since this will be important for our arguments.

\begin{lemma}[Bounded distortion]\label{lem:bounded_distortion}
Let $(K,\psi)$ be a regular Cantor set of class $C^{1+\varepsilon}$, $\varepsilon>0$, and $\{I_1,\dots,I_r\}$ a Markov partition. Given $\delta>0$, there exists a constant $C(K,\delta)$, decreasing on $\delta$, with the following property: if $x,y\in K$ satisfy
\begin{enumerate}
    \item $|\psi^n(x)-\psi^n(y)|\leq\delta$;
    \item The interval $[\psi^i(x),\psi^i(y)]$ is contained in $I_1\cup\dots\cup I_r$ for  $i=0,1,\dots,n-1$, 
\end{enumerate}
then
\begin{equation*}
    e^{-C(K,\delta)}\leq \frac{|(\psi^n)^\prime(x)|}{|(\psi^n)^\prime(y)|}\leq e^{C(K,\delta)},
\end{equation*}
where
\begin{equation*}
    C(K,\delta)=C_0\delta^\varepsilon\cdot\frac{\sigma^{-\varepsilon}}{1-\sigma^{-\varepsilon}},
\end{equation*}
$\sigma=\inf_{t\in K}|\psi^\prime(t)|$ and $C_0>0$ is such that $\left|\log|\psi^\prime(s)|-\log|\psi^\prime(t)|\right|\leq C_0|s-t|^\varepsilon$ for all $s,t\in K$. In particular $C(K,\delta)\to0$ as $\delta\to0$.
\end{lemma}

\begin{proof}
Since one has $|\psi^j(x)-\psi^j(y)|\leq\sigma^{j-n}\delta$ for all $0\leq j\leq n$, we obtain
\begin{align*}
    \left|\log|(\psi^n)^\prime(x)|-\log|(\psi^n)^\prime(y)|\right|&\leq\sum_{j=0}^{n-1}\left|\log|\psi^\prime(\psi^j(x))|-\log|\psi^\prime(\psi^j(y))|\right| \\
    &\leq C_0\sum_{j=0}^{n-1}(\sigma^{j-n}\cdot\delta)^\varepsilon < C_0\delta^\varepsilon\cdot\sum_{j=1}^\infty\sigma^{-j\varepsilon}=C_0\delta^\varepsilon\cdot\frac{\sigma^{-\varepsilon}}{1-\sigma^{-\varepsilon}}.
\end{align*}
\end{proof}

In this paper we will only work with Gauss-Cantor sets.

\begin{definition}\label{def:Gauss_Cantor} Given $B=\{\beta_1,\dots,\beta_\ell\}$, $\ell\geq 2$, a finite alphabet of finite words $\beta_j\in(\mathbb{N}^*)^{r_j}$, which is primitive (in the sense that $\beta_i$ does not begin by $\beta_j$ for all $i\neq j$) then the Gauss-Cantor set $K(B)\subseteq [0,1]$ associated with $B$ is defined as 
    \[
        K(B)\colonequals \{[0;\gamma_1, \gamma_2, \dots] \ \mid\ \gamma_i\in B\}.
    \]
\end{definition}

The set $K(B)$ is a dynamically defined Cantor set of class $C^2$. We will now exhibit its Markov partition and the expanding map which defines it.

For each word $\beta_j\in(\mathbb{N}^*)^{r_j}$, let $I_j$ be the convex hull of the set $\{[0;\beta_j, \gamma_1, \gamma_2, \dots]\ \mid\ \gamma_i\in B\}$ and $\psi|_{I_j}\colonequals g^{r_j}|_{I_j}$ be an iterate of the Gauss map. This defines an expanding map $\psi\colon I_1\cup\dotsb\cup I_\ell\to I$. Let $I=[\min K(B), \max K(B)]$. Then $I$ is the convex hull of $I_1\cup\dotsb\cup I_\ell$ and $\psi(I_j)=I$ for every $j\le \ell$.

Applying \Cref{lem:bounded_distortion} to a Gauss-Cantor set $K(B)$, we obtain that given $\gamma_1,\dots,\gamma_n\in B$, there is $C=C(B)\in\RR$ such for any $x,y,z,w\in I(\gamma_1\dotsb\gamma_n)$ we have

\begin{equation}\label{eq:bounded_distortion}
    e^{-C\cdot\diam K(B)}\frac{|z-x|}{|y-w|}\leq\frac{|g^{r_1+\dots+r_n}(z)-g^{r_1+\dots+r_n}(x)|}{|g^{r_1+\dots+r_n}(y)-g^{r_1+\dots+r_n}(w)|}\leq e^{C\cdot\diam K(B)}\frac{|z-x|}{|y-w|}
\end{equation}
where $C=C(B)=\frac{2}{\min K(B)}\cdot\frac{(\max K(B))^2}{1-(\max K(B))^2}$.

\subsection{Hausdorff dimension of Gauss-Cantor sets with big blocks of 1's}\label{sec:hausdorff_dimension} 

We want to estimate $\dim_H(K(B))$. These estimates are important for our arguments, since once one has an approximation for the dimension, we can use it to estimate the number of elements of any Markov partition. We shall use this to count the number of words in the construction of $K_{\mathrm{small}}\times K_{\mathrm{small}}$, where $K_{\mathrm{small}}$ is the Gauss--Cantor set defined in the beginning of the next section.

Recall that a Gauss-Cantor set $K(B)$ is defined by an alphabet $B=\{\beta_1,\dots,\beta_\ell\}$ and for each word $\beta_j\in(\mathbb{N}^*)^{r_j}$, the interval $I_j$ is the convex hull of the set $\{[0;\beta_j, \gamma_1, \gamma_2, \dots]\ \mid\ \gamma_i\in B\}$ and $\psi|_{I_j}\colonequals g^{r_j}|_{I_j}$ is an iterate of the Gauss map. This defines an expanding map $\psi\colon I_1\cup\dotsb\cup I_\ell\to I$.

According to Palis--Takens \cite[Chapter 4, Pages 68-71]{PalisTakens}, if we let 
	\[
		\lambda_j=\inf |\psi'|_{I_j}|, \qquad \Lambda_j=\sup |\psi'|_{I_j}|
	\]
	and $d_1,d_2\geq 0$ be such that
	\[
		\sum_{i=1}^{\ell}\lambda_j^{-d_2}=1, \qquad \sum_{i=1}^{\ell}\Lambda_j^{-d_1}=1,
	\]
	then 
	\begin{equation}\label{eq:Palis-Takens}
		d_1\leq \dim_H(K(B))\leq d_2.
	\end{equation}
    Moreover, it also holds that if $d_1$ satisfies \eqref{eq:Palis-Takens}, then $d_1$--power sums of intervals sizes do not decrease when we refine intervals in the construction of $K(B)$. More explicitly, if $I(\beta_{j_1}\dots\beta_{j_r})$ is any interval of the construction of $K(B)$ (i.e. $\beta_{j_1},\dots,\beta_{j_r}\in B$), then one has that 
    \begin{equation}\label{eq:d-measurecover}
        \sum_{i=1}^\ell |I(\beta_{j_1}\dots\beta_{j_r}\beta_i)|^{d_1} \geq |I(\beta_{j_1}\dots\beta_{j_r})|^{d_1}.
    \end{equation}

	\bigskip
	Let us now discuss how to find estimates for $d_1$ and $d_2$.
	
	\noindent The iterates of the Gauss map are given explicitly by
	\[
	    \psi|_{I_j}(x)=\dfrac{q^{(j)}_{r_{j}}x - p^{(j)}_{r_{j}}}{p^{(j)}_{r_j-1} - q^{(j)}_{r_j-1}x}
    \]
	 where $\dfrac{p^{(j)}_k}{q^{(j)}_k} = [0; b^{(j)}_1, \dots, b^{(j)}_k]$ and $\beta_j = (b^{(j)}_1, \dots, b^{(j)}_{r_j})$. 
	
	\noindent Hence 
	\[
		(\psi|_{I_j})'(x)=\frac{(-1)^{r_j-1}}{(p^{(j)}_{r_j-1} - q^{(j)}_{r_j-1}x)^2}.   
	\]
	
	Also, one has the following classical fact about continued fractions: 
	\begin{lemma}\label{lemma_aprox}
		Let $x=[c_0,c_1,c_2,\dots]$ and $\frac{p_n}{q_n}=[c_0,c_1,\dots, c_n]$. Then
		\[
			\frac{1}{2q_nq_{n+1}}<\frac{1}{q_n(q_n+q_{n+1})}<\left|x-\frac{p_{n}}{q_{n}}\right|<\frac{1}{q_nq_{n+1}},
		\]
		and therefore 
		\[
		    \frac{1}{2q_{n+1}}<|q_n x-p_n|<\frac{1}{q_{n+1}}.
        \]
	\end{lemma}
	
	\noindent Therefore, \Cref{lemma_aprox} implies that 
	\[
		(q_{r_j}^{(j)})^2<|(\psi|_{I_j})'(x)|=\frac{1}{(p^{(j)}_{r_j-1} - q^{(j)}_{r_j-1}x)^2}<(2q_{r_j}^{(j)})^2.
	\]
	Thus
	\begin{equation}\label{eq:derivative_inequality}
		(q_{r_j}^{(j)})^2\leq\lambda_j=\inf|\psi'|_{I_j}|\leq\Lambda_j=\sup|\psi'|_{I_j}|\leq (2q_{r_j}^{(j)})^2.
	\end{equation}
	\medskip

Let $s:\NN\to\{0,1,\dots\}\cup\{\infty\}$. Now we will estimate the Hausdorff dimension of Gauss-Cantor sets of the form
\begin{equation*}
    K(B)=\{[0;\gamma_1,\gamma_2,\dots]:\gamma_i\in\{1^r22,1^{r+1}22,\dots,1^{r+s(r)}22\},\forall i\geq 1\}.
\end{equation*}
where $r$ is a large positive integer. The case $s=\infty$ can be written as a Gauss-Cantor set by
\begin{equation*}
    K(B)=\{[0;1^r,\gamma_1,\gamma_2,\dots]:\gamma_i\in\{1^r,221^r,1221^r,\dots,1^{r-1}221^r\},\forall i\geq 1\}.
\end{equation*}

We will perform similar computations as the ones done in \cite[Section 5]{EGRS2024}.

\begin{lemma}\label{lem:hausdorff_dimension_gauss_cantor}
Let $s:\NN\to\{0,1,\dots\}\cup\{\infty\}$ be such that $s(r)(\log s(r))/(r+s(r)+5)\geq C$ for some constant $C>0$. For all large integers $r$, the Gauss-Cantor set $K(B)$ where $B=\{1^r22,1^{r+1}22,\dots,1^{r+s(r)}22\}$, has Hausdorff dimension 
\begin{equation*}
    \dim_H(K(B))=\frac{W(r)}{re^{-c_0}}+O\left(\frac{1}{r}\left(\frac{\log r}{r}\right)^{\min\{1,(1+o(1))s(r)/r\}}\right)
\end{equation*}
where $c_0:=-\log \log((3+\sqrt{5})/2)=0.0383\ldots$. 
\end{lemma}

An elementary estimate on the functions involved in the previous asymptotic expression is the following:

\begin{proposition}
Consider the functions $F_1(x)=W(e^{c_0}x)/x$ and $F_2(x)=(\log x)/x$. Then given constants $0<c_3<c_4$, we have that
\begin{equation}\label{eq:main_terms_difference_error}
    F_j(x+c_3)-F_j(x+c_4)=\left(\log(c_4/c_3)+o(1)\right)\frac{\log x}{x^2},
\end{equation}
as $x\to\infty$ for $j \in \{1,2\}$.
\end{proposition}

We begin the proof of \Cref{lem:hausdorff_dimension_gauss_cantor} by recalling that if $\alpha=\alpha_1\alpha_2\cdots \alpha_m$ and $\beta=\beta_1\beta_2\cdots \beta_n$ are finite words, then
\[
    q_m(\alpha)q_n(\beta)<q_{m+n}(\alpha\beta)<2q_m(\alpha)q_n(\beta).
\]

If $s(r)<\infty$, let $\beta_j=1^{r+j-1}22$ for all $1\leq j\leq s+1$ and if $s(r)=\infty$ let $\beta_0:=1^r$, $\beta_j=1^{j-1}221^r$ for all $1\leq j\leq r$. Then, in either case, one has that
\begin{equation*}
    \Lambda_j=\sup|\psi'|_{I(\beta_j)}|\leq 800\cdot\left(\frac{1+\sqrt5}2\right)^{2(r+j-2)}\leq\left(\frac{1+\sqrt5}2\right)^{2(r+j+5)}.
\end{equation*}
Hence one has that $\dim_H(K(B))\geq\tilde{d}$ where $\tilde d$ is the solution of
\[
    \sum_{j=1}^{s+1}\left(\left(\frac{1+\sqrt5}2\right)^{2(r+j+5)}\right)^{-\tilde d}=1.
\]

Similarly it holds that
\begin{equation*}
    \lambda_j=\inf|\psi'|_{I(\beta_j)}|\geq\left(\frac{1+\sqrt5}2\right)^{2(r+j-2)},
\end{equation*}
so we have $\dim_H(K(B))\leq\hat{d}$ where $\hat{d}$ is the solution of
\begin{equation*}
    \sum_{j=0}^{s}\left(\left(\frac{1+\sqrt5}2\right)^{2(r+j-1)}\right)^{-\hat{d}}=1.
\end{equation*}

Both sums only differ by a constant factor, which only contributes to the error term. Therefore, the proof of Lemma \ref{lem:hausdorff_dimension_gauss_cantor} is reduced to solve
\begin{equation}\label{eq:exponential_sum}
    \sum_{j=0}^{s-1}(\varphi^{2(r+j+L)})^{-d}=1
\end{equation}
where $L\in\{0,6\}$ and, as usual, $\varphi:=(1+\sqrt{5})/2$ is the golden mean. In other words, the proof of Lemma \ref{lem:hausdorff_dimension_gauss_cantor} will be complete once we establish the next proposition.

\begin{proposition}
Let $s:\NN\to\{0,1,\dots\}\cup\{\infty\}$ be such that $s(r)(\log s(r))/(r+s(r)+5)\geq C$ for some constant $C>0$. Then the solution $d\in[0,1]$ of \eqref{eq:exponential_sum} satisfies 
\begin{equation*}
    d=\frac{W(r)}{re^{-c_0}}+O\left(\frac{1}{r}\left(\frac{\log r}{r}\right)^{\min\{1,(1+o(1))s(r)/r\}}\right).
\end{equation*}
\end{proposition}

\begin{proof} We have that
\begin{equation*}
    \sum_{j=0}^{s-1}\varphi^{-2d(r+j+L)}=\varphi^{-2d(r+L)}\left(\frac{1-\varphi^{-2ds}}{1-\varphi^{-2d}}\right).
\end{equation*}

Writing $c_1:=\log(\varphi^2)=0.9624\ldots$, we have $\varphi^2=e^{c_1}$ and $\left(\varphi^{2s}\right)^{-d}= e^{-c_1 sd}$, and on the other hand 
\begin{equation} \label{eq:ast}
    \varphi^{-2d(r+L)}\left(\frac{1-\varphi^{-2ds}}{1-\varphi^{-2d}}\right)=1.
\end{equation}
From \eqref{eq:exponential_sum} we get
\begin{equation*}
    1\geq s\varphi^{-2d(r+L+s-1)},
\end{equation*}
and so by hypothesis $d\geq\frac{\log s}{\log(\varphi^2)}\frac{1}{r+L+s-1}\geq\frac{C}{c_1s}$ (for $r$ large), whence $\varphi^{-2ds}=e^{-c_1sd}<e^{-C}$. Since $d=o(1)$, we have $\varphi^{-2d}=e^{-c_1d}=1-c_1d+O(d^2)$ and therefore $1-\varphi^{-2d}=c_1d+O(d^2)=(1+O(d))c_1d$. It follows that 
\[
    1\ge \frac{\varphi^{-2d(r+L)}(1-e^{-C})}{\left(1-\varphi^{-2d}\right)}=\frac{\varphi^{-2d(r+L)}(1-e^{-C})}{(1+O(d))c_1d}\ge \frac{\varphi^{-2d(r+L)}}{2d}
\]
and thus $0\ge -(r+L)dc_1-\log 2-\log d$. It follows that $-rdc_1\le \log d+O(1) = (1+o(1))\log d$, and thus $\varphi^{-2sd}=e^{-rdc_1(s/r)}=d^{(1+o(1))s/r}$. From \eqref{eq:ast}, we get
\begin{align*}
    1&=\varphi^{-2d(r+L)}\cdot\frac{1-\varphi^{-2ds}}{1-\varphi^{-2d}} \\
    &=\varphi^{-2d(r+L)}\cdot\frac{1-d^{(1+o(1))s/r})}{(1+O(d))c_1 d}=\left(1+O(d^{\min\{1,(1+o(1))s/r\}})\right)\frac{\varphi^{-2rd}}{c_1 d}
\end{align*}
and thus $0=-rd c_1+O(d^{\min\{1,(1+o(1))s/r\}})+c_0-\log d$ and therefore 
\begin{equation}\label{eq:doble_ast}
    rd c_1=|\log d|+c_0+O(d^{\min\{1,(1+o(1))}s/r\}),
\end{equation}
where $c_0:=-\log c_1=0.0383\ldots$.

In particular $rdc_1 = (1+o(1))|\log d|$, which implies that $d=O(\frac{\log r}{r})$. To get a precise asymptotic expression, recall the Lambert function $W\colon [e^{-1},+\infty)\to [-1,+\infty)$, which is the inverse function of $f\colon[-1, +\infty)\to [e^{-1},+\infty), f(x)=x e^x$ (which is increasing in the domain $[-1,+\infty)$) and let $g\colon (0,+\infty)\to\mathbb R$ given by $g(x)=c_1rx+\log x$. We have $g(d)=c_1rd+\log d=c_0+O(d^{\min\{1,(1+o(1))s/r\}})$. Let $d_0\in (0,+\infty)$ be the solution of $g(d_0)=c_0$. Since $g'(x)=c_1r+1/x>c_1r$ for every $x\in (0,+\infty)$, and there exists $t$ between $d_0$ and $d$ such that $|g(d)-c_0|=|g(d)-g(d_0)|=|g'(t)(d-d_0)|\ge c_1r|d-d_0|$, it follows that
\[
    |d-d_0|\le \frac{1}{c_1r} |g(d)-c_0|=O(d^{\min\{1,(1+o(1))s/r\}}/r)=O\left(\frac{1}{r}\left(\frac{\log r}{r}\right)^{\min\{1,(1+o(1))s/r\}}\right)
\]
On the other hand, since $c_1rd_0+\log d_0=g(d_0)=c_0$, we have $d_0 e^{c_1rd_0}=e^{c_0}$, and so $f(c_1rd_0)=c_1rd_0e^{c_1rd_0}=c_1re^{c_0}=r$ and thus $c_1rd_0=W(r)$, which gives a closed expression for $d_0$: $d_0=\frac{1}{re^{-c_0}} W(r)$, from which we get 
\[
    d=\frac{W(r)}{re^{-c_0}}+O\left(\frac{1}{r}\left(\frac{\log r}{r}\right)^{\min\{1,(1+o(1))s/r\}}\right)=\frac{1+O(1/r^{\min\{1,(1+o(1))s/r\}})}{re^{-c_0}}\cdot W(r).
\]

This ends the proof of the desired proposition (and, \emph{a fortiori}, of Lemma \ref{lem:hausdorff_dimension_gauss_cantor}).  
\end{proof}

\subsection{Linear forms in logarithms}

In order to obtain precise information about projections of specific Gauss-Cantor sets we will need to solve some exponential Diophantine equations. An useful tool to solve such equations is the following Theorem of Baker-W{\"u}stholz \cite{Baker}. Let $\eta$ be an algebraic number of degree $d$ over $\mathbb{Q}$ with minimal primitive polynomial over the integers $m(z) := a_0 \prod_{i=1}^{d}(z-\eta^{(i)}) \in \mathbb{Z}[z],$ where the leading coefficient $a_0$ is positive and the $\eta^{(i)}$'s are the conjugates of $\eta$. The absolute logarithmic height of  $\eta$ is given by
$$
h(\eta) := \dfrac{1}{d}\left(\log a_0 + \sum_{i=1}^{d}\log\max\{|\eta^{(i)}|,1\}\right).
$$

\begin{theorem}
	\label{thm:baker}
	Let $\gamma_1, \ldots, \gamma_n$ be algebraic numbers, not 0 or 1, and $d$ the degree of the number field $\QQ(\gamma_1,\dots,\gamma_n)$, and $b_1, \ldots,  b_n$ rational integers. Put
	$$
    \Lambda := \gamma_1^{b_1} \cdots \gamma_n^{b_n} - 1, \quad {\rm with} \quad		B \geq \max\{|b_1|, \ldots ,|b_n|, e^{1/d}\}.
	$$
	Let $A_1,\ldots, A_n$ be real numbers such that $A_i \geq \max\{h(\gamma_i), (1/d)|\log \gamma_i|, 1/d\}$, for $i = 1, \ldots, n$.
	Then, assuming that $\Lambda \not = 0$, we have
	$$
	|\Lambda| > \exp(-C(n,d) A_1 \cdots A_n \log B).
	$$ 
    where $C(n,d)=18(n+1)!n^{n+1}(32d)^{n+2}\log(2nd)$ is a constant. 
\end{theorem}

Although there are versions of this theorem with much better constants, for our purposes the value of $C(n,d)$ is irrelevant. We will use the above theorem only for proving \Cref{lem:exponential_equation}.

\subsection{Estimating denominators of convergents}

\begin{lemma}
Let $r_1,\dots,r_{k+1}\in\NN^*$ and $k\geq 1$ and $r_i\geq 2$ for all $i$. Then
\begin{equation}\label{eq:convergents_first_approximation}
    q(1^{r_1}22 1^{r_2}\dots 221^{r_{k+1}})=\varphi^{r_{k+1}}\left(\frac{3\varphi^3}{\sqrt{5}}\right)q(1^{r_1}22 1^{r_2}\dots 221^{r_{k}})\left(1+\frac{E^\prime_{k+1}}{\varphi^{2r_{k+1}}}\right)\left(1+\frac{E^{\prime\prime}_{k+1}}{\varphi^{2r_{k}}}\right)
\end{equation}
where $|E^\prime_{k+1}|,|E^{\prime\prime}_{k+1}|\leq\frac{1}{3}$.
\end{lemma}

\begin{proof}
Recall that the $n$--term of the linear recurrent sequence $x_{n+2}=x_{n+1}+x_n$ with given initial conditions $x_0$ and $x_1$ is given by 
\begin{equation}\label{eq:linear_recurrence}
    x_n = \left(\frac{x_0\varphi^{-1}+x_1}{\sqrt{5}}\right)\varphi^{n}+\left(\frac{x_0\varphi-x_1}{\sqrt{5}}\right)(-\varphi^{-1})^{n},
\end{equation}
where $\varphi = (1+\sqrt{5})/2$. 

Note that $q(1^{r_1}22\dots 221^{r_k}221^{r_{k+1}})$ is the $(r_{k+1}+1)$--term of the linear recurrent sequence $x_{n+2}=x_{n+1}+x_n$ with initial terms 
\[
    x_1=q(1^{r_1}22\dots 221^{r_k}22)=5q(1^{r_1}22\dots 221^{r_k})+2q(1^{r_1}22\dots 221^{r_k-1})
\]
and
\[
    x_0=q(1^{r_1}22\dots 221^{r_k}2)=2q(1^{r_1}22\dots 221^{r_k})+q(1^{r_1}22\dots 221^{r_k-1})
\]
Thus 
\begin{align*}
    x_0\varphi^{-1}+x_1 &= (3+2\varphi)q(1^{r_1}22\dots 221^{r_k})+(1+\varphi)q(1^{r_1}22\dots 221^{r_k-1}) \\
    &= 3\varphi^2q(1^{r_1}22\dots 221^{r_k})\left(1+\frac{1}{3}\left(\frac{q(1^{r_1}22\dots 221^{r_k-1})-\varphi^{-1}q(1^{r_1}22\dots 221^{r_k})}{q(1^{r_1}22\dots 221^{r_k})}\right)\right)
\end{align*}
Applying the linear recurrence formula \eqref{eq:linear_recurrence}, but now with initial conditions $\tilde{x_1}=q(1^{r_1}22\dots 221^{r_{k-1}}22)$,  $\tilde{x_0}=q(1^{r_1}22\dots 221^{r_{k-1}}2)$, if $k>1$ we get that
\begin{equation*}
    q(1^{r_1}22\dots 221^{r_k-1})-\varphi^{-1}q(1^{r_1}22\dots 221^{r_k})=(1+\varphi^{-2})(-\varphi^{-1})^{r_k}\left(\frac{\tilde{x}_0\varphi-\tilde{x}_1}{\sqrt{5}}\right),
\end{equation*}
while for $k=1$ the initial conditions are $\tilde{x}_0=0$ and $\tilde{x}_1=1$, so 
\begin{equation*}
    q(1^{r_1-1})-\varphi^{-1}q(1^{r_1})=(-\varphi^{-1})^{r_1+1}.
\end{equation*}

The first estimate we want to apply to \eqref{eq:linear_recurrence} is
\begin{equation*}
    x_n=\varphi^{n}\left(\frac{x_0\varphi^{-1}+x_1}{\sqrt{5}}\right)\left(1+\frac{E^\prime}{\varphi^{2n}}\right)
\end{equation*}
where $\abs{E^\prime}\leq \varphi^{-1}$, since $0<x_0\leq x_1$ implies $x_0\varphi-x_1\leq (\varphi-1)x_1 = \varphi^{-1}x_1 < \varphi^{-1}(x_0\varphi^{-1}+x_1)$. Hence
\begin{equation*}
    q(1^{r_1}22\dots 1^{r_2}\dots 221^{r_{k+1}})=\varphi^{r_{k+1}}\left(\frac{3\varphi^3}{\sqrt{5}}\right)q(1^{r_1}22\dots 1^{r_2}\dots 221^{r_{k}})\left(1+\frac{E^\prime}{\varphi^{2r_{k+1}+2}}\right)\left(1+\frac{E^{\prime\prime}}{\varphi^{2r_k}}\right),
\end{equation*}
where for $k>1$, we have
\begin{equation*}
    E^{\prime\prime}=\frac{1}{3}(-1)^{r_k}\varphi^{r_k}(1+\varphi^{-2})\left(\frac{\tilde{x}_0\varphi-\tilde{x}_1}{\sqrt{5}}\right)/q(1^{r_1}22\dots 221^{r_k})
\end{equation*}
while $E^{\prime\prime}=\frac{1}{3}\frac{(-\varphi^{-1})^{r_1+1}}{q(1^{r_1})}\varphi^{2r_1}$, for $k=1$.

It only rests to estimate $|E^{\prime\prime}|$. We will prove $|E^{\prime\prime}|\leq 1/3$ and \eqref{eq:convergents_first_approximation} by induction on $k$. For $k=1$ one has
\begin{align*}
    |E^{\prime\prime}|=\frac{1}{3}\frac{\varphi^{r_1-1}}{q(1^{r_1})}\leq\frac{1}{3}.
\end{align*}

For $k>1$, we use the estimate
\begin{align*}
    \tilde{x}_1-\tilde{x}_0\varphi&=\varphi^{-2}q(1^{r_1}22\dots 221^{r_{k-1}}2)+q(1^{r_1}22\dots 221^{r_{k-1}}) \\
    &\leq (3\varphi^{-2}+1)q(1^{r_1}22\dots 221^{r_{k-1}})
\end{align*}

Hence, by the induction hypothesis over \eqref{eq:convergents_first_approximation},  we get 
\begin{align*}
    |E^{\prime\prime}|&\leq \frac{1}{3}\varphi^{r_k}\frac{(1+\varphi^{-2})}{\sqrt{5}}(3\varphi^{-2}+1)\frac{q(1^{r_1}22\dots 221^{r_{k-1}})}{q(1^{r_1}22\dots 221^{r_{k}})} \\
    &\leq \frac{1}{3} \frac{(1+\varphi^{-2})}{\sqrt{5}}(3\varphi^{-2}+1)\frac{1}{\frac{3\varphi^3}{\sqrt{5}}(1-\varphi^{-4}/3)^2}\leq 1/3.
\end{align*}
\end{proof}

\begin{corollary}\label{cor:convergents_estimate}
Let $r_1,\dots,r_{k}\in\NN^*$ and $k\geq 2$ and $r_i\geq 2$ for all $i$. Then
\begin{equation*}
    q(1^{r_1}22 1^{r_2}\dots 221^{r_{k}})=\frac{\varphi}{\sqrt{5}}\varphi^{r_1+\dots +r_{k}}\left(\frac{3\varphi^3}{\sqrt{5}}\right)^{k-1}\prod_{i=1}^{k-1}\left(1+\frac{E_i}{\varphi^{2r_{i}}}\right)^2\left(1+\frac{E_k^\prime}{\varphi^{2r_{k}}}\right)
\end{equation*}
where $|E_i|\leq\frac{1}{3}$ for all $1\leq i\leq k-1$ and $|E_k^\prime|\leq 1/3$.
\end{corollary}

\begin{proof}
The term $E_i$ is determined for $1\leq i\leq k-1$ by
\begin{equation*}
    \left(1+\frac{E_i}{\varphi^{2r_i}}\right)^2=\left(1+\frac{E_i^\prime}{\varphi^{2r_i}}\right)\left(1+\frac{E_{i+1}^{\prime\prime}}{\varphi^{2r_i}}\right)
\end{equation*}
where we set $E_1^\prime=(-1)^{r_1}\varphi^{-2}$ since $q(1^{r_1})=\frac{\varphi^{r_1+1}}{\sqrt{5}}(1+(-1)^{r_1}\varphi^{-r_1-2})$. For $2\leq i\leq k-1$ the right hand side is clearly between $(1-\frac{1}{3}\varphi^{-2r_i})^2$ and $(1+\frac{1}{3}\varphi^{-2r_i})^2$, which shows that $|E_i|\leq 1/3$. For $i=1$, one uses that $E_1^\prime$ and $E_2^{\prime\prime}$ have opposite signs to arrive to the same conclusion.

\end{proof}

\begin{proposition}\label{prop:base_case}
For any positive integer $n$ it holds
\begin{equation*}
    [0;1^n,2,2,\overline{1}]=\frac{1}{\varphi}+(-1)^{n+1}\left(1+O(\varphi^{-2n})\right)\frac{2(3\varphi-4)}{3\varphi^4}\frac{1}{\varphi^{2n-2}}.
\end{equation*}
\end{proposition}

\begin{proof}
We have 
\begin{align*}
    [0;1^{n},2,2,\overline{1}]&=\left[0;1^{n},2+\frac1{2+\varphi^{-1}}\right]\\
    &=[0;1^{n},4-\varphi]=\frac{(4-\varphi)\mathrm{F}_n+\mathrm{F}_{n-1}}{(4-\varphi)\mathrm{F}_{n+1}+\mathrm{F}_n}\\
    &=\frac{\mathrm{F}_{n-1}/\mathrm{F}_n+(4-\varphi)}{(4-\varphi)\mathrm{F}_{n-1}/\mathrm{F}_n+5-\varphi}.
\end{align*}

Let us estimate $\mathrm{F}_n/\mathrm{F}_{n+1}-\varphi^{-1}$. We have 
\begin{align*}
    \frac{\mathrm{F}_n}{\mathrm{F}_{n+1}}-\frac1{\varphi}&=\frac{\varphi^n-(-\varphi^{-1})^n}{\varphi^{n+1}-(-\varphi^{-1})^{n+1}}-\frac1{\varphi}\\
    &=\left(1+O(\varphi^{-2n})\right)\frac{(-1)^{n+1}(\varphi+\varphi^{-1})\varphi^{-n}}{\varphi^{n+2}}\\
    &=\frac{(-1)^{n+1}\left(3\varphi-4+O(\varphi^{-2n})\right)}{\varphi^{2n}}.
\end{align*}

On the other hand, the identity $\frac{au+b}{cu+d}-\frac{av+b}{cv+d}=\frac{(ad-bc)(u-v)}{(cu+d)(cv+d)}$ applied for $a=1, b=4-\varphi, c=4-\varphi, d=5-\varphi, u=F_n/F_{n+1}$ and $v=\varphi^{-1}$ together with $(cu+d)(cv+d)=(1+O(\varphi^{-2n}))(c\varphi^{-1}+d)^2=(1+O(\varphi^{-2n}))(3\varphi)^2$ gives 
\begin{align*}
    [0;1^{n},2,2,\dots]-[0;\overline{1}]
    =\left(1+O(\varphi^{-2n})\right)\frac{12-6\varphi}{(3\varphi)^2}(u-v)=\left(1+O(\varphi^{-2n})\right)\frac2{3\varphi^4}(u-v).
\end{align*}

Using the asymptotic for $u-v=\mathrm{F}_n/\mathrm{F}_{n+1}-\varphi^{-1}$ concludes the proof. 

\end{proof}

\subsection{Probabilistically unlikely combinatorics}\label{ss.minority-report}\hfill

We denote $\Sigma_t:=\{\theta\in(\mathbb{N}^*)^{\mathbb{Z}}:m(\theta)\leq t\}$, $\Sigma(t,n)$ the set of words of length $n$ appearing in elements of $\Sigma_t$, $\mathcal{Q}_s:=\{c_1\dots c_m: \sizer(c_1\dots c_m)\geq s > \sizer(c_1\dots c_{m-1})\}$. Recall that through the paper we always work with finite words $w=w_1\dots w_n\in\{1,2\}^n$ and that we use $|w|=n$ to denote its length in the alphabet $\{1,2\}$. Given $\delta>0$ and $s\in\NN^*$ define $\Sigma^{(s)}(3+\delta) := \{w\in\mathcal{Q}_s\cap\Sigma(3+\delta,|w|)\}$. In \cite{EGRS2024}, based on the ideas of \cite{B:markoff_numbers}, the authors introduced the notion of weakly renormalizable words to understand the structure of $\Sigma^{(s)}(3+\delta)$: for the sake of convenience of the reader, we recall this notion in the sequel.

Given a pair of words $(u, v)$, we define the operations $\overline{U}(u, v) = (uv, v)$ and $\overline{V}(u, v) = (u, uv)$. Let $\overline{T}$ be the tree obtained by successive applications of the operations $\overline{U}$ and $\overline{V}$, starting at the root $(a, b)$. Let $\overline{P}$ be the set of vertices of $\overline{T}$ and let $\overline{P}_n$, for $n \geq 0$, be the set of elements of $\overline{P}$ whose distance to the root $(a, b)$ is exactly $n$. We call the elements $(\alpha,\beta)\in\overline{P}$ by \emph{alphabets} and call $\overline{T}$ by \emph{the tree of alphabets} (depicted in \Cref{fig:tree_of_alphabets}). We will use these alphabets to write finite words in $\{1,2\}$ through the substitution $a=22$ and $b=11$.

\begin{figure}
\centering
\includegraphics[scale=0.8]{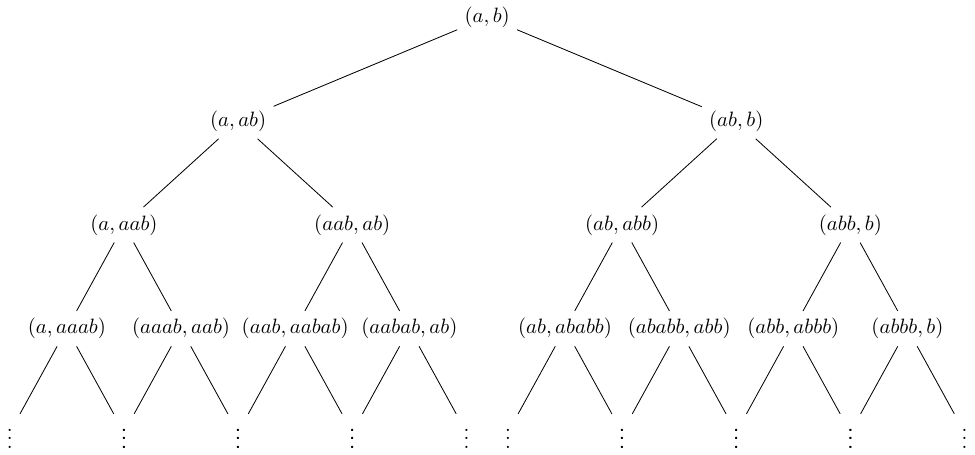}
\caption{The tree of alphabets $\overline{T}$.}
\label{fig:tree_of_alphabets}
\end{figure}

\begin{definition}[Weakly renomalizable]\label{def:weaklyrenormalizable}
Let $(\alpha,\beta)\in \overline P$ and $w\in\{1,2\}^{|w|}$ be a finite word. We say that $w$ is \emph{$(\alpha,\beta)$-weakly renormalizable} if we can write $w=w_1 \gamma w_2$ where $\gamma$ is a word (called the \emph{renormalization kernel}) in the alphabet $\{\alpha,\beta\}$ and $w_1, w_2$ are (possibly empty) finite words with $|w_1|, |w_2|<\max\{|\alpha|,|\beta|\}$ such that $w_2$ is a prefix of $\alpha\beta$ and $w_1$ is a suffix of $\alpha\beta$, with the following restrictions:
if $(\alpha,\beta)=(u,uv)$ for some $(u,v)\in\overline P$ and $\gamma$ ends with $\alpha$, then $|v|\le |w_2|$; if $(\alpha,\beta)=(uv,v)$ for some $(u,v)\in\overline P$ and $\gamma$ starts with $\beta$, then $|u|\le |w_1|$.
\end{definition}

For instance, the finite word $w=2211 222211 222211 2222112211 222$ which appears in periodic sequences such as $m(\overline{222 211 222211 222211 2222112211 2^*2221 1}))\approx 3-1.08\cdot 10^{-21}$ or  $m(\overline{222211 222211 2222^*112211 2222112211})\approx3+4.69\cdot 10^{-10}$, can be exhibited as weakly renormalizable through the alphabets: 
\begin{align*}
    w&=\alpha_0\beta_0 \alpha_0\alpha_0\beta_0 \alpha_0\beta_0 \alpha_0\alpha_0\beta_0\alpha_0\beta_0 \alpha_0\alpha_0 w_2, && (\alpha_0,\beta_0)=(a,b),w_2=2, \\
    &= \beta_1 \alpha_1\beta_1 \alpha_1\beta_1 \alpha_1\beta_1\beta_1 w_2, && (\alpha_1,\beta_1)=(\alpha_0,\alpha_0\beta_0),w_2=\alpha_02, \\
    &= w_1 \alpha_2 \alpha_2 \alpha_2\beta_2 w_2, && (\alpha_2,\beta_2)=(\alpha_1\beta_1,\beta_1),w_1=\beta_1,w_2=\alpha_02, \\
    &= w_1 \alpha_3 \alpha_3 \beta_3 w_2, && (\alpha_3,\beta_3)=(\alpha_2,\alpha_2\beta_2),w_1=\beta_1,w_2=\alpha_02, \\
    &= w_1 \alpha_4 \beta_4 w_2, && (\alpha_4,\beta_4)=(\alpha_3,\alpha_3\beta_3),w_1=\beta_1,w_2=\alpha_02.
\end{align*}
Note that the renormalization kernel will typically cover most of the word.

The alphabets satisfy several nice combinatorial properties (we refer the reader to \cite[Section 3.2]{EGRS2024}). To illustrate this, let us prove two lemmas that will be used later.
\begin{lemma}\label{lem:combinatorial_lemma}
Let $(u,v)\in\overline{P}$ and $\theta\in\{u,v\}$. If $(\alpha,\beta)\in\overline{P}$ is such that $\alpha\beta$ is a subword of $\theta^s$ for some positive integer $s$, then $\alpha\beta$ is subword of $\theta$ and $(u,v)$ is a child of $(\alpha,\beta)$ (in the sense that it can be derived from $(\alpha,\beta)$ by successively applying the operations $\overline{U}$ or $\overline{V}$ finitely many times).
\end{lemma}
\begin{proof}
Suppose $(u,v)\in\overline{P}_n$ and $(\alpha,\beta)\in\overline{P}_m$. Since $\alpha\beta$ always begins with $a$ and ends with $b$, we know that $\theta\not\in\{a,b\}$ and so $n>0$. If $m<n$, then in particular we know that $(u,v)$ is a child of some alphabet $(\tilde{u},\tilde{v})\in\overline{P}_{m}$ and $\theta$ can be written in $(\tilde{u},\tilde{v})$, so $\tau\alpha\beta\tilde{\tau}=\theta^s=w_1\dots w_\ell$ for some words $\tau,\tilde{\tau}$ and  $w_k\in\{\tilde{u},\tilde{v}\}$. By \cite[Lemma 3.10]{EGRS2024} we conclude that $(\alpha,\beta)=(\tilde{u},\tilde{v})$, so in particular $\theta$ contains $\alpha\beta$. In case that $m\geq n$, there is some alphabet $(\tilde{\alpha},\tilde{\beta})\in\overline{P}_n$ which we use to write $\alpha\beta=\tilde{w}_1\dots\tilde{w}_\ell$ with $\tilde{w}_k\in\{\tilde{\alpha},\tilde{\beta}\}$ and such that $\tilde{w}_i=\tilde{\alpha}, \tilde{w}_{i+1}=\tilde{\beta}$ for some $1\leq i<\ell$. Since by hypothesis there are words $\tau,\tilde{\tau}$ such that $\tau\tilde{w}_1\dots\tilde{w}_\ell\tilde{\tau}=\theta^s$, by \cite[Lemma 3.10]{EGRS2024} again we have that $\tilde{w}_k=\theta$ for all $1\leq k\leq \ell$, which is impossible because $\tilde{w}_i=\tilde{\alpha}\neq\tilde{\beta}=\tilde{w}_{i+1}$.
\end{proof}

\begin{lemma}\label{lem:combinatorial_lemma2}
Let $(u,v)\in\overline{P}$ and $\theta\in\{u,v\}$ with $\theta\not\in\{a,b\}$. If $(\alpha,\beta)\in\overline{P}$ is such that $\theta^s$ is a subword of some word in the alphabet $\{\alpha,\beta\}$ for some positive integer $s$ with $|\theta^s|>2\max\{|\alpha|,|\beta|\}$, then $\theta\in\{\alpha,\beta\}$ or $(u,v)$ is a child of $(\alpha,\beta)$.
\end{lemma}

\begin{proof} If $\alpha\beta$ is a subword of $\theta^s$, then by \Cref{lem:combinatorial_lemma} we obtain immediately that $(u,v)$ is a child of $(\alpha,\beta)$. Hence assume that $\theta^s$ does not contain $\alpha\beta$. By hypothesis this gives that $\theta^s$ is a subword of $\eta^\prime\eta^t\eta^\prime$ with $\{\alpha,\beta\}=\{\eta,\eta^\prime\}$, $\eta^t$ a subword of $\theta^s$ and $t\geq 1$. Then again by \Cref{lem:combinatorial_lemma} we have that $\eta$ is a subword of $\theta$ (the case $\eta\in\{a,b\}$ is direct). If some $\theta$ is completely contained in $\eta^t$, then since $\theta\not\in\{a,b\}$, by \Cref{lem:combinatorial_lemma} in this case one has $\theta=\eta$. Now assume that $\theta$ is not contained in $\eta^t$. Note that in the cases $(\alpha,\beta)\in\{(a,a^nb),(ab^n,b)\}$, also yield that $\alpha\beta$ is a subword of $\theta^s$, since $\theta$ must begin and end with $a$ and $b$ respectively. Since this situation already appeared, assume further that $\min\{|\alpha|,|\beta|\}>2$. Writing $\eta=\tilde{\alpha}\tilde{\beta}$ for some $(\tilde{\alpha},\tilde{\beta})\in\overline{P}$ by \Cref{lem:combinatorial_lemma} we have that $(u,v)$ is a child of $(\tilde{\alpha},\tilde{\beta})$ and $\theta$ can be written in that alphabet. In particular either $\theta\in\{\alpha,\beta\}$ (i.e. $(u,v)$ is an immediate child) or $\theta$ can be written in a larger alphabet (since contains $\tilde{\alpha}\tilde{\beta}$) and so $\theta$ can not begin and finish simultaneously with $\tilde{\alpha}\tilde{\beta}$. Moreover, $\theta$ must begin with $\tilde{\alpha}$ and end with $\tilde{\beta}$. Since there are words $\tau,\tilde{\tau}$ such that $\tau(\tilde{\alpha}\tilde{\beta})^t\tilde{\tau}=\theta^s$, by \cite[Lemma 3.10]{EGRS2024} we obtain that $\tau,\tilde{\tau}$ are also words in the alphabet $(\tilde{\alpha},\tilde{\beta})$.  Hence $\eta^\prime\in\{\tilde{\alpha},\tilde{\beta}\}$ is impossible, because we are assuming $\theta$ is not contained in $\eta^t$. The final situation is when $\eta^\prime\in\{(\tilde{\alpha}\tilde{\beta})^k\tilde{\beta},\tilde{\alpha}(\tilde{\alpha}\tilde{\beta})^k\}$. However, since $\theta^s$ begins with $\tilde{\alpha}$ and ends with $\tilde{\beta}$, the fact that $\theta^s=\tau(\tilde{\alpha}\tilde{\beta})^t\tilde{\tau}$ is contained in $\eta^\prime(\tilde{\alpha}\tilde{\beta})^t\eta^\prime$ forces $\theta^s$ to contain $\alpha\beta$ or to begin and end with $\tilde{\alpha}\tilde{\beta}$, a contradiction. 
\end{proof}

Clearly if one wants to understand $\Sigma^{(s)}(3+\delta)$ using alphabets, the scale $s$ must be bounded in terms of $\delta$, since a very large word could not be well described by a single alphabet (in the sense of \Cref{def:weaklyrenormalizable}).

By the same proof of \cite[Corollary 3.26]{EGRS2024}, we have the following lemma.

\begin{lemma}\label{lem:renormalization}
Let $w \in \Sigma^{(r-4)}(3 + e^{-r})$. Then, if necessary, by extending $w$ to $\tilde{w}\in\{2w1,2w,w1\}$, we have that $\tilde{w}$ is $(\alpha, \beta)$-weakly renormalizable for some alphabet $(\alpha, \beta)$ satisfying $|\alpha\beta| \geq r/6$. Moreover any such extension satisfies $\sizer(\tilde{w})\leq r$.
\end{lemma}

In particular, applying this lemma for $w$ and extending $w$ if necessary to some word $\sizer(w)\leq r$, there is a sequence of alphabets $(\alpha_j,\beta_j)$ such that, for all $0\le j\le m$, $w$ is $(\alpha_j,\beta_j)$-weakly renormalizable, with
\[
(\alpha_0,\beta_0)=(a,b) \quad \text{ and } \quad (\alpha_{j+1},\beta_{j+1})\in\{(\alpha_j\beta_j,\beta_j), (\alpha_j,\alpha_j\beta_j)\},
\]
for each $0 \leq j<m$ and $|\alpha_m\beta_m|\ge r/6$. 

The alphabets are not only useful for understanding single factors $w\in\Sigma^{(r-4)}(3+e^{-r})$ but also for writing its admissible extensions. Depending on the structure of the renormalization kernel of $w$, we can write large extensions $\overline{w}\in\Sigma^{(Tr)}(3+e^{-r})$ (where $T\in\NN$ depends on $w$ and $r$) that are admissible $w\overline{w}\in\Sigma(3+e^{-r},|w\overline{w}|)$ in the same alphabets we used to write $w$. This is the description of the lemmas present in \cite[Section 4.1]{EGRS2024}.

We will show below that the words $w\in\Sigma^{(r-4)}(3+e^{-r})$ not containing a factor $1^s$ with size comparable to $r$, i.e., such that $\sizer(1^s)\geq r-C(\log r)^2$ for some constant $C$, are ``associated" to a product of Cantor sets each with upper box (and Hausdorff) dimension less than \[\frac{\log r-\log\log r}r.\] 
In other words, the words in $\Sigma^{(r-4)}(3+e^{-r})$ not containing big blocks of $1$'s should be a minority. More concretely, by performing a similar strategy as was done in \cite[Pages 43-47]{EGRS2024}, we shall see below that these words have few continuations, i.e., they intersect few rectangles, once one takes an advanced step in the construction (related to a choice of parameter $T=\lfloor (\log r)^2\rfloor$ in \cite{EGRS2024}). This is the content of the next statement (whose proof will use several lemmas from \cite[Section 4]{EGRS2024}). The parameter $T$ is chosen that way because we need to refine several times our rectangles (i.e., extend the words to a large scale) in order to detect that these combinatorics are (exponentially) less likely than the ones coming from large blocks of $1$'s. We will count the rectangles associated with large blocks of $1$'s in \Cref{sec:counting_words} and compare it with the counting of the next statement.

\begin{lemma}\label{lem:few_long_combinatorics}
Let $w\in\Sigma^{(r-4)}(3+e^{-r})$ for $r\in\NN$ big enough. Suppose that $w$ does not contain a factor of the form $1^s$ with $\sizer(1^s)\geq r-64(\log r)^2$. Then the number of continuations $\overline{w}\in\Sigma^{(r\lfloor (\log r)^2\rfloor-2)}(3+e^{-r})$ such that $w\overline{w}\in\Sigma(3+e^{-r},|w\overline{w}|)$ is at most 
\begin{equation*}
    e^{(\log r)^2(\log r-\log\log r-\log(5/4))}.
\end{equation*}
\end{lemma}

\begin{proof}

\noindent We consider a renormalization $(\alpha_t,\beta_t)$ of $w$ (extending $w$ if necessary) with
\[
    \frac{r}{512(\log r)^2}\le |\alpha_t| + |\beta_t|<\frac{r}{256(\log r)^2}.
\]

Note that from \eqref{eq:rcomparedwithsize} we have that $\sizer(w\overline{w})\leq r((\log r)^2+1)+2$. We want to write $w\overline{w}$ in a suitable alphabet (depending on the structure of $w$) and then count the number of choices $(w,\overline{w})$.

\medbreak
\noindent\textbf{Case 1:} Let $\theta\in\{\alpha_t,\beta_t\}$ be such that $w$ has a factor $\theta^s$ with $|\theta^s|\geq|w|/3$.
\medbreak

If $\theta=uv$ with $(u,v)\in\overline{P}$, take the first factor $(uv)^{s_1}$ of $w$ such that $\sizer((uv)^{s_1})\leq r-4|uv|$ with $s_1$ maximal. Thus we have that either $s_1=s$ or that $\sizer((uv)^{s_1})\geq r-6|uv|$ for $r$ big enough. Note that from \eqref{eq:rcomparedwithsize} we have that $|w|\geq 3r/10$. From the previous observations we always have $s_1\geq r/(10|uv|)>128(\log r)^2/5$. In particular $s_1^2|uv|\log\varphi\geq 6r(\log r)^2/5$ holds. We need this inequality to apply \cite[Lemma 4.2]{EGRS2024}. By this lemma, from the first such factor $(uv)^{s_1}$, the sequence should be \[
    (uv)^{s_1}\theta_1(uv)^{s_2}\theta_2\dots(uv)^{s_\ell},
\]
with $\theta_j\in\{uuv, uvv\}$ and $\ell\leq 2.1r((\log r)^2+2)/|(uv)^{s_1}|+1<22(\log r)^2$. Moreover:
\begin{itemize}
\item If $\sizer((uv)^{s_{j+1}})\leq r-6|uv|$, then $\theta_j=\theta_{j+1}$, for $1\leq j\leq \ell-2$.
\item If $\sizer((uv)^{s_j})\leq r-10|uv|$, then $|s_j-s_{j+1}|\leq 1$ for $1\leq j\leq \ell-2$. 
\end{itemize}

Therefore, from this factor $\theta^{s_1}$, the continuation of $w\overline w$ is an initial factor of a word of the form $(uv)^{s_1}\theta_1(uv)^{s_2}\theta_2\ldots\theta_{\ell-1}(uv)^{s_\ell}$, and 
\[
    r(\log r)^2\le \sizer((uv)^{s_1}\theta_1(uv)^{s_2}\theta_2\dots(uv)^{s_{\ell}})<r((\log r)^2+3),
\]
with $\theta_j\in\{uuv, uvv\}$, $\ell\le 22(\log r)^2$ (and such that $(uv)^{s_1}\theta_1(uv)^{s_2}\theta_2\ldots\theta_{\ell-1}$ is an initial factor of this word beginning in this factor $(uv)^{s_1}$ and going till the end of $w\overline w$). The lower bound is because by definition $\sizer(w\overline{w})\geq\sizer(w)+\sizer(\overline{w})-1\geq r((\log r)^2+1)-5$.

\medbreak
\noindent\textbf{Case 1.1:} Suppose that $|\theta|>r^{31/32}$ and  $\theta=uv$ with $(u,v)\in\overline{P}$. 
\medbreak

By \eqref{eq:rcomparedwithsize} we will get that
\begin{equation*}
    s_1+\dots+s_\ell+\ell \leq 2r^{1/32}(\log r)^2=:M.
\end{equation*}

The number of solutions $(s_1,\dots,s_\ell,\ell)$ is at most
\[\ell\cdot\binom{M+\ell}{\ell}\leq (3r^{1/32}(\log r)^2)^{22(\log r)^2+1}<e^{12(\log r)^3/15}.\]

\medbreak
\noindent\textbf{Case 1.2:}  
Suppose $|\theta|\leq r^{31/32}$ and $\sizer(\theta^{s_1})\geq r-32(\log r)^2|\theta|$.  \\
\medbreak

\noindent\textbf{Case 1.2.1} Suppose that that $|\theta|>2$.

So, $\theta=uv$ with $(u,v)\in\overline{P}$. Then \cite[Lemma 4.2]{EGRS2024} guarantees that from $(uv)^{s_1}$ the extension $w\overline{w}$ is covered by
\[
    (uv)^{s_1}\theta_1(uv)^{s_2}\ldots\theta_{\ell-1}(uv)^{s_\ell},
\]
where each $\theta_i\in\{uuv,uvv\}$, $\ell\leq 22(\log r)^2$ for $r$ big enough. Moreover, we will have that $\sizer((uv)^{s_j})\ge r-11(\log r)^2|uv|>r-\lfloor 12(\log r)^2\cdot r^{31/32}\rfloor\equalscolon \widetilde{M}$ for all $1\leq j\leq\ell$. Let $s_0$ be the smallest integer satisfying $\sizer(\theta^{s_0})\ge \widetilde{M}$. Then, $s_j=s_0+\tilde s_j$ with $\tilde s_j\ge 0$ for each $1 \leq j\le \ell$. By \eqref{eq:rcomparedwithsize} we have that
\begin{align*}
    r((\log r)^2+2)\geq\sizer(w\overline{w})&\geq\sizer((uv)^{s_1}\theta_1(uv)^{s_2}\theta_2\dotsb(uv)^{s_\ell}) \\
    &\geq \ell\sizer(\theta^{s_0})+(\tilde s_1+\dots+\tilde s_\ell)|uv|\log(\varphi^2)-8\ell     
\end{align*}

Since $((\log r)^2+3)\widetilde{M}=((\log r)^2+3)(r-\lfloor 12(\log r)^2 \cdot r^{31/32}\rfloor)>((\log r)^2+2)r$, it follows that, given $1\le \ell\le 22(\log r)^2$, the number of choices of the $s_j, j\le \ell$ is at most the number of natural solutions of 
\[
    \tilde s_1+\tilde s_2+\dots+\tilde s_{\ell}\le ((\log r)^2+3-\ell)\widetilde{M}/3.8,
\]
which by \cite[Lemma 4.5]{EGRS2024} is at most
\begin{equation*}
    ((\log r)^2+3)e^{W(\widetilde{M}/3.8)((\log r)^2+4)}<e^{W(\widetilde{M}/3)(\log r)^2}<e^{(\log r)^2(\log r-\log\log r-1)},
\end{equation*}
for $r$ big enough.

\noindent\textbf{Case 1.2.2} Assume that $\theta=a$:

Then, $w$ has a factor $2^{s_1}$ satisfying $\sizer(2^{s_1})\ge r-64(\log r)^2$. Observe that $(s_1+1)\log (3+2\sqrt{2})\geq \sizer(2^{s_1})\geq r-64(\log r)^2$, so  $((\log r)^2+2)r\leq s_1^2(\log (3+2\sqrt{2}))/4$ holds for large $r$. Using \cite[Lemma 4.3]{EGRS2024}, from this factor $2^{s_1}$ the continuation of $w\overline w$ is an initial factor of a word of the form $2^{s_1}b 2^{s_2}b\ldots b 2^{s_\ell}$ with $\ell<22(\log r)^2$ for large $r$ and with $\sizer(2^{s_j})\ge r-11(\log r)^2>r-\lfloor 12(\log r)^2\rfloor\equalscolon N$ (notice that if $\sizer(2^{s_j})<r-7$ then $s_j$ is even and $|s_{j+1}-s_j|\in \{0,2\}$). Let $s_0$ minimum such that $\sizer(2^{s_0})\ge N$. Then $s_j=s_0+\tilde s_j$ with $\tilde s_j\ge 0$ for each $1 \leq j\le \ell$. Notice that $\overline w$ is determined by the choice of $(\ell,s_1,s_2,\dots,s_\ell)$. 

We have, using equation (A.4) of \cite{EGRS2024}, that
\begin{align*}
    ((\log r)^2+2)r&>\sizer(2^{s_1}112^{s_2}11\ldots 2^{s_{\ell}}) \\
    &\ge (s_1+\ldots+s_\ell+\ell-2)\log(3+2\sqrt{2}) \\
    &\ge \ell N+5(\tilde s_1+\tilde s_2+\dots+\tilde s_{\ell})/3.
\end{align*}
Since 
\[
    ((\log r)^2 +3)N=((\log r)^2+3)(r-\lfloor 12\log^2 r\rfloor)>((\log r)^2+2)r,
\]
it follows that, given $1\le \ell\le 22(\log r)^2$, the number of choices of the $s_j, j\le \ell$ is at most the number of natural solutions of $\tilde s_1+\tilde s_2+\dots+\tilde s_{\ell}\le 3((\log r)^2+3-\ell)N/5$. By \cite[Lemma 4.5]{EGRS2024}, it is at most
\[
    ((\log r)^2+3)e^{W(3N/5)((\log r)^2+3)}=e^{(\log(N)-\log\log(N)-\log(4/3)+o(1))(\log r)^2},
\]
since $\log N=\log r+o(1)$. We have at most $O(r^3)$ choices for $w$, so we have at most
\[
    O(r^3) e^{(\log N-\log\log N-\log(4/3)+o(1))(\log r)^2}=e^{(\log r-\log\log r-\log(4/3)+o(1))(\log r)^2}
\]
such words $w\overline w$ for large $r$.

\noindent\textbf{Case 1.2.3} Assume that $\theta=b$:

This situation does not occur because $w$ does not contain a factor $1^s$ with $\sizer(1^s)\geq r-32(\log r)^2|\theta|=r-64(\log r)^2$.

\medbreak
\noindent\textbf{Case 1.3:} Let $\theta\in\{\alpha_t,\beta_t\}$ be such that $w$ has a factor $\theta^s$ with $|\theta^s|\geq|w|/3$, $|\theta|\leq r^{31/32}$  and that all such factors satisfy $\sizer(\theta^{s})< r-32(\log r)^2|\theta|$. \\
\medbreak

We will assume that $\theta=\alpha_t$, since the case of $w$ having such a factor $\theta=\beta_t^s$ is completely analogous. 

We claim there is $\beta$ such that $(\alpha_t,\beta)\in\overline{P}$,  $\beta(\alpha_t^{s_1})\beta$ is a subword of $w$, and the continuation $w\overline{w}$ has the form $\beta(\alpha_t)^{s_1}\beta(\alpha_t)^{s_2}\ldots(\alpha_t)^{s_\ell}$ with $|s_j-s_{j+1}|\leq 1$ and for all $1\leq j\leq\ell$. If $(\alpha_t,\beta_t)=(uv,v)$ for some $(u,v)\in\overline{P}$ or $(\alpha_t,\beta)=(a,b)$ then we set $\beta=\beta_t$, while if $\beta_t=\alpha_t^k\tilde{\beta}$ with $(\alpha_t,\Tilde{\beta})\in\overline{P}$ and $|\Tilde{\beta}|\leq|\alpha_t|$, then we set $\beta=\Tilde{\beta}$.

Indeed, if $(\alpha_t,\beta)=(uv,v)$ with $(u,v)\in\overline{P}$ then we use \cite[Lemma 4.2]{EGRS2024} to obtain a continuation $\hat{w}$ with  $\hat{\gamma}=(uv)^{s_1}\theta_1(uv)^{s_2}\theta_2\ldots$ with $\theta_i\in\{uuv,uvv\}$ and $\hat{w}=\hat{\gamma}w_2$. Moreover $\ell\leq 22(\log r)^2$. Since $\sizer(\alpha_t^{s_1})\leq r-32(\log r)^2|\theta|$, by induction $\sizer(\theta^{s_j})\leq r-(32(\log r)^2-2j)|\theta|$ and since $\ell\leq 22(\log r)^2$, we get that all $\theta_j$ are equal to $\theta_j=\alpha_t\beta=uvv$ (since $\alpha_t^s\beta_t=(uv)^fv$ for some $f\geq 1$ and the fact that $uuv$ and $uv$ start with $v_a$)  and that $|s_j-s_{j+1}|\leq 1$ for all $j\geq 1$. 
        
When $\alpha_t=a$, we use \cite[Lemma 4.3]{EGRS2024} to find a continuation $\hat{w}$ such that $\hat{w}=2^{e_1}b 2^{e_2}b\ldots$. But observe that $2^{e_1}=\alpha_t^{e_1/2}$ is inside $w$, so by hypothesis $\sizer(2^{e_1})<r-64(\log r)^2$ and there is $b$ before $2^{e_1}$, so $e_2$ is even and $|e_1-e_2|\in\{0,2\}$. By induction we obtain $\sizer(2^{e_j})\leq r-(64(\log r)^2-2j)$, which forces all $e_j$ to be even and $|e_{j+1}-e_j|\in\{0,2\}$. In this case $\beta=b$, so we get $\hat{w}=2^{e_1}b 2^{e_2}b\ldots=\alpha_t^{s_1}\beta\alpha_t^{s_2}\beta\ldots$ with $|s_{j+1}-s_j|\leq 1$ for all $j\geq 1$.

Therefore, given a continuation $\overline w$ of $w$, there is $\beta$ with $(\alpha_t,\beta)\in \overline P$ such that $w$ has a factor $\beta(\alpha_t)^{s_1}\beta$, after which the continuation of $w \overline w$ is a concatenation of at most $22(\log r)^2$ sequences of the type $(\alpha_t)^{s_j}\beta$, $2\le j\le 22(\log r)^2$ with $|s_{j+1}-s_j|\le 1$ for every $j\ge 1$. This gives at most $3^{22(\log r)^2}$ continuations of $\beta(\alpha_t)^{s_1}\beta$, and so, since we have at most $O(r^3)$ choices for $w$, we have in total, $O(3^{22(\log r)^2}\cdot r^3)=O(e^{10(\log r)^2})$ such words $w\overline w$.

\medbreak
\noindent\textbf{Case 2:} Suppose that, for every factor of the form $\alpha_t^s$ or $\beta_t^s$ of $w$, we have $|\alpha_t^s|<|w|/3$ and $|\beta_t^s|<|w|/3$. \\

By \cite[Lemma 4.4]{EGRS2024}, we have that $w\overline{w}$ is completely contained in a word in the alphabet $(\alpha_t,\beta_t)$. Observe that we can apply this lemma since $\delta^\prime=1-(3\log\varphi)^{-1}>3/10$ so $\delta^\prime r/(16|\alpha_t\beta_t|)>2(\log r)^2$. Thus, the continuation of the first factor of the form $\alpha_t\beta_t$ of $w$ in $w\overline w$ is a concatenation of factors of the form $\alpha_t^{j}\beta_t$ or $\alpha_t\beta_t^{j}$. The number of such factors is at most $|w \overline w|/|\alpha_t\beta_t|\leq 1.1r((\log r)^2+2)/|\alpha_t\beta_t|\leq 300(\log r)^{4}$. By \eqref{eq:rcomparedwithsize} we have that $\sizer(\theta^s)\leq 7r/10<r-10|\theta|$ for both $\theta\in\{\alpha_t,\beta_t\}$ for large $r$.

We claim that if $\theta\in\{\alpha_t,\beta_t\}$ satisfies $|\theta|\leq r^{31/32}$, then all the respective factors $\theta^j$ inside $w\overline{w}$ satisfy $|\theta^j|<r-10|\theta|$. Otherwise, if $\theta^\prime\in\{\alpha_t,\beta_t\}$, $\theta^\prime\neq\theta$, then $w\overline{w}$ would contain a factor of the form $\theta^{s_1}\theta^\prime\theta^{s_2}\theta^\prime\dotsb\theta^{s_\ell}$ with $r-10|\theta|\leq\sizer(\theta^{s_\ell})\leq r-6|\theta|$, $\sizer(\theta^{s_1})<7r/10$ and $|\theta^{s_i}|<r-6|\theta|$ for all $1\leq i\leq\ell$. By \cite[Lemma 4.1]{EGRS2024} we will have that $|s_i-s_{i+1}|\leq 1$ for all $i$. By induction and by \eqref{eq:rcomparedwithsize}, one must have also that $\sizer(\theta^i)\geq r-(i+12)|\theta|$. Since the number of such factors is at most $300(\log r)^{4}$, this yields $7r/10>\sizer(\theta^{s_1})>r-(300(\log r)^4+12)r^{31/32}$ which is a contradiction for large $r$.

Since the number of factors $\alpha_t\beta_t$ inside $w\overline{w}$ is very large, we will write this continuation in larger alphabets instead. Given a subblock $\tilde{w}\in\Sigma^{(r\lfloor \sqrt{\log r}\rfloor)}(3+e^{-r})$, we will use \Cref{lem:renormalization} to write an extension of the prefix $\hat{w}\in\Sigma^{(r-4)}(3+e^{-r})$ of $\tilde{w}$ as a $(\hat{\alpha},\hat{\beta})$--renormalizable word with 
\[
\frac{r}{2\sqrt{\log r}}<|\hat{\alpha}|+|\hat{\beta}|<\frac{r}{\sqrt{\log r}}.
\]

We have two possibilities: either every factor $\hat{\theta}^s$ of $\hat{w}$ where $\hat{\theta}\in\{\hat{\alpha},\hat{\beta}\}$ satisfies $\sizer(\hat{\theta}^s)\leq r-10|\hat{\theta}|$ or there is a factor with $\sizer(\hat{\theta}^s)> r-10|\hat{\theta}|$. We call them blocks of first and second type, respectively. To count the possibilities for $(w,\overline{w})$, we will divide $w\overline{w}$ in such blocks, that is, the word $w\overline{w}$ is contained in $\tilde{w}_1\dots\tilde{w}_{L+1}\tilde{\tau}$ with $\tilde{w}_i\in\Sigma^{(r\lfloor \sqrt{\log r}\rfloor)}(3+e^{-r})$ for $1\leq i\leq L+1$ and where all the first $L$ blocks are of the first type except possibly the last one $\tilde{w}_{L+1}$. We construct the blocks greedily: if no block of the second type appeared until covering all $w\overline{w}$, then we set $\tilde{\tau}$ as empty, otherwise we let $\tilde{w}_{L+1}$ to be the last block and $\tilde{\tau}$ to be rest of the continuation $w\overline{w}$ after $\tilde{w}_{L+1}$. To count how many choices we have for each $\tilde{w}_i$ we will use the alphabets that appeared in the prefix of each $\tilde{w}_i$.

For blocks of the first type, we use \cite[Lemma 4.4]{EGRS2024} to cover $\tilde{w}$ with a word in the alphabet $(\hat{\alpha},\hat{\beta})$. This can be done since $\delta^\prime r/(16|\hat{\alpha}\hat{\beta}|)>4\sqrt{\log r}$. Recall that this word will be a concatenation of factors of the form $\hat{\alpha}^{j}\hat{\beta}$ or $\hat{\alpha}\hat{\beta}^{j}$. The number of such factors is at most $|\tilde{w}|/|\hat{\alpha}\hat{\beta}|\leq 1.1r(\sqrt{\log r})/|\hat{\alpha}\hat{\beta}|\leq 2.2\log r$. If we have to consecutive such factors $\hat{\alpha}^{j_1}\hat{\beta}$ and $\hat{\alpha}^{j_2}\hat{\beta}$ (or $\hat{\alpha}\hat{\beta}^{j_1}$ and $\hat{\alpha}\hat{\beta}^{j_2}$), then $|j_1-j_2|\le 1$, and if we have two consecutive factors $\hat{\beta}\hat{\alpha}^{j_1}\hat{\beta}$ and $\hat{\alpha}\hat{\beta}^{j_2}\hat{\alpha}$ then $2\le |j_1|+|j_2| \le 3$, because of \cite[Lemma 4.1]{EGRS2024} and the fact that $\sizer(\hat{\alpha}^{j})\leq r-10|\hat{\alpha}|$ and $\sizer(\hat{\beta}^{j})\leq r-10|\hat{\beta}|$ respectively. This implies that each of these factors of the form $\hat{\alpha}^{j}\hat{\beta}$ or $\hat{\alpha}\hat{\beta}^{j}$ has at most $3$ continuations of this form, and so the number of such continuations $\tilde{w}$ of $\hat{w}$ is at most $3^{2.2\log r}<e^{3\log r}$.

For a block of second type, that is, when $\hat{w}$ contains a factor with $\sizer(\hat{\theta}^s)>r-10|\hat{\theta}|$, we claim that we also have $|\hat{\theta}|>r^{31/32}$. Indeed, we claim that $\hat{\theta}=\hat{u}\hat{v}$ for some $(\hat{u},\hat{v})\in\overline{P}$ which is a child of the alphabet $(\alpha_t,\beta_t)$. To see this, note that $\theta^s$ is a subword of some word in the alphabet $(\alpha_t,\beta_t)$ (because $\hat{\theta}^s$ is a subword of $w\overline{w}$). By \eqref{eq:rcomparedwithsize} we have $|\hat{\theta}^s|>r/2$ and since $|\alpha_t\beta_t|<\frac{r}{256(\log r)^2}$, we must have (assuming $r$ is large) by \Cref{lem:combinatorial_lemma2} that either $(\hat{u},\hat{v})$ is a child of $(\alpha_t,\beta_t)$ or $\hat{\theta}\in\{\alpha_t,\beta_t\}$. However the second situation is not allowed because it gives the existence of a factor $\theta^s$ of $\hat{w}$ with $\theta\in\{\alpha_t,\beta_t\}$ with $\sizer(\theta^s)>r-10|\theta|$, a contradiction. As consequence $\hat{\theta}$ contains $\alpha_t\beta_t$ and so $|\hat{\theta}|\geq|\alpha_t\beta_t|>r^{31/32}$ (for large $r$). When this second type factor $\tilde{w}_{L+1}$ appears, we can not only cover $\tilde{w}$, but we can also cover the rest of the word $w\overline{w}$ from this factor $\tilde{w}_{L+1}$. In other words we have at most one block of second type, which will be the last one and $\tilde{\tau}$ can be written in the same alphabet $(\hat{\alpha},\hat{\beta})$. The counting of possibilities for $(\tilde{w}_{L+1},\tilde{\tau})$ in this second case is completely analogous to Case 1.1, which now yields at most $(3r^{1/32}(\log r)^2)^{3(\log r)^2}<e^{(\log r)^3/8}$.

Finally, note that we can divide $w\overline{w}$ in at most $L\leq\sizer(w\overline{w})/(r\sqrt{\log r})<(\log r)^{3/2}+2$ blocks $\tilde{w}_i$. Since all of them but one are going to correspond to blocks of the first type, the number of continuations $w\overline{w}$ is at most

\begin{equation*}
    O(r^3(e^{3\log r})^{(\log r)^{3/2}+2}e^{(\log r)^3/8})=O(e^{(\log r)^3/7}).
\end{equation*}

\end{proof}

\section{Projections of Cantor sets}\label{sec:Cantor_sets}

In general terms, our strategy will be to transform the problem of local uniqueness for an arithmetic sum of Gauss-Cantor sets $K_1+K_2$ to a problem of local uniqueness for an arithmetic difference $K-\mu K$, where $K$ is a much more simple Gauss-Cantor set and $\mu\in \RR$ is very close to a product of some fixed algebraic numbers of degree 2. See \Cref{fig:Markovpartition,fig:not_local_uniqueness,fig:local_uniqueness} for a geometrical interpretation of local uniqueness for product of Cantor sets.

We begin by recalling that most of the finite words $w\in\{1,2\}^{|w|}$ with size at most $\sizer(w)\leq r$  that are subwords of bi-infinite sequences $\underline{c}\in\Sigma_{3+e^{-r}}$, $m(\underline{c})=\lambda_0(\underline{c})$ have the form $w=\dots221^{s_{-2}}221^{s_{-1}}22^*1^{s_1}221^{s_2}22\dots$ and consequently $\underline{c}$ has Markov value
\begin{align*}
    m(\underline{c})=\lambda_0(\underline{c})&=[2;2,1^{s_{-1}},2,2,1^{s_{-2}},2,2,\dots]+[0;1^{s_{1}},2,2,1^{s_{2}},2,2,\dots] \\
    &=3+[0;1^{s_{1}},2,2,1^{s_{2}},2,2,\dots]-[0;1^{s_{-1}+2},2,2,1^{s_{-2}},2,2,\dots],
\end{align*}
where we used \eqref{eq:cont_frac_identity}. In fact, this is exactly the content of \Cref{lem:few_long_combinatorics}. 

This motivates the study of the Gauss-Cantor set $K_{\mathrm{big}}$ given by
\begin{align*}
    K_{\mathrm{big}}&:=\left\{[0;1^{s_1},2,2,1^{s_2},2,2,\dots]:s_i\geq n-\lfloor 128(\log n)^2\rfloor \text{ for all } i\geq 1\right\} \\
    &=\left\{[0;1^{n-128\lfloor(\log n)^2\rfloor},\gamma_1,\gamma_2,\dots]:\gamma_i\in\{221^{n-\lfloor 128(\log n)^2\rfloor},1\} \text{ for all } i\geq 1\right\},
\end{align*}
where $n\in\NN$. We also define 
\begin{align*}
    K_{\mathrm{small}}&:=\left\{[0;1^{s_1},2,2,1^{s_2},2,2,\dots]:n+\lfloor n/\sqrt{\log n}\rfloor\geq s_i\geq n \text{ for all } i\geq 1\right\} \\
    &=\left\{[0;\gamma_1,\gamma_2,\dots]:\gamma_i\in\{1^{n}22,\dots,1^{n+\lfloor n/\sqrt{\log n}\rfloor}22\} \text{ for all } i\geq 1\right\}.
\end{align*}

Note that $K_{\mathrm{small}}\subset K_{\mathrm{big}}$ and that their Hausdorff dimensions are very similar. Indeed, it follows from \Cref{lem:hausdorff_dimension_gauss_cantor} that
\begin{equation*}
    \dim_H(K_{\mathrm{big}})=\frac{W(m)}{me^{-c_0}}+O\left(\frac{\log m}{m^2}\right)   
\end{equation*}
where $c_0=-\log\log(\varphi^2)$ and $m=n-\lfloor128(\log n)^2\rfloor$ and also that 
\begin{equation*}
    \dim_H(K_{\mathrm{small}})=\frac{W(n)}{ne^{-c_0}}+O\left(\frac{e^{-(\log n-\log\log n)/\sqrt{\log n}}}{n}\right),
\end{equation*}
so $\dim_H(K_{\mathrm{big}})/\dim_H(K_{\mathrm{small}})\to 1$ as $n\to\infty$ because of the approximation $W(x)=\log x - \log\log x + o(1)$ for $x\to\infty$.

We will show that $(M\setminus L)\cap(3,3+\varepsilon)$ contains a bi-Lipschitz copy of 
\begin{align*}
    K_{\mathrm{mod}}&:=\left\{[0;1^{s_1},2,2,1^{s_2},2,2,\dots]:s_i\geq n+1 \text{ for all } i\geq 1\right\} \\
    &=\left\{[0;1^{n+1},\gamma_1,\gamma_2,\dots]:\gamma_i\in\{221^{n+1},1\} \text{ for all } i\geq 1\right\}.
\end{align*}
where $n\approx e^{c_0}|\log\varepsilon| = |\log\varepsilon|/\log(\varphi^2)$ is an odd integer. Since one has that (again by \Cref{lem:hausdorff_dimension_gauss_cantor})
\begin{equation*}
    \dim_H(K_{\mathrm{mod}})=\frac{W(n+1)}{(n+1)e^{-c_0}}+O\left(\frac{\log n}{n^2}\right),
\end{equation*}
this will give the lower bound of \Cref{thm:dimension} thanks to \eqref{eq:main_terms_difference_error}.

Now we want to study the first step in the construction of $K_{\mathrm{big}}\times K_{\mathrm{big}}=K\times K$.

\subsection{First step}

\begin{proposition}\label{prop:31}
Let $n\in\NN$ be large. Suppose $n\leq e_1,e_2,e_{-1},e_{-2}\leq n+\lfloor n/\sqrt{\log n}\rfloor$, $n-\lfloor 128(\log n)^2\rfloor \leq f_1,f_2,f_{-1},f_{-2}$ are positive integers with $e_1\neq e_{-1}$ and $f_1\neq f_{-1}$. Suppose that
\begin{align*}
    U&:=[0;1^{e_1},2,2,1^{e_2},2,2,\dots]-[0;1^{e_{-1}},2,2,1^{e_{-2}},2,2,\dots]>0 \\
    V&:=[0;1^{f_1},2,2,1^{f_2},2,2,\dots]-[0;1^{f_{-1}},2,2,1^{f_{-2}},2,2,\dots]>0.
\end{align*}

If
\begin{equation*}
    \big|U-V\big|<\frac{2(3\varphi-4)}{3\varphi^7}\cdot\varphi^{-2\max\{e_1,e_{-1}\}},
\end{equation*}
then we must have $e_1=f_1$ and $e_{-1}=f_{-1}$.

\end{proposition}

\begin{proof}
We will use \Cref{prop:base_case}. Without loss of generality assume that $e_1<e_{-1}, f_1<f_{-1}$. From that proposition we have
\begin{align*}
    U&=[0;1^{e_1},2,2,1^{e_2},2,2,\dots]-[0;1^{e_{-1}},2,2,1^{e_{-2}},2,2,\dots] \\
    &=(1+o(1))\frac{2(3\varphi-4)}{3\varphi^4}\left(\frac{1}{\varphi^{2(e_1-1)}}+\frac{(-1)^{e_1-e_{-1}+1}}{\varphi^{2(e_{-1}-1)}}\right)
\end{align*}
and
\begin{align*}
    V&=[0;1^{f_1},2,2,1^{f_2},2,2,\dots]-[0;1^{f_{-1}},2,2,1^{f_{-2}},2,2,\dots] \\
    &=(1+o(1))\frac{2(3\varphi-4)}{3\varphi^4}\left(\frac{1}{\varphi^{2(f_1-1)}}+\frac{(-1)^{f_1-f_{-1}+1}}{\varphi^{2(f_{-1}-1)}}\right)
\end{align*}
Let $g_1:=\min\{e_1,f_1\}$. If $e_1\neq f_1$ then we will have that
\begin{align*}
    \varphi^{-2e_{-1}}\cdot\frac{2(3\varphi-4)}{3\varphi^7}>|U-V| &\geq (1+o(1))\frac{2(3\varphi-4)}{3\varphi^4}\left(\varphi^{-2(g_1-1)}-\varphi^{-2(g_1-1)-2}-\varphi^{-2g_1}-\varphi^{-2g_1-2}\right) \\
    &=(1+o(1))\frac{2(3\varphi-4)}{3\varphi^4}\varphi^{-2g_1-2}(\varphi^4-2\varphi^2-1) \\
    &=(1+o(1))\frac{2(3\varphi-4)}{3\varphi^4}\varphi^{-2g_1-3}
\end{align*}
which is a contradiction since $e_1\geq g_1$. Similarly, since we now have $e_1=f_1$, we can write
\begin{equation*}
    |U-V|=|[0;1^{e_{-1}},2,2,\dots]-[0;1^{f_{-1}},2,2,\dots]|+O(\varphi^{-3n}).
\end{equation*}
So, setting $g_{-1}:=\min\{e_{-1},f_{-1}\}$, if $e_{-1}\neq f_{-1}$, using once more \Cref{prop:base_case} we will obtain
\begin{align*}
   \varphi^{-2e_{-1}}\cdot \frac{2(3\varphi-4)}{3\varphi^7}+O(\varphi^{-3n}) &> |[0;1^{e_{-1}},2,2,\dots]-[0;1^{f_{-1}},2,2,\dots]| \\
    &= (1+o(1))\frac{2(3\varphi-4)}{3\varphi^4}\frac{1}{\varphi^{2(g_{-1}-1)}}\left|1-(-\varphi^2)^{-|e_{-1}-f_{-1}|}\right| \\
    & \geq (1+o(1))\frac{2(3\varphi-4)}{3\varphi^5}\varphi^{-2g_{-1}+2}
\end{align*}
which again yields a contradiction, since $n+\lfloor n/\sqrt{\log n}\rfloor \geq e_{-1}\geq g_{-1}\geq n-\lfloor 128(\log n)^2\rfloor$.

\end{proof}

This proposition is illustrated in Figure \ref{fig:1} below.

\begin{figure}
\centering
\includegraphics[scale=0.8]{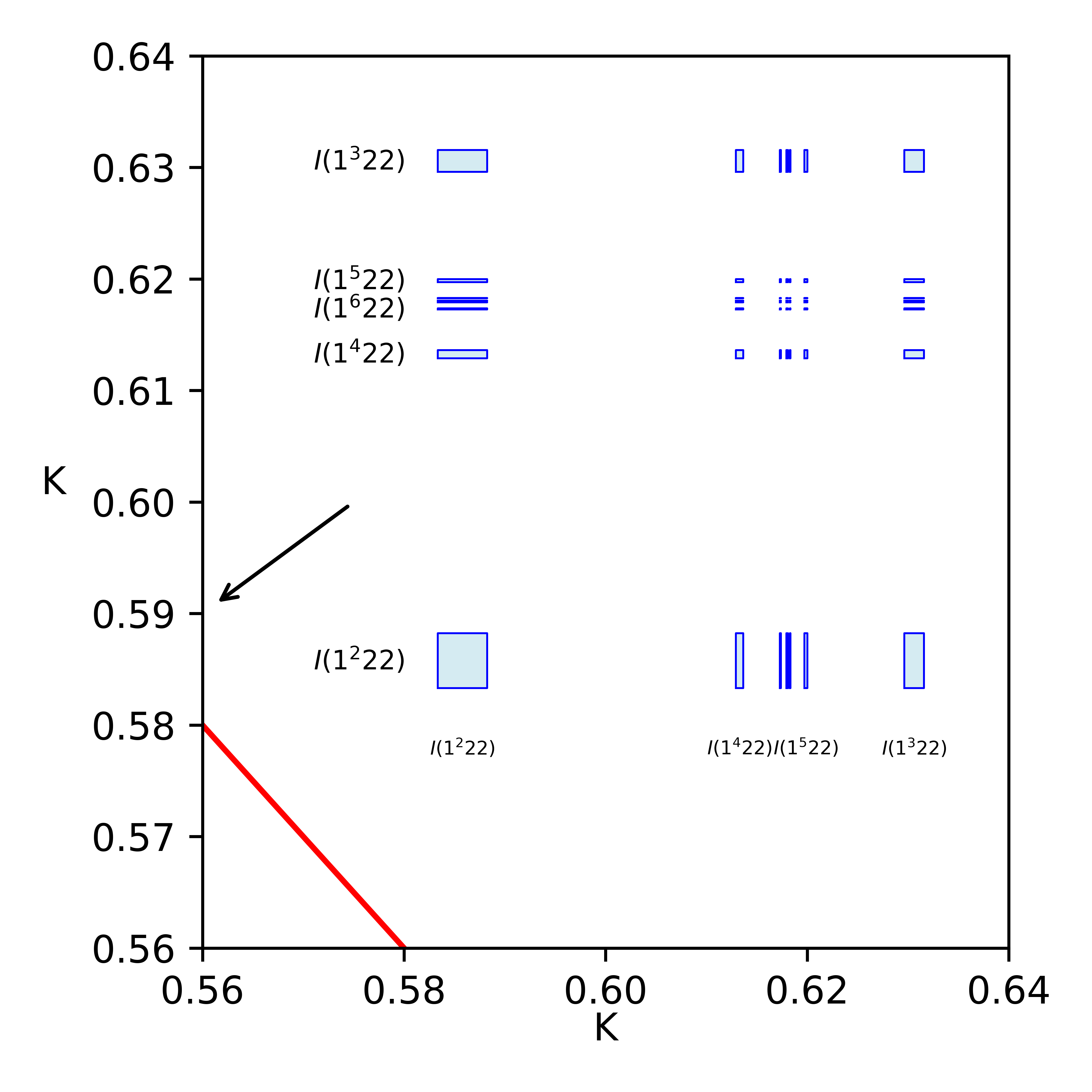}
\caption{The product $K\times K$ projects into $K-K$ by the function $(x,y)\mapsto x-y$. Each rectangle $I(1^{e_1}221^{n_0})\times I(1^{e_{-1}}221^{n_0})$ has a disjoint projection from the others.}\label{fig:1}
\end{figure}

\subsection{Second step}

Before working out the general case, we will outline the strategy of the induction from the very first step. Let $n_0 := n - \lfloor 128(\log n)^2\rfloor$. Now that we have proven that all rectangles $I(1^{e_1}221^{n_0})\times I(1^{e_{-1}}221^{n_0})$  have disjoint projections, we can focus inside each rectangle. The key observation is that inside of each such rectangle, the projection problem is analogous to the problem in the first step, but now one deals with the projection onto a difference of the form $K-\mu K$ for some $\mu\in\RR$. The appearance of this factor $\mu$ is because of the bounded distortion property \eqref{eq:bounded_distortion} of the Gauss map $g$. 

Indeed, note that the next subdivision is given by rectangles of the form $I(1^{e_1}221^{e_2}221^{n_0})\times I(1^{e_{-1}}221^{e_{-2}}221^{n_0})$ and these rectangles are images of the rectangles $I(1^{e_2}221^{n_0})\times I(1^{e_{-2}}221^{n_0})$ by inverse branches of products of the Gauss map $g^{e_1+2}\times g^{e_{-1}+2}$, more precisely by the map $(x,y)\mapsto ([0;1^{e_1},2,2,x],[0;1^{e_{-1}},2,2,y])$. Since we already know that $I(1^{e_2}221^{n_0})\times I(1^{e_{-2}}221^{n_0})$ project disjointly, we expect that their images do as well in view of the bounded distortion property. More precisely, using \eqref{eq:bounded_distortion} we have, assuming that $m\geq n-\lfloor 128(\log n)^2\rfloor$, that for any $x,y\in I(1^{m}22)\cap K_{\mathrm{big}}$ one has

\begin{equation}\label{eq:156}
    e^{-4\cdot\diam(K_{\mathrm{big}})}\frac{|x-y|}{|I(1^{m}22)|}\leq \frac{|g^{m+2}(x)-g^{m+2}(y)|}{\diam(K_{\mathrm{big}})}\leq e^{4\cdot\diam(K_{\mathrm{big}})}\frac{|x-y|}{|I(1^{m}22)|}.
\end{equation}
Suppose there is a point $p_1-p_2=q_1-q_2$ in the projection of both rectangles $(p_1,p_2)\in I(1^{e_1}221^{e_2}221^{n_0})\times I(1^{e_{-1}}221^{e_{-2}}221^{n_0})$ and $(q_1,q_2)\in I(1^{e_1}221^{f_2}221^{n_0})\times I(1^{e_{-1}}221^{f_{-2}}221^{n_0})$. Since $p_1,q_1$ are in the interval $I(1^{e_1}22)$, we can apply \eqref{eq:156} to get
\begin{equation*}
    p_1-q_1=(-1)^{e_1}e^{D^\prime}\frac{|I(1^{e_1}22)|}{\diam K_{\mathrm{big}}}(g^{e_1+2}(p_1)-g^{e_1+2}(q_1)),
\end{equation*}
where $|D^\prime|\leq 4\cdot\diam(K_{\mathrm{big}})$. Similarly one has 
\begin{equation*}
    p_2-q_2=(-1)^{e_{-1}} e^{D^{\prime\prime}}\frac{|I(1^{e_{-1}}22)|}{\diam K_{\mathrm{big}}}(g^{e_{-1}+2}(p_2)-g^{e_{-1}+2}(q_2)),
\end{equation*}
where $|D^{\prime\prime}|\leq 4\cdot\diam(K_{\mathrm{big}})$. From the equality $p_1-q_1=p_2-q_2$, dividing both sides by the same constant, one arrives to 
\begin{equation}\label{eq:456}
    g^{e_1+2}(p_1)-g^{e_1+2}(q_1)=(-1)^{e_1+e_{-1}}\cdot e^{D}\mu(g^{e_{-1}+2}(p_2)-g^{e_{-1}+2}(q_2)),
\end{equation}
where $\mu=|I(1^{e_{-1}}22)|/|I(1^{e_1}22)|$ and $|D|\leq8\cdot\diam(K_{\mathrm{big}})$.

On the other hand using \Cref{prop:base_case} one has that if $k\geq n-\lfloor 128(\log n)^2\rfloor$ then
\begin{align*}
    [0;1^k,2,2,1^{n-\lfloor 128(\log n)^2\rfloor},\dots] &= [0;1^k,2,2,\overline{1}]+O(\varphi^{-4n+1000(\log n)^2}) \\
    &=\frac{1}{\varphi}+\frac{2(3\varphi-4)}{3\varphi^2}\frac{(-1)^{k+1}}{\varphi^{2k}}+O(\varphi^{-4n+1000(\log n)^2}).
\end{align*}
In particular
\begin{equation*}
    g^{e_1+2}(p_1)-g^{e_1+2}(q_1)=-\frac{2(3\varphi-4)}{3\varphi^2}\left((-\varphi^2)^{-e_2}-(-\varphi^2)^{-f_2}\right)+O(\varphi^{-4n+1000(\log n)^2}),
\end{equation*}
and similarly for $g^{e_{-1}+2}(p_2)-g^{e_{-1}+2}(q_2)$. Moreover, it also implies that $\diam(K_{\mathrm{big}})=O(\varphi^{-3n})$. Finally, using that $e^{D}=1+O(\varphi^{-3n})=O(1)$, from \eqref{eq:456} we get an exponential Diophantine equation 
\begin{equation*}
    (-\varphi^2)^{-e_2}-(-\varphi^2)^{-f_2}-(-1)^{e_1+e_{-1}}\mu\left((-\varphi^2)^{-e_{-2}}-(-\varphi^2)^{-f_{-2}}\right)=O\left(\max\{1,\mu\}\cdot\varphi^{-4n+1000(\log n)^2}\right).
\end{equation*}

From the choices of our parameters, we will see in \Cref{lem:exponential_equation} below that this equation yields $e_2=f_2$ and $e_{-2}=f_{-2}$.

\subsection{Inductive step}

Suppose that $1\leq N,M\leq 2(\log n)^2$, and that we have a rectangle of the form 
\begin{equation}\label{eq:rectangle}
    I(1^{e_1}22\dots 1^{e_N}22)\times I(1^{e_{-1}}22\dots 1^{e_{-M}}22)
\end{equation}
that has a disjoint projection from any other rectangle of this step of the construction. 

Let $\mu=|I(1^{e_{-1}}22\dots 1^{e_{-M}}22)|/|I(1^{e_1}22\dots 1^{e_N}22)|$ be the quotient of the sides. We consider two kind of rectangles, the \textit{typical} rectangles and the \textit{exceptional} rectangles. We define typical rectangles as the ones that satisfy 
\begin{equation}\label{eq:typical_rectangle}
    \varphi^{-n-2n/\sqrt{\log n}}<\mu<\varphi^{n+2n/\sqrt{\log n}}.
\end{equation}
Exceptional rectangles, by definition, are the ones that satisfy
\begin{equation}\label{eq:exceptional_rectangle}
    n+2n/\sqrt{\log n}\leq\left|\frac{\log\mu}{\log\varphi}\right|<n+5n/\sqrt{\log n}.
\end{equation}

Note that in any case (typical or exceptional), all such rectangles have limited distortion
\begin{equation}\label{eq:mu_inequality}
    \varphi^{-n-5n/\sqrt{\log n}}<\mu<\varphi^{n+5n/\sqrt{\log n}}.
\end{equation}

We will say that a rectangle is \textit{distorted} if
\begin{equation}\label{eq:dis_rect}
    (\log\log n)^4<\left|\frac{\log\mu}{\log\varphi}\right|.
\end{equation}

Note that exceptional rectangles are clearly distorted.

We will construct our rectangles in such a way that at each step the rectangles are either typical or exceptional. At each step of the construction, we will either:
\begin{itemize}
\item Refine both sides of the rectangle if the initial rectangle is typical.
\item Refine just one of the sides in case the rectangle is exceptional. In this situation we refine the larger side. More precisely, if $\mu>1$ we subdivide $I(1^{e_{-1}}22\dots 1^{e_{-M}}22)$ while if $\mu<1$ we subdivide $I(1^{e_1}22\dots 1^{e_N}22)$.
\end{itemize}

We claim that any exceptional rectangle is always preceded and followed by typical rectangles. Since we will begin our construction with a typical rectangle, it suffices to prove:

\begin{lemma}
If a typical rectangle $R$ is subdivided into $R'$ which is not typical, then $R'$ is exceptional and the next subdivision of $R'$ will be typical and distorted.
\end{lemma}
\begin{proof}
Given a typical rectangle \eqref{eq:rectangle} satisfying \eqref{eq:typical_rectangle}, note that, using \eqref{eq:intervals_length}, its subdivision satisfies 
\begin{equation*}
    \mu_{i+1}=\frac{|I(1^{e_{-1}}22\dots 1^{e_{-M}}221^{t}22)|}{|I(1^{e_1}22\dots 1^{e_N}221^{s}22)|}=e^C \mu_i\cdot \frac{|I(1^t22)|}{|I(1^s22)|}.
\end{equation*}
where $1/2\leq e^C\leq 2$. On the other hand, since $|s-t|\leq n/\sqrt{\log n}$, it follows that
\begin{equation*}
    \left|\log\varphi^{2(s-t)}\right|\leq\left|\log\frac{|I(1^t22)|}{|I(1^s22)|}\right|+20\log\varphi<(\log\varphi)(3n/\sqrt{\log n}).
\end{equation*}

Hence, if this rectangle is not typical, this is because one has 
\begin{equation}\label{eq:213}
    n+2n/\sqrt{\log n}\leq\left|\frac{\log \mu_{i+1}}{\log\varphi}\right|<n+5n/\sqrt{\log n}.
\end{equation} 
In particular, this rectangle is exceptional. 

In the next step we are forced to subdivide just one of the sides. Without loss of generality, assume that $\mu_{i+1}>1$. In particular the next subdivision has the form 
\begin{equation*}
    \mu_{i+2}=\frac{|I(1^{e_{-1}}22\dots 1^{e_{-M}}221^{t}22)|}{|I(1^{e_1}22\dots 1^{e_N}221^{s}221^{p}22)|}=e^{C^\prime} \mu_{i+1}\frac{1}{|I(1^{p}22)|}.
\end{equation*}
for some integer $n\leq p\leq n+n/\sqrt{\log n}$ and where $1/2\leq e^{C^\prime}\leq 2$. In particular from \eqref{eq:213} we are led to
\begin{equation*}
    -n<\frac{\log\mu_{i+2}}{\log\varphi}<-n+5n/\sqrt{\log n}.
\end{equation*}
Thus, this last rectangle is typical. Also, it is distorted since $n-5n/\sqrt{\log n}>(\log\log n)^4$.

\end{proof}

\begin{corollary}
If the first rectangle is typical, then all rectangles in the construction satisfy \eqref{eq:mu_inequality}.
\end{corollary}

Let us now analyse the division of typical rectangles. For this sake, consider $n\leq s,t\leq n+n/\sqrt{\log n}$ and rectangles of the form
\begin{equation*}
    I(1^{e_1}22\dots 1^{e_N}221^s22)\times I(1^{e_{-1}}22\dots 1^{e_{-M}}221^t22).
\end{equation*}
We want to show that this rectangle has a disjoint projection from all other rectangles of the big Cantor set, namely the rectangles where we replace $(s,t)$ for any $(u,v)$ with $n_0:=n-\lfloor 128(\log n)^2\rfloor\leq u,v$. We will only be interested in distorted rectangles, so we will assume
\begin{equation}\label{eq:distorted_rectangle_condition}
    \left|\log\left|\frac{\varphi^{-2s}}{\mu\varphi^{-2t}}\right|\right|>(\log\log n)^4/2.
\end{equation}
where $\mu=|I(1^{e_{-1}}22\dots 1^{e_{-M}}22)|/|I(1^{e_1}22\dots 1^{e_N}22)|$. In virtue of \Cref{cor:convergents_estimate} we will have that
\begin{equation*}
    q(1^{e_1}22\dots 221^{e_N})^{-2}=\frac{5}{\varphi^2}\varphi^{-2(e_1+\dots +e_{N})}\left(\frac{3\varphi^3}{\sqrt{5}}\right)^{-2N+2}\widetilde{E},
\end{equation*}
where using $x/(1+x)<\log(1+x)<x$ for all $x>-1$ yields
\begin{equation*}
    |\log \widetilde{E}| \leq 2\sum_{i=1}^N\frac{2}{3\varphi^{n}}\leq \frac{(8/3)(\log n)^2}{\varphi^{n}},
\end{equation*} 
and similarly for $q(1^{e_{-1}}22\dots 221^{e_{-M}})$. This implies that
\begin{equation}\label{eq:denominator_asymptotics}
    q(1^{e_1}22\dots 221^{e_N})^{-2}=\frac{5}{\varphi^2}\varphi^{-2(e_1+\dots +e_{N})}\left(\frac{3\varphi^3}{\sqrt{5}}\right)^{-2N+2}(1+O(\varphi^{-n+5\log\log n})).
\end{equation}
and similarly for $q(1^{e_{-1}}22\dots 221^{e_{-M}})$. Hence
\begin{equation}\label{eq:mu_asymptotics}
    \mu = \varphi^{2c}\left(\frac{3\varphi^{3}}{\sqrt{5}}\right)^d(1+O(\varphi^{-n+5\log\log n}))
\end{equation}
with $c=e_1+\dots+e_N-e_{-1}\dots-e_M$ and $d=2(N-M)=O((\log n)^2)$.

Suppose there is a point $p_1-p_2=q_1-q_2$ where
\begin{equation*}
    (p_1,p_2)\in I(1^{e_1}22\dots 1^{e_N}221^s221^{n_0})\times I(1^{e_{-1}}22\dots 1^{e_{-M}}221^t221^{n_0})
\end{equation*}
and 
\begin{equation*}
    (q_1,q_2)\in I(1^{e_1}22\dots 1^{e_N}221^u221^{n_0})\times I(1^{e_{-1}}221^{e_{-M}}221^v221^{n_0}).
\end{equation*}

By bounded distortion \eqref{eq:bounded_distortion} and by doing the same manipulations as before, we have that 
\begin{equation*}
    g^{e_1+\dots+e_N+2N}(p_1)-g^{e_1+\dots+e_N+2N}(q_1)=(-1)^{c}e^{D}\mu\left(g^{e_{-1}+\dots+e_{-M}+2M}(p_2)-g^{e_{-1}+\dots+e_{-M}+2M}(q_2)\right)
\end{equation*}
where $|D|\leq 8\diam K_{\mathrm{big}}$. As before, \Cref{prop:base_case} implies that
\begin{equation*}
    g^{e_1+\dots+e_N+2N}(p_1)-g^{e_1+\dots+e_N+2N}(q_1)=-\frac{2(3\varphi-4)}{3\varphi^2}\left((-\varphi^2)^{-s}-(-\varphi^2)^{-u}\right)+O(\varphi^{-4n+1000(\log n)^2}),
\end{equation*}
and similarly for $g^{e_{-1}+\dots+e_{-M}+2M}(p_2)-g^{e_{-1}+\dots+e_{-M}+2M}(q_2)$.

Now dividing both sides by the same constant and using that $e^{D}=1+O(\varphi^{-3n})=O(1)$, we obtain again an exponential Diophantine equation
\begin{equation}\label{eq:exponential_equation}
    (-\varphi^2)^{-s}-(-\varphi^2)^{-u}-(-1)^{c}\mu\left((-\varphi^2)^{-t}-(-\varphi^2)^{-v}\right)=O\left(\max\{1,\mu\}\cdot\varphi^{-4n+1000(\log n)^2}\right).
\end{equation}

\begin{lemma}\label{lem:exponential_equation}
Let $n\in\NN$ be large. Suppose $n\leq s,t\leq n+\lfloor n/\sqrt{\log n}\rfloor$ and $n-\lfloor 128(\log n)^2\rfloor\leq u,v$ are positive integers such that \eqref{eq:exponential_equation} holds. If \eqref{eq:distorted_rectangle_condition} and \eqref{eq:mu_inequality} are satisfied, then $s=u$ and $t=v$.
\end{lemma}

\begin{proof}
By \eqref{eq:mu_inequality} we can deduce for all large $n$ that it holds
\begin{equation*}
    |(-\varphi^2)^{-s}-(-\varphi^2)^{-u}\pm\mu((-\varphi^2)^{-t}-(-\varphi^2)^{-v})|=O\left(\varphi^{-3n+6n/\sqrt{\log n}}\right).
\end{equation*}

Without loss of generality we can assume that $\varphi^{-2s}>\mu\varphi^{-2t}$, otherwise we divide by $\mu$ at both sides of \eqref{eq:exponential_equation} and obtain an analogous equation. It follows from this assumption and our hypothesis \eqref{eq:distorted_rectangle_condition} of distortion of the rectangle that $\varphi^{-2s}\pm \mu\varphi^{-2t}=(1+O(e^{-(\log\log n)^4/2}))\varphi^{-2s}$, so 

\begin{equation*}
    (1+O(e^{-(\log\log n)^4/2}))\varphi^{-2s}=|(-\varphi^2)^{-u}\pm\mu(-\varphi^2)^{-v}|+O\left(\varphi^{-3n+6n/\sqrt{\log n}}\right).
\end{equation*}

Since $s\leq n+n/\sqrt{\log n}$, for $n$ sufficiently large we deduce that
\begin{equation}\label{eq:exponential_equation_win}
    (1+O(e^{-(\log\log n)^4/2}))\varphi^{-2s}=|(-\varphi^2)^{-u}\pm\mu(-\varphi^2)^{-v}|.
\end{equation}

\medbreak
\noindent\textbf{Case 1:} 
\begin{equation}\label{eq:distorted_rectangle_big_cantor_1}
    \left|\log \left|\frac{\varphi^{-2u}}{\mu\varphi^{-2v}}\right|\right|>(\log\log n)^2.
\end{equation}
\medbreak

If $\varphi^{-2u}>\mu\varphi^{-2v}e^{(\log\log n)^2}$ holds, then \eqref{eq:exponential_equation_win} implies that
\begin{equation*}
    1+O(e^{-(\log\log n)^4/2})=\varphi^{2s}\left|\varphi^{-2u}\pm\mu\varphi^{-2v}\right|=\left(1+O(e^{-(\log\log n)^2})\right)\varphi^{2s-2u}
\end{equation*}
and $s=u$. In the other situation, we write \eqref{eq:exponential_equation_win} as
\begin{equation*}
    1+O(e^{-(\log\log n)^4})=\left(1+O(e^{-(\log\log n)^2})\right)\mu\varphi^{2s-2v}.
\end{equation*}
Using \eqref{eq:mu_asymptotics} one has that
\begin{equation}\label{eq:489}
    \varphi^{2c}\left(\frac{3\varphi^{3}}{\sqrt{5}}\right)^d\varphi^{2s-2v}=1+O(e^{-(\log\log n)^2}).
\end{equation}
Since the right hand side is bounded and $d=O((\log n)^2)$, we will have that $|2c+2s-2v|=O((\log n)^2)$. 

Note that $\varphi$ and $3\varphi^3/\sqrt{5}$ are multiplicatively independent elements of the number field $\QQ(\sqrt{5})$, because the norm of $\varphi$ in the extension $\QQ(\sqrt{5})/\QQ$ is equal to $1$ while the norm of $3\varphi^3/\sqrt{5}$ is equal to $9/5$, which are multiplicatively independent rational numbers. In particular $\left(\frac{3\varphi^{3}}{\sqrt{5}}\right)^d\varphi^{2c+2s-2v} = 1$ implies that $d=0$ and $2c+2s-2v=0$, which in turn gives $\mu=\varphi^{2c}(1+O(\varphi^{-n+5\log\log n}))$. For this particular case we can go back to \eqref{eq:exponential_equation} to obtain
\begin{equation*}
    \left|(-\varphi^2)^{-s}-(-\varphi^2)^{-u}-\left((-\varphi^2)^{c-t}-(-\varphi^2)^{c-v}\right)\left(1+O(\varphi^{-n+5\log\log n})\right)\right|=O\left(\varphi^{-3n+6n/\sqrt{\log n}}\right).
\end{equation*}
Now using that $c-v=-s$ and the assumption that $\varphi^{-2s}\geq\mu\varphi^{-2t}e^{(\log\log n)^4/2}=\varphi^{2c-2t}e^{(\log\log n)^4/2}\left(1+O(\varphi^{-n+5\log\log n})\right)$ (because the rectangle is distorted), we get
\begin{equation*}
    \left|2(-\varphi^2)^{-s}(1+O(e^{-(\log\log n)^4/2}))-(-\varphi^2)^{-u}\right|=O\left(\varphi^{-3n+6n/\sqrt{\log n}}\right).
\end{equation*}

Hence 
\begin{equation}\label{eq:contradiction_special_case}
    \varphi^{-1}\leq |2-(-\varphi^2)^{s-u}|=O(e^{-(\log\log n)^4/2})
\end{equation}
gives the desired contradiction when $\left(\frac{3\varphi^{3}}{\sqrt{5}}\right)^d\varphi^{2c+2s-2v} = 1$.

In case that $\left(\frac{3\varphi^{3}}{\sqrt{5}}\right)^d\varphi^{2c+2s-2v} \neq 1$, we can bound the left hand side of \eqref{eq:489} using Baker-W{\"u}stholz (\Cref{thm:baker}): 
\begin{equation*}
    e^{-C\cdot\log\log n}<\left|\left(\frac{3\varphi^{3}}{\sqrt{5}}\right)^d\varphi^{2c+2s-2v}-1\right|=O(e^{-(\log\log n)^2}),
\end{equation*}
for some large constant $C>0$, which is clearly a contradiction.

Once we know that $s=u$, going back to \eqref{eq:exponential_equation} one obtains
\begin{equation*}
    |(-\varphi^2)^{-t}-(-\varphi^2)^{-v})|=O\left(\varphi^{-3n+6n/\sqrt{\log n}}\right).
\end{equation*}
where we used \eqref{eq:mu_inequality}. Since $t\leq n+n/\sqrt{\log n}$ we deduce
\begin{equation*}
    |1-(-\varphi^2)^{t-v}|=O\left(\varphi^{-n+8n/\sqrt{\log n}}\right),
\end{equation*}
which clearly can only be true if $t=v$.

\medbreak
\noindent\textbf{Case 2:} 
\begin{equation}\label{eq:distorted_rectangle_big_cantor_2}
    \left|\log \left|\frac{\varphi^{-2u}}{\mu\varphi^{-2v}}\right|\right|\leq (\log\log n)^2.
\end{equation}
\medbreak

We can rewrite \eqref{eq:exponential_equation_win} as
\begin{equation}\label{eq:739}
    1+O(e^{-(\log\log n)^4/2})=(-1)^{s-u}\varphi^{2s-2u}-(-1)^{s-v}\mu\varphi^{2s-2v}=(-\varphi^2)^y-\mu(-\varphi^2)^x
\end{equation}
with $y=s-u$ and $x=s-v$. 

Note that $y=0$ is equivalent to $s=u$ and the same manipulation before yields $t=v$. Now we will show that the situation $y\neq 0$ does not occur. Note that this implies that $|\mu\varphi^{2x}|>\frac{1}{2}|1\pm\varphi^{2y}|\geq\frac{1}{2}(1-\varphi^{-2})=\frac{1}{2}\varphi^{-1}$. 

Observe that \eqref{eq:mu_asymptotics} together with $\varphi^{-n-5n/\sqrt{\log n}}<\mu<\varphi^{n+5n/\sqrt{\log n}}$ gives that for large $n$
\begin{equation*}
    \frac{1}{2}\mu<\varphi^{2c}\left(\frac{3\varphi^{3}}{\sqrt{5}}\right)^d<2\mu.
\end{equation*}

Furthermore, the estimate \eqref{eq:mu_asymptotics} yields
\begin{equation}\label{eq:subcases_large_y}
    (-\varphi^2)^y-\varphi^{2c}\left(\frac{3\varphi^{3}}{\sqrt{5}}\right)^d(-\varphi^2)^x=1+O(e^{-(\log\log n)^4/2})+\varphi^{2c}\left(\frac{3\varphi^{3}}{\sqrt{5}}\right)^d\varphi^{2x}O(\varphi^{-n+5\log\log n}).
\end{equation}

\medbreak
\noindent\textbf{Subcase 2.1:} 
Suppose $|y|>(\log\log n)^2$.
\medbreak

By \eqref{eq:mu_asymptotics} and \eqref{eq:distorted_rectangle_big_cantor_2} we have for large $n$
\begin{equation*}
    \left|\log\left|\left(\frac{3\varphi^{3}}{\sqrt{5}}\right)^d\varphi^{2c+2x-2y}\right|\right| = \left|\log \left|\frac{\mu\varphi^{-2v}}{\varphi^{-2u}}\right|+O(\varphi^{-n+5\log\log n})\right|\leq2(\log\log n)^2
\end{equation*}
and since we already know that $d=O((\log n)^2)$, it follows that $|2c+2x-2y|=O((\log n)^2)$ as well. If $y>(\log\log n)^2$, then we divide both sides of \eqref{eq:subcases_large_y} by $\varphi^{2y}$ and use that $\varphi^{-n-5n/\sqrt{\log n}}<\mu<\varphi^{n+5n/\sqrt{\log n}}$ to obtain
\begin{equation*}
    \left|1\pm\varphi^{2c}\left(\frac{3\varphi^{3}}{\sqrt{5}}\right)^d\varphi^{2x-2y}\right|=\frac{1+O(e^{-(\log\log n)^4/2})}{\varphi^{2y}}+O\left(\frac{\mu\varphi^{2x}}{\varphi^{2y}}\cdot\varphi^{-n+5\log\log n}\right)=O(\varphi^{-(\log\log n)^2}).
\end{equation*}

Note that $\varphi^{2c+2x-2y}\left(\frac{3\varphi^{3}}{\sqrt{5}}\right)^d=1$ implies that $c+x-y=0$ and $d=0$, hence from \eqref{eq:mu_asymptotics} one gets $\mu\varphi^{2x}=\varphi^{2y}(1+O(\varphi^{-n+5\log\log n}))=\varphi^{2y}+O(\varphi^{-n+3n/\sqrt{\log n}})$ (since $y=s-u \leq n+n/\sqrt{\log n}$) so replacing in \eqref{eq:739} one would obtain $1+O(e^{-(\log\log n)^4/2})=\varphi^{2y}((-1)^y-(-1)^x)$ which clearly is a contradiction.

Finally, when $\varphi^{2c+2x-2y}\left(\frac{3\varphi^{3}}{\sqrt{5}}\right)^d\neq 1$ we bound the left hand side by Baker-W{\"u}stholz (\Cref{thm:baker}) 
\begin{equation*}
    e^{-C\cdot\log\log n}<\left|\varphi^{2c}\left(\frac{3\varphi^{3}}{\sqrt{5}}\right)^d\varphi^{2x-2y}\pm 1\right|=O(\varphi^{-(\log\log n)^2}).
\end{equation*}
This is a contradiction for large $n$. In case that $y<-(\log\log n)^2$, we write \eqref{eq:subcases_large_y} as
\begin{equation*}
    \left|1\pm\varphi^{2c}\left(\frac{3\varphi^{3}}{\sqrt{5}}\right)^d\varphi^{2x}\right|=O(e^{-(\log\log n)^2})+O\left(\frac{\mu\varphi^{2x}}{\varphi^{2y}}\cdot\varphi^{-n+5\log\log n}\right)=O(\varphi^{-(\log\log n)^2}).
\end{equation*}
Again we can bound the left hand side by Baker-W{\"u}stholz (\Cref{thm:baker})
\begin{equation*}
    e^{-C\cdot\log\log n}<\left|1\pm\varphi^{2c}\left(\frac{3\varphi^{3}}{\sqrt{5}}\right)^d\varphi^{2x}\right|=O(\varphi^{-(\log\log n)^2})
\end{equation*} 
whenever $\varphi^{2c}\left(\frac{3\varphi^{3}}{\sqrt{5}}\right)^d\varphi^{2x}\neq 1$, and we can handle $\varphi^{2c}\left(\frac{3\varphi^{3}}{\sqrt{5}}\right)^d\varphi^{2x} = 1$, i.e., $c+x=0=d$ separately by noticing that $c+x=c+s-v=0$ again yields \eqref{eq:contradiction_special_case}. In any event, we get a contradiction.

\medbreak
\noindent\textbf{Subcase 2.2:} 
Suppose $0<|y|\leq(\log\log n)^2$.
\medbreak

We will write \eqref{eq:subcases_large_y} as
\begin{equation*}
    (\varphi^{2y}\pm 1)\varphi^{-2c}\left(\frac{3\varphi^{3}}{\sqrt{5}}\right)^{-d}\varphi^{-2x}\pm 1=O(e^{-(\log\log n)^4/2})+O(\varphi^{-n+5\log\log n})=O(e^{-(\log\log n)^4/2}),
\end{equation*}
where we used that $|\mu\varphi^{2x}|>\frac{1}{2}\varphi^{-1}$.

On the other hand, by \eqref{eq:distorted_rectangle_big_cantor_2} we have for large $n$
\begin{equation*}
    \left|\log\left|\varphi^{-2c}\left(\frac{3\varphi^{3}}{\sqrt{5}}\right)^{-d}\varphi^{2y-2x}\right|\right| = \left|\log \left|\frac{\varphi^{-2u}}{\mu\varphi^{-2v}}\right|+O(\varphi^{-n+5\log\log n})\right|\leq2(\log\log n)^2.
\end{equation*}
Since we already know that $d=O((\log n)^2)$ and we are assuming that $|y|\leq(\log\log n)^2$,  it follows that $|2x+2c|=O((\log n)^2)$.

Note that if $a,b$ are positive integers such that $(\varphi^{2y}\pm 1)\varphi^{a}\left(\frac{3\varphi^{3}}{\sqrt{5}}\right)^{b}=1$, then one must have that $b=0$ (since $(\varphi^{2y}\pm 1)\varphi^{a}$ is an algebraic integer), however $(\varphi^{2y}\pm 1)\varphi^{a}=1$ can not hold because $\varphi^{2y}\pm 1$ does not have norm equal to $\pm1$ in the number field $\QQ(\sqrt{5})$, unless $y=0$. Thus, $(\varphi^{2y}\pm 1)\varphi^{a}\left(\frac{3\varphi^{3}}{\sqrt{5}}\right)^{b}\neq 1$ and we can apply again Baker-W{\"u}stholz (\Cref{thm:baker}) to obtain that 
\begin{equation*}
    \left|(\varphi^{2y}\pm 1)\varphi^{-2c}\left(\frac{3\varphi^{3}}{\sqrt{5}}\right)^{-d}\varphi^{-2x}\pm 1\right|\geq e^{-C\cdot h(\varphi^{2y}\pm 1)\cdot\log\log n}
\end{equation*}
for some large constant $C>0$. Since the logarithmic height satisfies $h(\varphi^{2y}\pm 1)\leq\log(\varphi^{|2y|}\pm 1)\leq |2y|+1\leq 3(\log\log n)^2$, we conclude that
\begin{equation*}
    e^{-C\cdot(\log\log n)^3}<\left|(\varphi^{2y}\pm 1)\varphi^{-2c}\left(\frac{3\varphi^{3}}{\sqrt{5}}\right)^{-d}\varphi^{-2x}-1\right|=O(e^{-(\log\log n)^4/2})
\end{equation*}
which is a contradiction for large $n$.

\end{proof}

\section{Local uniqueness and self-replication}\label{sec:construction_of_words}

\subsection{Construction of non semisymmetric words}

We will take $r\in\NN$ large. We chose $n=2k-1\in\NN$ to be the least odd integer that satisfies $\sizer(1^n)\geq r$. Note that $n$ is large if and only if $r$ is large. Since $\lfloor (2m-1)\log\varphi\rfloor\leq\sizer(1^m)\leq \lfloor (2m+1)\log\varphi\rfloor$ for any $m\geq 1$, we have that $n\in((r-1)/\log(\varphi^2),(r+2)/\log(\varphi^2))$. As consequence using \eqref{eq:sizer_length}
\begin{equation*}
    \sizer(1^{n-\lfloor 128(\log n)^2\rfloor})\leq \sizer(1^{n-2})-\sizer(1^{\lfloor 128(\log n)^2\rfloor-2})+1 < r - 64(\log r)^2.
\end{equation*}

We will always work with sequences $\underline{c}=(c_n)_{n\in\ZZ}\in\Sigma_{3+e^{-r}}$ with $m(\underline{c})=\lambda_0(\underline{c})$. The previous argument implies that if $\underline{c}=\dots221^{e_{-2}}221^{e_{-1}}22^*1^{e_1}221^{e_2}22\dots$ is such that a central block $u^*=c_{-n_1}\dots c_0^*\dots c_{n_2}$ of size $\sizer(u)\leq r$ does contain a subblocks of the form $\sizer(1^s)\geq r-64(\log r)^2$, then this central block $u^*$ determines a rectangle $I(1^{e_1}221^{e_2}22\dots)\times I(1^{e_{-1}+2}221^{e_{-2}}22\dots)$ of the construction of the construction of $K_{\mathrm{big}}\times K_{\mathrm{big}}$.

Before applying the results from the previous section, we will fix some parameters first. We choose  
\begin{equation}\label{eq:conditions_on_w}
    s_1=2k-1, \qquad s_{-1}>s_2\geq 2k+1 ~\text{ both odd}, \qquad s_{-1}=2k+\lceil(\log\log (2k))^4\rceil.
\end{equation}
From \Cref{lem:cut_comparison}, we know that the only critical position\footnote{In what follows, a \emph{critical position} is a position where the Markov value can potentially be attained and, for this reason, such positions will sometimes also be called dangerous.} of $\dots 1^{s_{-1}}22^*1^{2k-1}221^{s_2}\dots$ is precisely the marked with $*$ (as $s_{-1}>s_2\geq n+1=2k$ are both odd). Moreover, if $m_i=\min\{s_i,s_{i-1}\}\geq n+1$ then \Cref{lem:calc_s} implies that 
\begin{align*}
    \lambda_0(\dots1^{s_{-1}}22^*1^{2k-1}221^{s_2}\dots)>3+\sizes(1^{2k+1}) \geq 3+\sizes(1^{m_i+1}) >\lambda_0(\dots 1^{s_{i-1}}22^*1^{s_{i}}\dots).
\end{align*}

Now that we have fixed $s_1$ and the conditions on $s_{-1},s_2$, we will proceed with the construction of an appropriate list of words. We saw in the previous section that sequences $\underline{c}=(c_n)_{n\in\ZZ}\in\Sigma_{3+e^{-r}}$, $m(\underline{c})=\lambda_0(\underline{c})$ of the form $\underline{c}=\dots221^{s_{-2}}221^{s_{-1}}22^*1^{s_1}221^{s_2}22\dots$ where $n\leq s_{i}\leq n+\lfloor n/\sqrt{\log n}\rfloor$ have nearby Markov value (in the sense of Proposition \ref{prop:31}) if and only if their respective exponents $s_1,s_{-1},s_2,s_{-2},\dots$ are equal. Therefore, we are naturally led to consider words of the form
\begin{equation}\label{eq:w}
    w^*=221^{s_{-M}}\dots 221^{s_{-1}}22^*1^{2k-1}221^{s_2}\dots 221^{s_N}.
\end{equation}

From \Cref{lem:cut_comparison} and the previous argument about critical positions, we know that $\overline{w}$ attains it Markov value at only one position per period. In fact, \Cref{lem:calc_s} implies that

\begin{lemma}\label{lem:markov_value_w}
Suppose $s_{-M},\dots,s_N\geq 2k$, $s_1=2k-1$ are positive integers with $s_{-1}>s_2$ both odd. For any $w$ of the form \eqref{eq:w} we have that $m(\overline{w})$ is attained only at the position marked with $*$ and moreover
\begin{equation*}
    \sizes(1^{2k+1}) < m(\overline{w}) - 3 < \sizes(1^{2k}).
\end{equation*}
\end{lemma}

Another simple observation is that our choices imply that these words are already not semisymmetric.

\begin{lemma}
Suppose $s_{-M},\dots,s_N\geq 2k$, $s_1=2k-1$ are positive integers with $s_{-1}\neq s_2$. Then $w$ defined on \eqref{eq:w} is not semisymmetric.
\end{lemma}

On the other side, we also want the size of $w$ to be large enough so that the words of the form \eqref{eq:w} form the majority of the combinatorics in $\Sigma^{(r-4)}(3+e^{-r})$. In this direction, we want the size $\sizer(w)$ of $w$ to be roughly $r\lfloor(\log r)^2\rfloor$. Recall that $\sizer(w)=\lfloor\log\sizes(w)^{-1}\rfloor$ and that $\sizes(w)$ is very similar to the product of the sides of the rectangle determined by $w^*$. We will assume that the rectangle that contains $w$ is minimal in the following sense: any refinement of the rectangle $w^*$ gives a word with size larger than $2r\lfloor(\log r)^2\rfloor$. In particular for large $r$ this gives the following estimate:
\begin{equation}\label{eq:sizer_of_notsemisymmetric_w}
    2r\lfloor(\log r)^2\rfloor-1.1r-1.1r/\sqrt{\log r}<\sizer(w)\leq 2r\lfloor(\log r)^2\rfloor.
\end{equation}

\subsection{Counting non semisymmetric words from the construction}\label{sec:counting_words}
We want to see that the majority of the words $w$ of the form \eqref{eq:w} have the \emph{extra} condition that each step of its construction is associated to a distorted rectangle: this will be important when applying Lemma \ref{lem:exponential_equation} to derive the local uniqueness property for most of these words. Let us simplify the notation by denoting
\begin{equation*}
    w^*=\eta_1\dots\eta_h^*\dots\eta_\ell
\end{equation*}
where each $\eta_i$ is of the form $221^{t_i}$ and $\eta_h^*=22^*1^{2k-1}$. By hypothesis, there are positive integers $1= k_{-L}\leq\dots\leq k_{-1}=h-1<k_1=h\leq \dots \leq k_L=\ell$ such that $\eta_{k_{-i}}\dots\eta_h^*\dots\eta_{k_i}22$ corresponds to a rectangle in the construction, namely the rectangle $R_i=I(\eta_{k_1}^T\eta_{k_2}^T\dots\eta_{k_i}^T)\times I(11\eta_{k_{-1}}^T\dots\eta_{k_{-i}}^T)$. 

We want to count words $w$ that satisfy \eqref{eq:sizer_of_notsemisymmetric_w} and such that each rectangle $R_i$ of its construction is distorted in the sense of \eqref{eq:dis_rect}. To count such words $w$, we will use \eqref{eq:d-measurecover} as we will see below. 

Remember that we already imposed the conditions \eqref{eq:conditions_on_w} on the word $w$. First, let us count how many choices we have for $(s_{-2},s_{-1},s_1,s_2)$ satisfying \eqref{eq:conditions_on_w} and such that the rectangle $I(1^{s_1}221^{s_2}22)\times I(1^{s_{-1}+2}221^{s_{-2}}22)$ is distorted. By \eqref{eq:mu_asymptotics} we have that 
\begin{equation*}
    |I(1^{s_{-1}+2}221^{s_{-2}}22)|/|I(1^{s_1}221^{s_2}22)|=\varphi^{s_2-s_{-2}-\lfloor(\log\log (n-1))^4\rfloor-2}(1+O(\varphi^{-n+5\log\log n})),
\end{equation*}
and since we want this rectangle to be distorted, it suffices to impose $|s_2-s_{-2}|>3(\log\log n)^4$. In particular since $s_{2}$ is odd we have at least $\frac{1}{2}(n/\sqrt{\log n}-3)(n/\sqrt{\log n}-6(\log\log n)^4)=\frac{1}{2}\frac{n^2}{\log n}(1+o(1))$ choices for $(s_{2},s_{-2})$.

Once we have fixed $(s_{-2},s_{-1},s_1,s_2)$, the collection of all rectangles satisfying \eqref{eq:sizer_of_notsemisymmetric_w} forms a covering of the set 
\begin{equation*}
    \left(K_{\mathrm{small}}\times K_{\mathrm{small}}\right)\cap \left(I(1^{s_1}221^{s_2}22)\times I(1^{s_{-1}+2}221^{s_{-2}}22)\right).
\end{equation*}

From \eqref{eq:d-measurecover}, we have that for each $3\leq i\leq L$
\begin{equation*}
    |I(\eta_{k_1}^T\dots\eta_{k_i}^T)|^d\cdot |I(11\eta_{k_{-1}}^T\dots\eta_{k_{-i}}^T)|^d\leq\sum_{s,t=n}^{n+\lfloor n/\sqrt{\log n}\rfloor}|I(\eta_{k_1}^T\dots\eta_{k_i}^T1^s22)|^d\cdot |I(11\eta_{k_{-1}}^T\dots\eta_{k_{-i}}^T1^t22)|^d,
\end{equation*}
where $d$ is the lower bound that appears in \eqref{eq:Palis-Takens}. In particular
\begin{equation*}
    d=\frac{W(n)+O\left(e^{-\frac{1}{2}\sqrt{\log n}}\right)}{ne^{-c_0}}=\frac{W(re^{c_0})+O\left(e^{-\frac{1}{2}\sqrt{\log r}}\right)}{r}
\end{equation*}
where we used \eqref{eq:main_terms_difference_error} and $c_0=-\log\log(\varphi^2)=0.0383\dots$. 

To count how many words have the property we want, we will control instead $d$--sums over the words $w$ that have one step \textit{not} distorted and use the previous inequality to bound by below the number of words $w$ that are always distorted.

In our subdivision of rectangles, each rectangle $R_i$ is either divided at only one side or it is divided at both sides in the rectangles
\begin{equation*}
    I(\eta_{k_1}^T\dots\eta_{k_i}^T1^{s}22)\times I(11\eta_{k_{-1}}^T\dots\eta_{k_{-i}}^T1^{t}22)
\end{equation*}
for all choices of $n\leq s,t\leq n+\lfloor n/\sqrt{\log n}\rfloor$. Denoting 
\begin{equation*}
    \mu=|I(11\eta_{k_{-1}}^T\dots\eta_{k_{-i}}^T1^{t}22)|/|I(\eta_{k_1}^T\dots\eta_{k_i}^T1^{s}22)|,
\end{equation*}
we want to estimate the $d$--sum of interval sizes for all not distorted choices of $(s,t)$. Note that the pairs $(s,t)$ such that $\mu$ does not satisfy \eqref{eq:dis_rect} must verify $|s-t+d_i|<2(\log\log n)^4$, where $d_i$ depends only in the previous quotient of sides of $R_i$. The sum over such pairs $(s,t)$ can be estimated by

\begin{align*}
    &\sum_{(s,t) \text{ not distorted}}|I(\eta_{k_1}^T\dots\eta_{k_i}^T1^{s}22)|^d\cdot |I(11\eta_{k_{-1}}^T\dots\eta_{k_{-i}}^T1^{t}22)|^d \\
    &\leq 2^{2d}|I(\eta_{k_1}^T\dots\eta_{k_i}^T)|^d\cdot |I(11\eta_{k_{-1}}^T\dots\eta_{k_{-i}}^T)|^d\sum_{|s-t+d_i|<2(\log\log n)^4}(\varphi^{-2n})^{2d} \\
    &\leq 4(n/\sqrt{\log n})(\log\log n)^4(\varphi^{-2n})^{2d}|I(\eta_{k_1}^T\dots\eta_{k_i}^T)|^d\cdot |I(11\eta_{k_{-1}}^T\dots\eta_{k_{-i}}^T)|^d \\
    &=\frac{4n(\log\log n)^4}{\sqrt{\log n}}e^{-2W(n)+O(e^{-1/2\sqrt{\log n}})}|I(\eta_{k_1}^T\dots\eta_{k_i}^T)|^d\cdot |I(11\eta_{k_{-1}}^T\dots\eta_{k_{-i}}^T)|^d \\
    &= \frac{4(\log n)^{3/2}(\log\log n)^4}{n}(1+o(1))|I(\eta_{k_1}^T\dots\eta_{k_i}^T)|^d\cdot |I(11\eta_{k_{-1}}^T\dots\eta_{k_{-i}}^T)|^d
\end{align*}

Therefore in the case the rectangle $R_i$ is divided at both sides, we have
\begin{multline}\label{eq:inductive_distorted_sum1}
    |I(\eta_{k_1}^T\dots\eta_{k_i}^T)|^d\cdot |I(11\eta_{k_{-1}}^T\dots\eta_{k_{-i}}^T)|^d\left(1-\frac{8(\log n)^{3/2}(\log\log n)^4}{n}\right)  \\
    \leq \sum_{\substack{(s,t) \text{ distorted} \\ s,t\neq n}}|I(\eta_{k_1}^T\dots\eta_{k_i}^T1^{s}22)|^d\cdot |I(11\eta_{k_{-1}}^T\dots\eta_{k_{-i}}^T1^{t}22)|^d 
\end{multline}
In the case where $R_i$ is only divided at one side, for instance in the right, from \eqref{eq:d-measurecover} one has 
\begin{multline}\label{eq:inductive_distorted_sum2}
    |I(\eta_{k_1}^T\dots\eta_{k_i}^T)|^d\cdot |I(11\eta_{k_{-1}}^T\dots\eta_{k_{-i}}^T)|^d\left(1-\frac{2(\log n)^{3/2}}{n}\right)  \\
    \leq \sum_{s= n+1}^{n+\lfloor n/\sqrt{\log n}\rfloor}|I(\eta_{k_1}^T\dots\eta_{k_i}^T1^{s}22)|^d\cdot |I(11\eta_{k_{-1}}^T\dots\eta_{k_{-i}}^T)|^d 
\end{multline}

We perform this reduction over all rectangles until we arrive to the rectangle we fixed at the beginning $I(1^{s_1}221^{s_2}22)\times I(1^{s_{-1}+2}221^{s_{-2}}22)$. At the end, using both \eqref{eq:inductive_distorted_sum1} and \eqref{eq:inductive_distorted_sum2} we arrive to the expression
\begin{multline*}
    |I(1^{s_1}221^{s_2}22)|^d\cdot |I(1^{s_{-1}+2}221^{s_{-2}}22)|^d\left(1-\frac{8(\log n)^{3/2}(\log\log n)^4}{n}\right)^{2(\log r)^2(1+o(1))} \\
    \leq \sum_{\text{distorted in all steps}}|I(1^{2k-1}\eta_{k_2}\dots\eta_{k_N}22)|^d\cdot |I(11\eta_{k_{-1}}^T\dots\eta_{k_{-M}}^T)|^d
\end{multline*}
where we used that \eqref{eq:sizer_of_notsemisymmetric_w} gives the estimate $M+N = 2(\log r)^2(1+o(1))$ which is an upper bound for the number of subdivision of any rectangle $R_L$. 

Observe that \eqref{eq:sizer_of_notsemisymmetric_w} and \eqref{eq:intervals_length} imply that 
\begin{multline*}
    |I(\eta_{k_1}^T\eta_{k_2}^T\dots\eta_{k_N}^T)|^{d}\cdot|I(11\eta_{k_{-1}}^T\dots\eta_{k_{-M}}^T)|^{d}\\
    <(2e^{-\sizer(w)})^d<e^{-2dr(\log r)^2(1+O((\log r)^{-2}))}=e^{-2(\log r)^2(\log r-\log\log r + c_0+o(1))}.
\end{multline*}
If we denote by $\mathcal{N}$ the number of words $w$ of the form \eqref{eq:w} satisfying \eqref{eq:conditions_on_w} and such that each step of its construction is distorted, then from the previous inequality we can conclude that
\begin{equation*}
    |I(1^{s_1}221^{s_2}22)|^d\cdot |I(1^{s_{-1}+2}221^{s_{-2}}22)|^d\cdot e^{O((\log r)^3/r)} \leq \mathcal{N}\cdot e^{-2(\log r)^2(\log r-\log\log r + c_0+o(1))},
\end{equation*}
whence using $|I(1^{s_1}221^{s_2}22)|\cdot |I(1^{s_{-1}+2}221^{s_{-2}}22)| \geq e^{-2r-2r/\sqrt{\log r}}$ and the expression for $d$ yields
\begin{equation}\label{eq:counting_distorted_w}
    \mathcal{N}\geq e^{2(\log r)^2(\log r-\log\log r+c_0+o(1))}.
\end{equation}

On the other hand, from \Cref{lem:few_long_combinatorics}, we know that for any word $\tilde{w}\in\Sigma^{(r-4)}(3+e^{-r})$ such that it does not contain factors of the form $1^s$ with $\sizer(1^s)\geq r-64(\log r)^2$, it holds that any extension $\tilde{w}_l\tilde{w}\tilde{w}_r\in\Sigma(3+e^{-r},|\tilde{w}_l\tilde{w}\tilde{w}_r|)$ with $\tilde{w}_l,\tilde{w}_r\in\Sigma^{(r\lfloor(\log r)^2\rfloor-2)}(3+e^{-r})$ we have at most 
\begin{equation*}
    \left(O(r^3)e^{(\log r)^2(\log r-\log\log r-\log(5/4))}\right)^2 
\end{equation*}
choices for $(\tilde{w}_l,\tilde{w},\tilde{w}_r)$. Since all rectangles associated with $w^*$ of the form \eqref{eq:w} (and distorted in all steps) have disjoint projections between themselves, any rectangle $\tilde{w}_l\tilde{w}^*\tilde{w}_r$ intersects at most two such projections, so all rectangles $\tilde{w}_l\tilde{w}^*\tilde{w}_r$ intersect at most
\begin{equation*}
    e^{2(\log r)^2(\log r-\log\log r-\log(5/4)+o(1))}
\end{equation*}
such rectangles $w^*$. Finally, since this number is much less that \eqref{eq:counting_distorted_w}, for the majority of $w^*$ of the form \eqref{eq:w}, we have disjoint projections from all other combinatorics.

\subsection{Extensions of non semisymmetric words}

From now on we will assume that $w^*$ has the form \eqref{eq:w}, satisfies \eqref{eq:conditions_on_w}, its distorted at all steps and that moreover has disjoint projections from all other combinatorics.

Observe that any finite extension of $w$ with Markov value close to $m(\overline{w}) < 3+e^{-r}$ is forced to be a longer word with the same structure (because of \Cref{lem:calc_s}). Hence the extension remains in the big Cantor set we are studying, namely, 
\begin{equation*}
    \dots 221^{s_{-M-2}}221^{s_{-M-1}}w^*221^{s_{N+1}}221^{s_{N+2}}\dots
\end{equation*}
where all the exponents are at least $n-\lfloor128(\log n)^2\rfloor$.

In particular, we would like to choose the exponents to achieve local uniqueness and self-replication, namely $s_{N+1}=s_{-M}, s_{N+2}=s_{-M+1}, \dots$ on the right and similarly $s_{-M-1}=s_N, s_{-M-2}=s_{N-1},\dots$ to the left.

Recall that
\begin{equation*}
    w^*=\eta_1\dots\eta_h^*\dots\eta_\ell
\end{equation*}
where each $\eta_i$ is of the form $221^{t_i}$ and $\eta_h^*=22^*1^{2k-1}$. By hypothesis, there are positive integers $1= k_{-L}\leq\dots\leq k_{-1}=h-1<k_1=h\leq \dots \leq k_L=\ell$ such that $\eta_{k_{-i}}\dots\eta_h^*\dots\eta_{k_i}$ is a rectangle in the construction. In particular, all these rectangles are distorted and either typical or exceptional. In other terms, we have that for each step in the construction the ratio
\begin{equation*}
    \mu_i=\frac{|I(11\eta_{h-1}^T\dots\eta_{k_{-i}}^T)|}{|I(1^{2k-1}\eta_{h+1}\dots\eta_{k_i}22)|}
\end{equation*}
satisfies $\varphi^{-n-5n/\sqrt{\log n}}<\mu_i<\varphi^{n+5n/\sqrt{\log n}}$ and
\begin{equation*}
    (\log\log n)^4<\left|\log\left|\frac{\varphi^{-2k_i}}{\mu_{i-1}\varphi^{-2k_{-i}}}\right|\right|.
\end{equation*}

\subsection{Local uniqueness}

Now we want to glue copies of $\eta_j$, following the same order as in the typical or exceptional rectangles division done with $w$, so we want
\begin{equation*}
    \eta_{\ell}w^*\eta_122, \qquad  \eta_{\ell-1}\eta_{\ell}w^*\eta_1\eta_222, \qquad \eta_{\ell-2}\eta_{\ell-1}\eta_{\ell}w^*\eta_1\eta_2\eta_322, \qquad\dots
\end{equation*}

The following lemma demonstrates that extending the word $w$ by repeating it on both sides, we will be still in the good framework required in \Cref{lem:exponential_equation}, that is, all rectangles $\eta_{k_i+1}\dots\eta_\ell w^*\eta_1\dots\eta_{k_{-i}-1}22$ of the construction of $K_{\mathrm{small}}\times K_{\mathrm{small}}$ still have the good properties from the previous section.

\begin{lemma}
The rectangles $\eta_{k_i+1}\dots\eta_\ell w^*\eta_1\dots\eta_{k_{-i}-1}22$ satisfy \eqref{eq:distorted_rectangle_condition} (they are almost distorted). Moreover, they all are typical or exceptional. 
\end{lemma}

\begin{proof}
This follows from the fact that the new quotients $\mu_{L+i}$, up to a small distortion factor given by \eqref{eq:denominator_asymptotics}, already appeared before:
\begin{align*}
    \mu_{L+i}&=\frac{|I(11\eta_{h-1}^T\dots\eta_{1}^T\eta_\ell^T\dots\eta_{k_{i}+1}^T)|}{|I(1^{2k-1}\eta_{h+1}\dots\eta_{\ell}\eta_1\dots\eta_{k_{-i}-1}22)|} \\
    &=(1+O(\varphi^{-n+5\log\log n}))\frac{|I(11\eta_{h-1}^T\dots\eta_{k_{-i}}^T)|}{|I(1^{2k-1}\eta_{h+1}\dots\eta_{k_{i}}22)|}=(1+O(\varphi^{-n+5\log\log n}))\mu_i.
\end{align*}
\end{proof}

If $\eta_{k_{-i}}\dots\eta_h^*\dots\eta_{k_i}22$ is typical, then $\eta_{k_i+1}\dots\eta_\ell w^*\eta_1\dots\eta_{k_{-i}-1}22$ will be (almost) typical and the next rectangle will be $\eta_{k_i}\dots\eta_\ell w^*\eta_1\dots\eta_{k_{-i}}22$ and if it is exceptional, then $\eta_{k_i+1}\dots\eta_\ell w^*\eta_1\dots\eta_{k_{-i}-1}22$ is also (almost) exceptional and we glue either $\eta_{k_i}$ to the left or $\eta_{k_{-i}}^T$ to the right. In particular the sequence of quotients of rectangles is $\mu_1,\dots,\mu_{L-1},\mu_L,(1+O(\varphi^{-n+5\log\log n}))\mu_{L-1},\dots (1+O(\varphi^{-n+5\log\log n}))\mu_1$.

In the very last step of this argument we will arrive to the rectangle
\[\eta_{h+1}\dots\eta_\ell w^*\eta_1\dots\eta_{h-2} 22=221^{s_2}\dots 221^{s_N}~w^*~221^{s_{-M}}\dots1^{s_{-2}}22.\]
This rectangle is typical and distorted. However, to apply \Cref{lem:exponential_equation} in the next step, one would need the next rectangle to be distorted as well. That is far from true, because one actually has
\begin{equation*}
    \frac{|I(11\eta_{h-1}^T\dots\eta_{1}^T\eta_\ell^T\dots\eta_{h+1}^T1^{2k-1}22)|}{|I(1^{2k-1}\eta_{h+1}\dots\eta_{\ell}\eta_1\dots\eta_{h-2}\eta_{h-1}22)|}=(1+O(\varphi^{-n+5\log\log n}))\varphi^{-4},
\end{equation*}
where we used \eqref{eq:denominator_asymptotics}. This situation is not so bad, since one can solve \eqref{eq:exponential_equation} directly without using linear forms in logarithms. Indeed, note that the previous quotient is also very close to a power of $\varphi^{-2}$, given by
\begin{equation*}
    \mu=\frac{|I(11\eta_{h-1}^T\dots\eta_{1}^T\eta_\ell^T\dots\eta_{h+1}^T)|}{|I(1^{2k-1}\eta_{h+1}\dots\eta_{\ell}\eta_1\dots\eta_{h-2}22)|}=(1+O(\varphi^{-n+5\log\log n}))\varphi^{-2l}.
\end{equation*}
where $l\geq 0$. Furthermore we know that the quantity of 1's in this last step is $l\equiv 2|w|-|\eta_h|-|\eta_{h-1}|\equiv 1+s_{-1}\equiv 0 \pmod{2}$. Hence in this last step \eqref{eq:exponential_equation} becomes
\begin{equation}\label{eq:exp_eq_spc_case}
    (-\varphi^2)^{-s}-(-\varphi^2)^{-u}-\left(1+O\left(\varphi^{-n+5\log\log n}\right)\right)\left((-\varphi^2)^{-s-2}-(-\varphi^2)^{-v-l}\right) =O\left(\varphi^{-2n+1000(\log n)^2}\right).
\end{equation}

\begin{lemma}\label{lem:exp_eq_spc_case}
Let $n\in\NN$ be large. Suppose $n\leq s\leq  n+\lfloor n/\sqrt{\log n}\rfloor$ and $n-\lfloor 128(\log n)^2\rfloor\leq u,v$, $l\geq 0$ are positive integers such that \eqref{eq:exp_eq_spc_case} holds. Then $s=u$ and $s+2=v+l$.
\end{lemma}

\begin{proof}
We begin by dividing \eqref{eq:exp_eq_spc_case} by $(-\varphi^2)^{-s}$ both sides to obtain
\begin{equation*}
    1-\varphi^{-4}+O(\varphi^{-n+5\log\log n})=(-\varphi^2)^{y}-(1+O(\varphi^{-n+5\log\log n}))(-\varphi^2)^{x}
\end{equation*}
where $y=s-u$ and $x=s-v-l$. Since $x\leq s-v-l\leq n/\sqrt{\log n}+128(\log n)^2\leq 2n/\sqrt{\log n}$ for large $n$, we can rewrite this equation as
\begin{equation*}
    1-\varphi^{-4}=(-\varphi^2)^{y}-(-\varphi^2)^{x}+O(\varphi^{-n+5n/\sqrt{\log n}}).
\end{equation*}
Since the left hand side is positive, we can rewrite this equation further as
\begin{equation*}
    1-\varphi^{-4}=\varphi^{2\max\{x,y\}}|1\pm(-\varphi^2)^{-|x-y|}|+O(\varphi^{-n+3n/\sqrt{\log n}}).
\end{equation*}
Clearly $|1\pm\varphi^{-2|x-y|}|\neq 0$, so we must have $|1\pm\varphi^{-2|x-y|}|\geq 1 - \varphi^{-2}=\varphi^{-1}$. In case that $\max\{x,y\}> 0$, we would obtain that $\varphi^{2\max\{x,y\}}|1\pm\varphi^{-2|x-y|}|\geq\varphi>1-\varphi^{-4}$, which is contradictory for large $r$, while if $\max\{x,y\}<0$ we would obtain $\varphi^{2\max\{x,y\}}|1\pm(-\varphi^2)^{-|x-y|}|\leq2\varphi^{-2}<1-\varphi^{-4}$, again a contradiction.

Hence $\max\{x,y\}=0$. If $x=0$, then using that $2-\varphi^{-4}=3\varphi^{-1}$ and the fact that $3$ can not be a power of $\varphi$, we reach a contradiction. Therefore, we conclude that $y=0$ and $x=2$. Equivalently, $s=u$ and $s+2=v+l$.

\end{proof}

The conclusion from our discussion so far is that we have local uniqueness for typical words in the construction we have done.

\begin{theorem}[Local uniqueness for most finite words]\label{thm:local_uniqueness}
For most $w$ of the form \eqref{eq:w}, there is an $\varepsilon_1=\varepsilon_1(w)>0$ such that for any bi-infinite sequence $\underline{\theta}\in\{1,2\}^{\ZZ}$ with $m(\underline{\theta})=\lambda_0(\underline{\theta})$, it holds that if $|m(\underline{\theta})-m(\overline{w})|<\varepsilon_1$ then
\begin{align*}
    \underline{\theta}&=\dots \eta_h\eta_{h+1}\dots \eta_\ell\eta_1 \dots \eta_{h-1}\eta_{h}^*\eta_{h+2}\dots\eta_{\ell}\eta_1\dots\eta_{h-1}22\dots \\
    &= \dots\eta_h\eta_{h+1}\dots\eta_\ell w^*\eta_1\dots\eta_{h-1}22\dots 
\end{align*}
\end{theorem}

\begin{remark}
We have the estimate 
\begin{equation*}
    \varepsilon_1(w):=\sizes(\eta_{h-1}^T\dots\eta_h^T\eta_\ell^T\dots\eta_h^T)+\sizes(\eta_h\dots\eta_\ell\eta_1\dots\eta_{h-1}22)\approx\sizes(w)\approx e^{-2r(\log r)^2+2r}.
\end{equation*}
The choice for $\varepsilon_1(w)$ is related to the size of the projection of the respective rectangle that contains all the continuations of $\eta_h\eta_{h+1}\dots\eta_\ell w^*\eta_1\dots\eta_{h-1}22$.
\end{remark}

It is worth mentioning that we stop the local uniqueness at the step given in \Cref{thm:local_uniqueness} because that is precisely the moment when the critical position $22^*1^{2k-1}22$ appears again, in both left and right sides of the initial block $w$. Note that $w^*$ also contains a critical position of that form. The geometrical interpretation of the local uniqueness for most words of the form \eqref{eq:w} is depicted in Figures \ref{fig:Markovpartition}, \ref{fig:not_local_uniqueness} and \ref{fig:local_uniqueness} below.

\tikzset{
    myline/.style={
        decoration={
            markings,
            mark=at position 0.3 with {\node[left] {$[$};},
            mark=at position 0.7 with {\node[right] {$]$};}
        },
        postaction={decorate}
    }
}

\begin{figure}[H]
\centering
\begin{tikzpicture}
\node foreach \x in {1,...,4} foreach \y in {1,...,4}[minimum size=5mm,draw] at (\x,\y) {};
\draw (3,7) -- (7,3);
\end{tikzpicture}
\caption{Markov partition of $K\times K$ projecting onto  $K-K$.}
\label{fig:Markovpartition}
\end{figure}
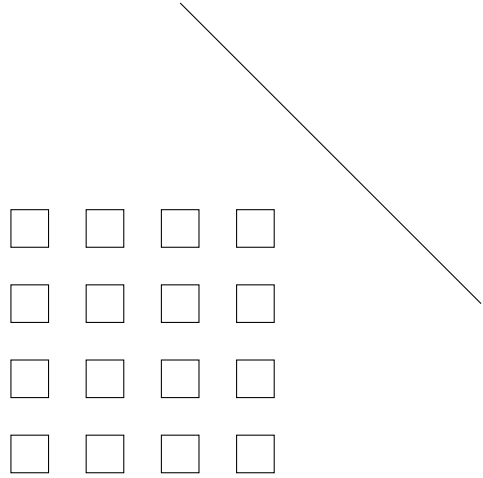

\begin{figure}[htbp]
    \centering
    \begin{minipage}{0.45\textwidth}
        \centering
        \begin{tikzpicture}
        \node foreach \x in {1,...,4} foreach \y in {1,...,4}[minimum size=8mm,draw,diamond] at (\x,\y) {};
        \foreach \x in {1,...,4} {
                \draw[myline] (\x-0.3,0) -- (\x+0.3,0);
            }
        \end{tikzpicture}
        \caption{No local uniqueness}
                \label{fig:not_local_uniqueness}
    \end{minipage}\hfill
    \begin{minipage}{0.45\textwidth}
        \centering
        \begin{tikzpicture}
        \node [minimum size=8mm,draw,diamond] at (1,3) {};
        \node [minimum size=8mm,draw,diamond] at (2,2) {};
        \node [minimum size=8mm,draw,diamond] at (3,3) {};
        \node [minimum size=8mm,draw,diamond] at (4,1) {};
        \foreach \x in {1,...,4} {
                \draw[myline] (\x-0.3,0) -- (\x+0.3,0);
            }
        \end{tikzpicture}
        \caption{Local uniqueness}
        \label{fig:local_uniqueness}
    \end{minipage}
\end{figure}

\subsection{Self-replication and end of proof of Theorem \ref{thm:dimension}}

From the previous subsection we obtained that the only rectangle that projects near $m(\overline{w})$ is the one given by a central sub-block of $\dots ww^*w\dots$, namely
\begin{equation}\label{eq:452}
    221^{2k-1}221^{s_2}\dots 221^{s_{N}}221^{s_{-M}}\dots 1^{s_{-1}}22^*1^{2k-1}221^{s_2}\dots 221^{s_{N}}221^{s_{-M}}\dots 1^{s_{-1}}22.
\end{equation}

Let us introduce more notation to explain how we will prove self-replication. Let 
\begin{equation*}
    w^*=w_L22^*w_R, \quad\text{where}\quad w_L=221^{s_{-M}}22\dots 221^{s_{-1}}, \quad w_R=1^{2k-1}22\dots 1^{s_{N}}.
\end{equation*}

Recall that in our construction the quotients are appearing in the order 
\begin{multline*}
    \mu_1,\dots,\mu_{L-1},\mu_L,(1+O(\varphi^{-(1+o(1))n}))\mu_{L-1},\dots,(1+O(\varphi^{-(1+o(1))n}))\mu_1,(1+O(\varphi^{-(1+o(1))n}))\varphi^{-4},\\ (1+O(\varphi^{-(1+o(1))n}))\mu_1,\dots,(1+O(\varphi^{-(1+o(1))n}))\mu_{L-1},(1+O(\varphi^{-(1+o(1))n}))\mu_L,(1+O(\varphi^{-(1+o(1))n}))\mu_{L-1},\dots
\end{multline*}

Since we can apply \Cref{lem:exponential_equation} and \Cref{lem:exp_eq_spc_case}, this ensures that rectangles of the form
\begin{equation}\label{eq:correct_extension}
    221^{s_{l-1}}\dots 1^{s_{-1}}22w_Rw^*w_L221^{2k-1}221^{s_{2}}\dots 1^{s_{r+1}}22,
\end{equation}
have disjoint projections from rectangles of the form 
\begin{equation}\label{eq:bad_extension}
    221^v221^{s_l}\dots 1^{s_{-1}}22w_Rw^*w_L221^{2k-1}221^{s_{2}}\dots 1^{s_r}221^u22.
\end{equation}
However for local uniqueness, we need the projection of the second rectangle to lie above of the rectangle that contains the periodic orbit $\overline{w}$. In other words, when we refine the rectangle $221^{s_l}\dots 1^{s_{-1}}22w_Rw^*w_L221^{2k-1}221^{s_{2}}\dots 1^{s_r}22$ all other refined rectangles project above the rectangle that contains $\overline{w}$ or they contain a subword with larger value than $\varepsilon_1$ (from local uniqueness). We claim that indeed this is the case and prove it in \Cref{thm:self-replication} below.

\begin{remark} In the process of gluing the $221^{s_i}$ as before, at the last step we will obtain
\begin{equation*}
    221^{s_2}\dots 221^{s_N}ww^*w221^{s_{-M}}\dots 1^{s_{-2}}22.
\end{equation*}
At this very last step we can either glue $221^{2k-1}$ to the left and $1^{s_{-1}}22$ to the right. As we shall see in the proof of Theorem \ref{thm:self-replication}, it is more expensive to glue to the left than to the right. Furthermore, the proof of self-replication shall not use linear forms in logarithms and does not rely on probabilistic considerations. This is why it holds for all non semisymmetric words of the form \eqref{eq:w} verifying \eqref{eq:conditions_on_w}.
\end{remark}

\begin{theorem}[Self-replication]\label{thm:self-replication}
For any $w$ of the form \eqref{eq:w} i.e. with $s_1=2k-1$, $s_{-1}>s_2$ odd and $s_{-1}= 2k+\lfloor(\log\log (2k))^4\rfloor$, there is an $\varepsilon_2=\varepsilon_2(w)>0$ such that for any bi-infinite sequence $\underline{\theta}\in\{1,\dots,A\}^{\ZZ}$ with $m(\underline{\theta})=\lambda_0(\underline{\theta})$, it holds that if $m(\underline{\theta})<m(\overline{w})+\varepsilon_2$ and $\underline{\theta}=\dots  22w_Rw^*w_L22\dots$, then $\underline{\theta}$ must be of the following form
\begin{align*}
    \underline{\theta}&=\dots 22w_R~ww^*w~221^{s_{-M}}\dots 1^{s_{-2}}221^{2k}\dots
\end{align*} 
\end{theorem}

\begin{remark}
We have the estimate $\varepsilon_2(w)\approx e^{-4r(\log r)^2+3r}/12$.
\end{remark}

When combining both local uniqueness (\Cref{thm:local_uniqueness}) and self-replication (\Cref{thm:self-replication}) we get the following characterization for sequences that have near Markov value of the periodic orbit. This characterization holds for most words \eqref{eq:w}.

\begin{corollary}\label{cor:self-replication}
For most $w$ from \eqref{eq:w}, there is an $\varepsilon=\varepsilon(w):=\min\{\varepsilon_1(w),\varepsilon_2(w)\}>0$ such that for any bi-infinite sequence $\underline{\theta}\in\{1,\dots,A\}^{\ZZ}$ with $|m(\underline{\theta})-m(\overline{w})|<\varepsilon$ and $m(\underline{\theta})=\lambda_0(\underline{\theta})$, it holds that
\begin{equation*}
    \underline{\theta}=\overline{w}w^*w~221^{s_{-M}}\dots 1^{s_{-2}}221^{2k}\dots
\end{equation*}
In particular $L\cap(m(\overline{w})-\varepsilon,m(\overline{w})+\varepsilon)=\{m(\overline{w})\}$.
\end{corollary}

At this point, it will not be hard to see below that Theorem \ref{thm:dimension} is an immediate consequence of the next statement (and our estimate for the Hausdorff dimension of $K_{\mathrm{mod}}$).

\begin{theorem}\label{thm:difference_cantor_set}
For most $w$ from \eqref{eq:w}, there is a Cantor set 
\begin{equation*}
    C=\{\lambda_0(\overline{w}w^*w~221^{s_{-M}}\dots 1^{s_{-2}}221^{s_{-1}-1}22~\gamma_1\gamma_2\dots):\gamma_i\in K_{\mathrm{mod}}, \forall i\geq 1\}
\end{equation*}
that satisfies $C\subset(M\setminus L)\cap(m(\overline{w}),m(\overline{w})+\varepsilon(w))$. In particular for any such $w$
\begin{equation*}
    \dim_H\left((M\setminus L)\cap(m(\overline{w}),m(\overline{w})+\varepsilon(w))\right)\geq\dim_H(K_{\mathrm{mod}}).
\end{equation*}
\end{theorem}

\begin{remark}\label{rem:estimate_epsilon_w}
We have the estimate $\varepsilon(w)\approx e^{-4r(\log r)^2+3r}/12$.
\end{remark}

Our arguments below will also yield the following statement:

\begin{theorem}\label{thm:isolated}
For most $w$ from \eqref{eq:w}, there is an $0<\varepsilon_3=\varepsilon_3(w)<\varepsilon_2(w)$, such that for any $\underline{\theta}\in\{1,...,A\}^\ZZ$ with $m(\underline{\theta})=\lambda_0(\underline{\theta})$, it holds that if $m(\underline{\theta})<m(\overline{w})+\varepsilon_3$ and $\underline{\theta}=\dots  22w_Rw^*w_L22\dots$ then $\underline{\theta}$ must have the form $\underline{\theta}=\dots  22w_Rww^*ww_L22\dots$. 
\end{theorem}

In particular, assuming the hypothesis of the previous theorem, one has the rigidity $\underline{\theta}=\overline{w}$ which implies that $m(\overline{w})$ is isolated in the Markov spectrum. 

\begin{corollary}
For most $w$ from \eqref{eq:w}, we have that $M\cap (m(\overline{w})-\varepsilon_3,m(\overline{w})+\varepsilon_3)=\{m(\overline{w})\}$.
\end{corollary}

\begin{remark}
We have the estimate $\varepsilon_3(w)\approx e^{-4r(\log r)^2+2r}/24$.
\end{remark}

Let us show why the main theorem follows from these results.

\begin{proof}[Proof of \Cref{thm:dimension} assuming \Cref{thm:difference_cantor_set}] Fix $r_0\in\NN$ sufficiently large. Given any $0<\varepsilon<e^{-r_0}$, take $r=\lceil \log(2/\varepsilon) \rceil$ and $n=2k-1$ as the least odd integer such that $\sizer(1^n)\geq r$. In particular $\sizes(1^n)<e^{-r}$. From \Cref{thm:difference_cantor_set} we have that there is a word $w$ of the form \eqref{eq:w} and a Gauss-Cantor set $C$ such that $C\subset(M\setminus L)\cap(m(\overline{w}),m(\overline{w})+\varepsilon(w))$. Using \Cref{lem:markov_value_w} we know that $m(\overline{w}) < 3+\sizes(1^{n+1}) < 3 + e^{-r}$. In particular from \Cref{rem:estimate_epsilon_w} we will have that $m(\overline{w})+\varepsilon(w) < 3+2e^{-r} <  3 + \varepsilon$. Finally since $C$ is bi-Lipschitz to $K_{\mathrm{mod}}$ and $n\in((r-1)e^{c_0},(r+2)e^{c_0})$, we conclude that 
\begin{align*}
    \dim_H\left((M\setminus L)\cap(3,3+\varepsilon)\right)&\geq \dim_H\left((M\setminus L)\cap(m(\overline{w}),m(\overline{w})+\varepsilon(w))\right) \\
    &\geq\dim_H(C)=
    \dim_H(K_{\mathrm{mod}})\\
    &=\frac{W(n+1)}{(n+1)e^{-c_0}}-O\left(\frac{\log n}{n^2}\right) \\
    &=\frac{W(e^{c_0}|\log \varepsilon|)}{|\log \varepsilon|}-O\left(\frac{\log |\log \varepsilon|}{|\log \varepsilon|^2}\right)
\end{align*}
where we used \eqref{eq:main_terms_difference_error} in the last step.
\end{proof}

Now we will complete this paper by proving all the previous theorems.

\begin{proof}[Proof of \Cref{thm:self-replication}]

Observe that the finite word \eqref{eq:452} has 3 critical positions, so we write \eqref{eq:452} as
\begin{equation*}
    22^{**}w_R~w_L22^{**}w_R~w_L22^{**},
\end{equation*}
where the double asterisks represents the dangerous positions. They are dangerous, since they give larger Markov value than any other position of the finite extensions we will consider. 

One important remark here is that since the length of $w$ is odd, then one increases the Markov value at the left most (or the right most) critical position if and only if the value at the middle decreases. More precisely, since the length between them is odd (because it coincides with the period $|2w_Rw_L2|=|w|$), one has that, for any finite extension $\tau\in\{1,2\}^{|\tau|}$ it holds:
\begin{center}
    The cut $\lambda_0(22w_R~w_L2|2w_R~w_L22\tau)$ increases if and only if $\lambda_0(22w_R~w_L22w_R~w_L2|2\tau)$ decreases.
\end{center}
\begin{center}
    The cut $\lambda_0(\tau22w_R~w_L2|2w_R~w_L22)$ increases if and only if $\lambda_0(\tau2|2w_R~w_L22w_R~w_L22)$ decreases.
\end{center}

Let us introduce more notation. Let $\tau_l=221^{s_{l}}\dots 221^{s_{-1}}$ and $\tau_r=1^{2k-1}221^{s_1}22\dots 1^{s_{r}}22$. Hence the extension \eqref{eq:correct_extension} is given by 
\begin{equation}\label{eq:correct_extension2}
    221^{s_{l-1}}\tau_l22^{**}w_R~w_L22^{**}w_R~w_L22^{**}\tau_r1^{s_{r+1}}22,
\end{equation}
and the extension \eqref{eq:bad_extension} with $(u,v)\neq(s_{r+1},s_{l-1})$ is 
\begin{equation}\label{eq:bad_extension2}
    221^{v}\tau_l22^{**}w_R~w_L22^{**}w_R~w_L22^{**}\tau_r1^{u}22.
\end{equation}

We want to show that the extension \eqref{eq:bad_extension2} leads to larger Markov values.

If both dangerous positions to the left and right are decreasing, then the middle value is increasing with respect to the real continuation and the jump can be estimated using \eqref{eq:lower_bound_interval_gap} by
\begin{multline*}
    \lambda_0(221^{v}\tau_l22w_R~w_L2|2w_R~w_L22\tau_r1^{u}22)-\lambda_0(221^{s_{l-1}}\tau_l22w_R~w_L2|2w_R~w_L22\tau_r1^{s_{r+1}}22) \\
    >\sizes(w_Rw_L22\tau_r1^{s_{r+1}+1})/12+\sizes(1^{s_{l-1}+1}\tau_l22w_R~w_L2)/12
\end{multline*}
Now suppose at least one of the lateral dangerous positions are larger than the central one. Without loss of generality we can assume is the right one. In other words, from \eqref{eq:lower_bound_interval_gap} we have the inequalities
\begin{equation*}
    [0;\tau_r,1^u,2,2,\dots] > [0;\tau_r,1^{s_{r+1}},2,2,\dots]+\sizes(\tau_r1^{s_{r+1}+1})/12 
\end{equation*}
and similarly using \eqref{eq:transpose_interval_sizes}
\begin{align*}
    &|[0;2,w_L^T,w^T,w_R^T,2,2,\tau_l^T,1^v,2,2,\dots]-[0;2,w_L^T,w^T,w_R^T,2,2,\tau_l^T,1^{s_{l-1}},2,2,\dots]|\\ 
    &<\sizes(2w_L^Tw^Tw_R^T22\tau_l^T1^{\min\{v,s_{l-1}\}+1}) \leq \sizes(1^{\min\{v,s_{l-1}\}+1}\tau_l22w_Rww_L2).
\end{align*}
Since $\tau_r1^{s_{r+1}}$ is a prefix of $w_Rw_L$ and $w=w_L22w_R$, we have for large $r$ that
\begin{align*}
    \sizes(1^{\min\{v,s_{l-1}\}+1}\tau_l22w_Rww_L2)&<4\sizes(1^{s_{l-1}}\tau_l22)\sizes(\tau_r1^{s_{r+1}})\sizes(22w_Rw_L) \\
    &<\sizes(\tau_r1^{s_{r+1}+1})\cdot e^{-r(\log r)^2}, 
\end{align*}
whence
\begin{align*}
&\quad\lambda_0(221^v\tau_l22~w_Rww_L~2|2\tau_r1^u22) \\
    &= [2;\tau_r,1^u,2,2,\dots]+[0;2,w_L^T,w^T,w_R^T,2,2,\tau_l^T,1^v,2,2,\dots] \\
    &>[2;\tau_r,1^{s_{r+1}},2,2,S]+[0;2,w_L^T,w^T,w_R^T,2,2,\tau_l^T,1^{s_{l-1}},2,2,R^T]\\
    &\qquad+\sizes(\tau_r1^{s_{r+1}+1})/12-\sizes(1^{\min\{v,s_{l-1}\}+1}\tau_l22w_Rww_L2) \\
    &>\lambda_0(R~221^{s_{l-1}}\tau_l22w_R~w~w_L2|2\tau_r1^{s_{r+1}}22~S) + (1/12-e^{-r(\log r)^2})\sizes(\tau_r1^{s_{r+1}+1}) \\
    &>\lambda_0(R~221^{s_{l-1}}\tau_l22w_R~w~w_L2|2\tau_r1^{s_{r+1}}22~S)+e^{-2r(\log r)^2}/24
\end{align*}
for any choice of $R=R_1R_2\dots$ and $S=S_1S_2\dots$ with $R_i,S_i\in\{1,2\}$ for each $i$ (we used that $\sizes(\tau_r1^{s_{r+1}+1})>\sizes(w)>e^{-2r(\log r)^2}$). Since we can find $R$ and $S$ such that  $R~221^{s_{l-1}}\tau_l22w_R~w~w_L2|2\tau_r1^{s_{r+1}}22~S=\dots w~w_L2|2w_R~w \dots=\overline{w}$, we have that 
\begin{equation*}
    \lambda_0(221^v\tau_l22~w_Rww_L~2|2\tau_r1^u22)>m(\overline{w})+e^{-2r(\log r)^2}/24.
\end{equation*}

Clearly the jump $e^{-2r(\log r)^2}/24$ in the value of $\lambda_0$ at this lateral dangerous positions is very large. 

At the last step of this process we will obtain
\begin{equation}\label{eq:last_step}
    221^{s_2}\dots 221^{s_{N}}~ww^*w~221^{s_{-M}}\dots 1^{s_{-2}}22.
\end{equation}

By induction we have any other possible continuation from $22w_R~w_L22w_R~w_L22 = 22w_R~w~w_L22$ has larger value than the maximum value of this continuation. 

This gives an estimate for $\varepsilon_2(w)$
\begin{equation*}
    \varepsilon_2(w):=\sizes(1^{s_2+1}\dots 221^{s_{N}}~ww_L2)/12. 
\end{equation*}

Indeed, note that to replicate to the right hand side we actually need $\varepsilon_2(w)=\min\{\sizes(1^{s_2+1}\dots 221^{s_{N}}~ww_L2)/12,\sizes(w_Rw~221^{s_{-M}}\dots 1^{s_{-2}+1})/12\}$, which gives the above estimate because using \eqref{eq:intervals_length}
\begin{align*}
    \sizes(1^{s_2+1}\dots 221^{s_{N}}~ww_L2)&<2\sizes(1^{s_2}\dots 221^{s_{N}}~w ~221^{s_{-M}}\dots 1^{s_{-2}+1})\sizes(1^{s_{-1}})\\
    &<4\sizes(w_Rw ~221^{s_{-M}}\dots 1^{s_{-2}+1})\sizes(1^{s_{-1}})\sizes(1^{2k-1}22)^{-1} \\
    &<\sizes(w_Rw~221^{s_{-M}}\dots 1^{s_{-2}+1})\cdot (32e^{-(\log\log r)^4})
\end{align*}
where we used that the very first rectangle was distorted (since $s_{-1}-2k > (\log\log n)^4\log\varphi$). 

As a remark, observe that using \eqref{eq:sizer_of_notsemisymmetric_w} we have the estimate 
$\varepsilon_2(w)\approx \sizes(ww)\sizes(1^{2k-1})^{-1}/12 \approx e^{-4r(\log r)^2+3r}/12$.

Note that in the continuation \eqref{eq:last_step}, the length from the right dangerous position to the end is $|w_R~221^{s_{-{M}}}\dots 1^{s_{-2}}22|=|w|-(s_{-1}+2)$ which is even, while the length from the left dangerous position to the beginning is $|221^{s_1}\dots 221^{s_{N}}w_L2|=|w|-2k$ which is odd. This asymmetry between the left and right dangerous positions is crucial.

The next continuations we have to consider have the form
\begin{equation*}
    221^{t}221^{s_2}\dots 221^{s_{N}}~ww^*w~221^{s_{-M}}\dots 1^{s_{-2}}221^{s}22.
\end{equation*}

Observe that we are forced to have $t=2k-1$: indeed, first notice that since the length of $|221^{s_1}\dots 221^{s_{N}}w_L2|=|w|-2k$ is odd, then $t$ even or $t\geq 2k$ would make the left dangerous position to jump at least $e^{-r(\log r)^2}/24$, so $t\leq 2k-1$ is odd. Finally, if $t$ is odd we must have $t\geq 2k-1$.

Therefore, we are forced to continue as 
\begin{equation*}
    221^{2k-1}221^{s_2}\dots 221^{s_{N}}~ww^*w~221^{s_{-M}}\dots 1^{s_{-2}}22.
\end{equation*}

Note that the length of $|w_R~221^{s_{-M}}\dots 1^{s_{-2}}22|=|w|-(s_{-1}+2)$ is even, we can continue with any even block of the form $1^s22$ without increasing the value at the right dangerous position (in fact it will have a lower value than the central position). Clearly a small even block of 1's could create a bad cut in the position $221^{s_{-2}}|221^s22$, hence we must require at least $s\geq 2k$. In other words, we are forced to continue as 
\begin{equation}\label{eq:final_forced_continuation}
    221^{2k-1}221^{s_2}\dots 221^{s_{N}}~ww^*w~221^{s_{-M}}\dots 1^{s_{-2}}221^{2k} = 22w_R~ww^*w~221^{s_{-M}}\dots 1^{s_{-2}}221^{2k}.
\end{equation}

\end{proof}

\begin{proof}[Proof of \Cref{cor:self-replication}]
Suppose that $w$ satisfies both \Cref{thm:local_uniqueness} and \Cref{thm:self-replication}. Let $\underline{\theta}\in\{1,\dots,A\}^\ZZ$ with $|m(\underline{\theta})-m(\overline{w})|<\varepsilon$ and $m(\underline{\theta})=\lambda_0(\underline{\theta})$. By \Cref{thm:local_uniqueness} we must have $\underline{\theta}= \dots22w_Rw^*w_L22\dots$ and by \Cref{thm:self-replication} we have $\underline{\theta}= \dots22w_Rww^*w~221^{s_{-M}}\dots 1^{s_{-2}}221^{2k}\dots$. By applying inductively \Cref{thm:self-replication} to the left, we get that $\underline{\theta}= \overline{w}w^*w~221^{s_{-M}}\dots 1^{s_{-2}}221^{2k}\dots$ Finally, any periodic orbit $\overline{p}$ satisfying $|m(\overline{w})-m(\overline{p})|<\varepsilon$ will be forced to be $\overline{p}=\overline{w}w^*w\dots$, which implies $\overline{p}=\overline{w}$ and $m(\overline{p})=m(\overline{w})$. 
\end{proof}

\begin{proof}[Proof of \Cref{thm:difference_cantor_set}]
Since a continuation of the form $22w_R~ww^*w~221^{s_{-M}}\dots 1^{s_{-2}}221^{2k+2l}22$ with $l\geq 0$ is not forced to jump too much, we are able to continue to the right and connect with a moderate Cantor set. Indeed, since $w=221^{s_{-M}}\dots221^{s_{-2}}221^{s_{-1}}\dots$ these continuations have value at most
\begin{align*}
    &\quad\lambda_0(\overline{w}~w_L2|2w_R~w~221^{s_{-M}}\dots 1^{s_{-2}}221^{s_{-1}-1}22\dots ) \\
    &= [2;2,w_L^T,\overline{w}]+[0;w_R,w,2,2,1^{s_{-M}},\dots,2,2,1^{s_{-2}},2,2,1^{s_{-1}-1},2,2,\dots] \\
    &<[2;2,w_L^T,\overline{w}]+[0;w_R,\overline{w}]+\sizes(w_Rw~221^{s_{-M}}\dots 221^{s_{-2}}221^{s_{-1}-2}) \\
    &=m(\overline{w})+\sizes(w_Rw~221^{s_{-M}}\dots 221^{s_{-2}}221^{s_{-1}-2}).
\end{align*}

Using \eqref{eq:intervals_length} we have that
\[
    \sizes(1^{s_2+1}\dots 221^{s_{N}}~ww_L2)>\frac{1}{4}\sizes(w_Rw~221^{s_{-M}}\dots 221^{s_{-2}}221^{s_{-1}-2})\sizes(1^{s_{-1}-2k-3}).
\]
Moreover, we know that the very first rectangle was distorted (as $s_{-1}-2k>(\log\log (2k))^4$).

Hence we can connect this continuation with the Cantor set 
\begin{align*}
    K_{\mathrm{mod}}&=\left\{[0;1^{s_1},2,2,1^{s_2},2,2,\dots]:s_i\geq 2k \text{ for all } i\geq 1\right\} \\
    &=\left\{[0;1^{2k},\gamma_1,\gamma_2,\dots]:\gamma_i\in\{221^{2k},1\} \text{ for all } i\geq 1\right\}.
\end{align*}

In other words, we have that
\begin{equation*}
    C:=\{\lambda_0(\overline{w}~w_L22|w_R~w~221^{s_{-M}}\dots 1^{s_{-2}}221^{s_{-1}-1}22~\gamma_1\gamma_2\dots):\gamma_i\in K_{\mathrm{mod}}, \forall i\geq 1\}
\end{equation*}
satisfies $C\subset(M\setminus L)\cap(m(\overline{w}),m(\overline{w})+\varepsilon(w))$. Finally, since $C$ is bi-Lipschitz to $K_{\textrm{mod}}$, we have that $\dim_H(C)=\dim_H(K_{\textrm{mod}})$ and
\begin{equation*}
    \dim_H\left((M\setminus L)\cap(m(\overline{w}),m(\overline{w})+\varepsilon(w))\right)\geq \dim_H(K_{\textrm{mod}}),
\end{equation*}
because of the monotonicity of the Hausdorff dimension. 
\end{proof}

\begin{proof}[Proof of \Cref{thm:isolated}]
If we force the continuation \eqref{eq:final_forced_continuation} to extend a little bit more, then we will arrive to the continuation $22w_R~www~w_L22$, which contains new copies of $22w_R~w~w_L22$. Hence we will self-replicate to both sides, which shows that in this situation the Markov value $m(\overline{w})$ is isolated and we have an explicit estimate of the neighborhood
\begin{equation*}
    \varepsilon_3(w) := \min\{\sizes(w_Rww_L1)/12,\sizes(1w_Rw~w_L2)/12\} \approx \sizes(ww)/24 \approx e^{-4r(\log r)^2+2r}/24.
\end{equation*}

\end{proof}



\end{document}